\documentclass{amsart}[12]

\usepackage[a4paper]{geometry}
\usepackage{amsmath}
\usepackage{amsthm}
\usepackage{mathrsfs}
\usepackage{amsfonts}
\usepackage{amssymb}
\usepackage{stmaryrd}
\usepackage{amscd}
\usepackage{array}
\usepackage{amssymb}
\usepackage[all, cmtip]{xy}
\usepackage[dvipsnames]{xcolor}
\usepackage{tikz}
\usetikzlibrary{arrows}
\usetikzlibrary{cd}
\usepackage{enumitem}
\usepackage{blkarray}
\usepackage{hyperref}
\usepackage{upgreek}

\newtheorem{theorem}{Theorem}[subsection]
\newtheorem*{theorem*}{Theorem}
\newtheorem{lemma}[theorem]{Lemma}

\newtheorem{definition}[theorem]{Definition}
\newtheorem{corollary}[theorem]{Corollary}
\newtheorem{proposition}[theorem]{Proposition}
\newtheorem{remark}[theorem]{Remark}

\newtheorem{notation}[theorem]{Notation}
\newtheorem{example}[theorem]{Example}

\newcommand{\CC}{\mathbb{C}}
\newcommand{\QQ}{\mathbb{Q}}

\newcommand{\PP}{\mathbb{P}}
\newcommand{\ZZ}{\mathbb{Z}}
\newcommand{\TT}{\mathbb{T}}
\newcommand{\II}{\mathbb{I}}
\newcommand{\T}{\mathscr{T}}
\newcommand{\EE}{\mathbb{E}}
\newcommand{\LL}{\mathbb{L}}
\renewcommand{\AA}{\mathbb{A}}

\newcommand{\OO}{\mathscr{O}}
\newcommand{\csigma}{\mathcal{S}}

\renewcommand{\P}{\mathscr{P}}

\newcommand{\wP}[1]{\mathbb{P}^1_{a,1}}

\newcommand{\fX}{\mathfrak{X}}
\newcommand{\fY}{\mathfrak{Y}}
\newcommand{\fm}{\mathfrak{m}}
\newcommand{\cB}{\mathcal{B}}

\newcommand{\F}{\mathcal{F}}

\renewcommand{\sslash}{\mathord{/\mkern-6mu/}}

\newcommand{\omegabul}{\omega^\bullet}

\usepackage{etoolbox}
\let\tinymatrix\smallmatrix

\patchcmd{\tinymatrix}{\scriptstyle}{\scriptscriptstyle}{}{}
\patchcmd{\tinymatrix}{\scriptstyle}{\scriptscriptstyle}{}{}
\patchcmd{\tinymatrix}{\vcenter}{\vtop}{}{}
\patchcmd{\tinymatrix}{\bgroup}{\bgroup\scriptsize}{}{}
\newcommand{\bmu}{\pmb{\mu}}
\newcommand{\bi}{\mathfrak{i}}
\newcommand{\bj}{\mathfrak{j}}

\newcommand{\rpic}{r_{\mathrm{Pic}}}
\newcommand{\rchi}{r_\chi}
\newcommand{\bdtilde}{{r_T}}
\newcommand{\bd}{r_G}

\newcommand{\tcover}[1]{S_{#1}(T)}

\newcommand{\cover}[2]{S_{#2}(#1)}

\newcommand{\cX}{\mathcal{X}}

\newcommand{\cZ}{\mathcal{Z}}
\newcommand{\cC}{\mathcal{C}}
\newcommand{\bx}{\mathbf{x}}
\renewcommand{\L}{\mathscr{L}}

\newcommand{\cT}{\mathcal{T}}

\newcommand{\cF}{\mathcal{F}}

\newcommand{\fU}{\mathfrak{U}}

\newcommand{\intalpha}{\underline{\tilde \alpha}}
\newcommand{\sE}{\mathscr{E}}
\newcommand{\sF}{\mathscr{F}}

\newcommand{\torus}{\Omega}

\usepackage[backend=biber,
style=alphabetic,
isbn=false,
url=false,
sorting=nyt]{biblatex}
\usepackage{bm}
\newcommand{\aff}[2]{\Spec(\CC[#1]^{#2})}

\renewcommand{\AA}{\mathbb{A}}

\newcommand{\cY}{\mathcal Y}

\newcommand{\vir}{\mathrm{vir}}
\newcommand{\mov}{\mathrm{mov}}

\newcommand{\one}{\mathbf{1}}

\newcommand{\Res}{Res}

\DeclareMathOperator{\fix}{Fix}
\DeclareMathOperator{\Aut}{Aut}

\DeclareMathOperator{\Spec}{Spec}

\DeclareMathOperator{\Pic}{Pic}

\DeclareMathOperator{\Hom}{Hom}
\DeclareMathOperator{\Eff}{Eff}
\DeclareMathOperator{\Conj}{Conj}
\DeclareMathOperator{\spn}{span}
\DeclareMathOperator{\sgn}{sgn}

\usepackage{xcolor}

\usepackage{xcolor}

\newcommand{\NOTEoff}{\newcommand{\Commentn}[1]{}}

\newcommand\TODOon{\newcommand{\Comment}[1]{\noindent\color{red}{\texttt TODO: }##1\color{black}}}

\begin{document}

\NOTEoff
\TODOon

\title{Abelianization and quantum Lefschetz for orbifold quasimap $I$-functions}
\author{Rachel Webb}

\address[R. Webb]{Department of Mathematics\\
University of California, Berkeley\\
Berkeley, CA 94720-3840\\
U.S.A.}
\email{rwebb@berkeley.edu}

\date{\today}

\begin{abstract}
Let $Y$ be a complete intersection in an affine variety $X$, with action by a complex reductive group $G$. Let $T \subset G$ be a maximal torus. A character $\theta$ of $G$ defines GIT quotients $Y\sslash_\theta G$ and $X\sslash_\theta T$. We prove formulas relating the small quasimap I-function of $Y\sslash_\theta G$ to that of $X\sslash_\theta T$. When $X$ is a vector space, this provides a completely explicit formula for the small $I$-function of $Y\sslash_\theta G$.
\end{abstract}



\maketitle

\setcounter{tocdepth}{1}
\tableofcontents

\section{Introduction}
Let $\cX$ be a smooth proper Deligne-Mumford stack. An \textit{I-function} for $X$ is generating series with coefficients in the cohomology of the inertia stack of $X$ that is known to lie on the Lagrangian cone of $X$. This means that it is related in a precise sense to certain Gromov-Witten invariants of $\cX$. The papers \cite{orb-qmaps, CCIT15} provide an explicit formula for an $I$-function of $\cX$ when $\cX= X\sslash_\theta T$ is an orbifold GIT quotient of a vector space $V$ by a torus $T$. This formula has found numerous applications, including the crepant transformation conjecture \cite{CIJ}, Fano classification \cite{CGKS}, and mirror symmetry for orbifold del Pezzo surfaces \cite{akhtar}.

The main goal of the current paper is to extend the explicit formula of \cite{orb-qmaps, CCIT15} as far as possible. We accomplish this via the small $I$-function recipe in \cite{orb-qmaps} (proven to lie on the Lagrangian cone in \cite{yang}), which is given in terms of virtual localization residues on certain moduli spaces. Our task is to compute these residues in terms of Chern classes of line bundles on $X$. Our main result is a nonabelian quantum Lefschetz theorem (Theorem \ref{thm:ql}) without any convexity assumptions, which simultaneously strengthens the abelianization result of \cite{nonab1} to include orbifold targets and relates the $I$-function of a complete intersection to that of its ambient space. As part of this result we analyze the behavior of the inertia orbifold of a GIT quotient under abelianization (Section \ref{sec:inertia}).

In the spirit of \cite{CCIT19}, the paper \cite{SW1} with Nawaz Sultani explains how to use the formulas in this paper to compute small $I$-functions of specific targets.

\subsection{Setup and main theorems}\label{sec:setup}
Let $X$ be an l.c.i. affine variety over $\CC$ (not necessarily irreducible) and let $G$ be a connected complex reductive algebraic group acting on $X$.
Let $\chi(G)$ denote the character group and choose $\theta \in \chi(G)$. We define sets $X^s(G)$ and $X^{ss}(G)$ as in \cite[Def~2.1]{king} of stable and semi-stable points of $X$ with respect to $G$ and $\theta$ (we omit $\theta$ from the notation). We will always assume
\[
X^s(G) = X^{ss}(G).
\]

For any scheme $Y$ with a given locus of $G$-stable points, and any subgroup $H \subset G$, we define Deligne-Mumford stacks
\[X\sslash G := [X^s(G)/G] \quad \quad \quad X\sslash_G H := [X^s(G)/H].\]
with associated cyclotomic inertia stacks $I_\mu(X\sslash G)$ (resp. $I_\mu(X\sslash_G H)$ and rigidified cyclotomic inertia stacks $\overline{I}_\mu(X\sslash G)$ (resp. $\overline{I}_\mu(X\sslash_GH$). The (rigidified) cyclotomic inertia stack is defined in \cite{AGV08}. 

The \textit{small I-function} of $X\sslash G$ (see Definition \ref{def:ifunc}) is a formal series of the form
\begin{equation}\label{eq:Ifunc_shape}
I^{X\sslash G}(z) = \one_X+ \sum_{\beta\neq 0}q^\beta I^{X\sslash G}_\beta(z) \quad \quad I_\beta^{X\sslash G}(z) \in A_*(I_\mu(X\sslash G))\otimes_\QQ \QQ[z, z^{-1}].
\end{equation}
There is a closely related series $\overline{I}^{X\sslash G}(z)$ which has the same form as in \eqref{eq:Ifunc_shape} but takes values in the Chow group of $\overline{I}_\mu(X\sslash G)$. The series $\overline{I}^{X\sslash G}$ is notated $I(0, q, z)$ in \cite{orb-qmaps}, and 
the series $I^{X\sslash G}(z)$ is the pullback of $\overline{I}^{X\sslash G}(z)$ along the rigidification map---see Remark \ref{rmk:ifuncs}.

To state our formulas for ${I^{X\sslash G}}$ and $\overline{I}^{X\sslash G}$, we define some notation for line bundles on global quotients. Let $Y$ be any variety with a $G$-action. There is an inclusion $\chi(G) \rightarrow \Pic^G(Y)$ sending a character $\xi$ to 
\begin{equation}\label{eq:chi-to-pic}\L_\xi := Y \times \CC_{\xi}.\end{equation}
The bundle $\L_\xi$ descends to a line bundle on $[Y/G]$ which we also denote $\L_\xi$. 

\begin{example}\label{ex:chern-classes}
Let $G=T$ be a torus. We will see in Section \ref{sec:inertia1} that $I_\mu(X\sslash T)$ has a decomposition into open and closed substacks that are global quotients by $T$. Hence any character $\xi$ of $T$ defines a line bundle $\L_\xi$ on $I_\mu(X\sslash T)$. Likewise, $\overline{I}_\mu(X\sslash T)$  has a decomposition into open and closed substacks $(\overline{I}_\mu(X\sslash T))_t$ that are global quotients by groups of the form $T/\langle t \rangle$, where $\langle t \rangle$ is the subgroup of $T$ generated by some element $t \in T$ of finite order $r$. In this case ${r\xi}$ is a character of $T/\langle t \rangle$, and we adopt the notational convention
\[
c_1(\L_\xi) := \frac{1}{r}c_1(\L_{r\xi}) \in A_*((\overline{I}_\mu(X\sslash T))_t)_\QQ.
\]
\end{example}

We continue to work with a $G$-variety $Y$ but we now assume $[Y/G]$ is a Deligne-Mumford stack. For $\xi \in \chi(G)$ and $\beta \in \Hom(\Pic^G(Y), \QQ)$ we define operational Chow classes on $[Y/G]$ by
\begin{align*}
C^\circ(\beta, \xi) 
&:= \begin{cases}
\prod_{\beta(\xi) < k < 0,  \;k-\beta(\xi) \in \ZZ}^{}(c_1(\L_{\xi}) + kz) & \beta(\xi) \leq 0\\
\left[\prod_{0 < k \leq \beta(\xi),  \;k-\beta(\xi) \in \ZZ}^{}(c_1(\L_{\xi}) + kz)\right]^{-1} & \beta(\xi) >0.
\end{cases}\\
C(\beta, \xi) 
&:= \begin{cases}
c_1(\L_{\xi})C^\circ(\beta, \xi) & \beta(\xi) \in \ZZ_{< 0}\\
C^\circ(\beta, \xi) & else.
\end{cases}\\
\end{align*}
Here the empty product (when $\beta(\xi) = 0$) is defined to be 1, and for an operational Chow class $\alpha$ and rational number $k\neq 0$ we use the convention 
\begin{equation}\label{eq:invert}(\alpha + kz)^{-1} := \frac{1}{kz} \sum_{i=0}^\infty \left(\frac{-\alpha}{kz} \right) ^i,\end{equation} noting that this infinite sum is a well-defined operator on $A_*([Y/G])$ by \cite[Cor~5.3.2]{kresch}.

\begin{remark}
The convention \eqref{eq:invert} also allows us to define $C^\circ(\beta, \xi)^{-1}$, and hence $C(\beta, \xi)^{-1}$ when $\beta(\xi) \not \in \ZZ_{< 0}$. 
\end{remark}

\subsubsection{Abelianization}\label{sec:abel-setup}We adopt the notation of Section \ref{sec:setup}. Let $T \subset G$ be a maximal torus; by restriction, $\theta$ defines a character of $T$. We assume:
\begin{enumerate}
\item $X^s(G)=X^{ss}(G)$ and $X^s(T)=X^{ss}(T)$, and these are both smooth and nonempty
\item If $g \in G$ has a nontrivial fixed locus in $X^s(G)$, then 
\begin{enumerate}
\item $g$ is semi-simple (meaning it is contained in a maximal torus), and 
\item the centrilizer $Z_G(g)$ is connected.\footnote{
Example: If $G = PGL_2(\CC)$ and $g = \left( \begin{array}{cc} 1 & 0\\
0 & -1 \end{array}\right)$ then $Z_G(g) = \left( \begin{array}{cc} a & 0\\
0 & b \end{array}\right) \cup \left( \begin{array}{cc} 0 & c\\
d & 0 \end{array}\right)$, so neither $Z(g)$ nor $Z(g)/\langle g \rangle$ is connected.}
\end{enumerate}
\end{enumerate}
The first assumption ensures that quasimap theory is defined for the pairs $(X, G)$ and $(X,T)$, while the second is necessary to relate the inertia stacks of $X\sslash T$ and $X\sslash G$.

\begin{remark}
The condition (2b) follows from (2a) if $G$ is simply connected or a direct product of general linear groups.
\end{remark}

The abelianization theorem uses the following morphisms:
\begin{equation}\label{eq:key_diagram}
\begin{tikzcd}
X\sslash_G T \arrow[hook, r, "j"] \arrow[d, "\varphi"] & X\sslash T\\
X\sslash G &
\end{tikzcd}
\end{equation}
These induce morphisms, which we denote with the same letters, of the corresponding (rigidified) inertia stacks. Let $\rho_1, \ldots, \rho_m \in \chi(T)$ denote the roots of $G$ with respect to $T$.

\begin{theorem}\label{thm:main}
The $I$-functions of $X\sslash G$ and $X\sslash T$ satisfy
\begin{equation}\label{eq:main}
\varphi^* I^{X\sslash G}_{\beta}(z) =  \sum_{\tilde \beta \rightarrow \beta}  \left(\prod_{i=1}^m C(\tilde \beta, \rho_i)^{-1}\right)j^*I^{X\sslash T}_{\tilde \beta}(z),
\end{equation}
where the sum is over all $\tilde \beta$ mapping to $\beta$ under the natural map $\Hom(\Pic^T(X), \ZZ) \rightarrow \Hom(\Pic^G(X), \ZZ)$. This formula also holds with $I^{X\sslash G}_\beta$ and $I^{X\sslash T}_{\tilde \beta}$ replaced by $\overline{I}^{X\sslash G}_{\beta}$ and $\overline{I}^{X\sslash T}_{\tilde \beta}$, and these formulas uniquely determine ${I}^{X\sslash G}$ and $\overline{I}^{X\sslash G}$.
\end{theorem}
See Example \ref{ex:chern-classes} for the definition of $c_1(\L_{\rho_i})$. The injectivity of $\varphi^*$ and the denominators appearing on the right hand side of \eqref{eq:main} are explained in Remark \ref{rmk:abel-explain}.

\begin{remark}
Theorem \ref{thm:main} is closely related to the main result of \cite{gonzalez}. However, in \cite{gonzalez} the authors only consider the untwisted sector, whereas our equality \eqref{eq:main} applies in all sectors. 
\end{remark}

\subsubsection{Quantum Lefschetz}\label{sec:ql-setup}
We extend the situation studied in Section \ref{sec:abel-setup} by adding the additional data of a vector bundle and section. That is, let $E$ be a $G$-representation, and let $s$ be a $G$-equivariant section of the vector bundle $E \times X \rightarrow X$ with $Y \subset X$ its zero locus. In addition to (1) and (2) in Section \ref{sec:abel-setup} we assume:
\begin{enumerate}
\setcounter{enumi}{2}
\item The section $s$ is regular, meaning that it is locally given by a regular sequence, and $X^s(G) \cap Y$ is smooth.
\end{enumerate}
Assumption (1) implies that $Y^s(G) = Y^{ss}(G)$ (see Lemma \ref{lem:stable-locus}) and (3) implies that $Y$ is l.c.i. with smooth stable locus, so quasimap theory is defined for $(Y,G)$. Observe that if $G$ is a torus then assumption (2) is automatic. 

We extend the diagram \eqref{eq:key_diagram} as follows. For $\tilde \delta \in \Hom(\Pic^T(X), \QQ)$, there there is a certain moduli stack of quasimaps $F_{\tilde \delta}(X\sslash T)$
(Section \ref{sec:lambda-action}) with a closed embedding $ev_\star$ to $I_\mu(X\sslash T)$ (Lemma \ref{lem:abelian}). The significance of these data is that the $I$-function coefficient $I^{X\sslash T}_{\tilde \delta}$  is defined by 
\[
I^{X\sslash T}_{\tilde \delta} = \iota_* ev_{\star} \Res^{X\sslash T}_{\tilde \delta}, 
\]where $\iota$ is the involution on the inertia stack $I_\mu(X\sslash T)$ and $\Res^{X\sslash T}_{\tilde \delta}$ is a certain element of the equivariant Chow group of $F_{\tilde \delta}(X\sslash T)$ (Section \ref{sec:loc-res} and Definition \ref{def:ifunc}). 
The closed embedding $ev_\star$ determines the following fibered diagram, where for the time being we define $F^0_{\tilde \delta}(Y\sslash T)$ and $F^0_{\tilde \delta}(X\sslash T)$ so that the diagram is fibered (but see Proposition \ref{prop:diagram2}).
\begin{equation}\label{eq:key-ql-diagram}
\begin{tikzcd}
F^0_{\tilde \delta} (Y\sslash T) \arrow[d, hookrightarrow, "ev_\star"] \arrow[r] & F^0_{\tilde \delta}(X\sslash T) \arrow[d, hookrightarrow]  \arrow[r, hookrightarrow, "j_F"] & F_{\tilde \delta}(X\sslash T) \arrow[d, hookrightarrow, "ev_\star"]\\
I_{\mu}(Y\sslash_G T) \arrow[r, "\bi"] \arrow[d, "\varphi"] & I_{\mu}(X\sslash_G T) \arrow[d] \arrow[r, hookrightarrow, "j"] & I_\mu(X\sslash T) \\
I_{\mu}(Y\sslash G) \arrow[r] & I_{\mu}(X\sslash G)
\end{tikzcd}
\end{equation}
    We show in Lemma \ref{lem:ql-inertia} that $\bi$ is a regular local immersion in the sense of \cite[Sec~3.1]{kresch}, and hence we have a Gysin map $\bi^!: A_*(I_{\mu}(X\sslash_G T))_\QQ \to A_*(I_{\mu}(Y\sslash_G T) )_\QQ$. On the other hand, in \eqref{eq:fiber-fixed} we define a refined Gysin map $0^!_{\tilde \delta}: A_*(F^0_{\tilde \delta}(X\sslash T))_\QQ \to A_*(F^0_{\tilde \delta}(Y\sslash T))_{\QQ}$. We will see in Theorem \ref{thm:ql} that $0^!_{\tilde \delta}$ is the correct morphism for comparing the $I$-functions of $X\sslash G$ and $Y\sslash G$ in general, while $\bi^!$ can be used if the class $\tilde \delta$ is \textit{I-nonnegative}.  

We define I-nonnegativity as follows. Denote the weights of $E$ with respect to $T$ by $\epsilon_j$, for $j=1, \ldots, r$. We say that $\tilde \delta$ is \textit{I-nonnegative} if the set
\[
\{ \epsilon \in \{\epsilon_j\}_{j=1}^r \mid \tilde \delta(\epsilon) \in \ZZ_{<0}\}
\]
is empty. We explain the reason for the term ``I-nonnegative'' in Remark \ref{rmk:reason}.

\begin{remark}\label{rmk:gysin}
The class $\tilde \delta$ is I-nonnegative if and only if $C(\tilde \delta, \epsilon_j) = C^\circ(\tilde \delta, \epsilon_j)$ for all $j=1, \ldots, r$. 
Moreover, we show in Lemma \ref{lem:compare-gysin} that if $\tilde \delta$ is I-nonnegative then the refined Gysin maps $\bi^!, 0^!_{\tilde \delta}: A_*(F^0_{\tilde \delta}(X\sslash T))_\QQ \to A_*(F^0_{\tilde \delta}(Y\sslash T))_{\QQ}$ agree. 
\end{remark}

\begin{theorem}\label{thm:ql}
For $\delta \in \Hom(\Pic^G(X), \QQ)$, we have 
\[
\sum_{\beta \mapsto \delta} \varphi^*I_\beta^{Y\sslash G}(z) = \sum_{\tilde \delta \mapsto \delta}
 I_{\tilde \delta}(z)\]
where the first sum is over all $\beta$ mapping to $\delta$ under the natural map $\Hom(\Pic^G(Y), \QQ) \to \Hom(\Pic^G(X),\QQ)$, the second sum is over all $\tilde \delta$ mapping to $\delta$ under the natural map $\Hom(\Pic^T(X), \QQ) \to \Hom(\Pic^G(X), \QQ)$, and we define \begin{equation}\label{eq:ql-main}
I_{\tilde \delta}(z) =
\left( \prod_{i=1}^m C(\tilde \delta, \rho_i)^{-1}\right)\left(
\prod_{j=1}^r  C^\circ(\tilde \delta, \epsilon_j)^{-1}\right) \iota_*(ev_{\star})_*0_{\tilde \delta}^!j_F^* Res^{X\sslash T}_{\tilde \delta}(z).
\end{equation}
 If $F^0_{\tilde \delta}(X\sslash T)$ is empty then $I_{\tilde \delta}(z)$ is defined to be zero. An analogous formula for $\overline{I}^{Y\sslash G}_\beta$ holds with $\iota_* (ev_\star)_*$ replaced by $(\widetilde{ev}_\star)_*$ (defined in Section \ref{sec:ifunc-def}).
\end{theorem}

See Example \ref{ex:chern-classes} for the definition of $c_1(\L_{\rho_i})$ and $c_1(\L_{\epsilon_j})$. When $\tilde \delta$ is I-nonnegative, we can permute the closed embedding $ev_\star$ and the refined Gysin map $0_{\tilde \delta}^!$ (equal to the refined Gysin map $\bi^!$ in this case, see Remark \ref{rmk:gysin}) and obtain the following formula.

\begin{corollary}\label{cor:ql-convex}
In Theorem \ref{thm:ql}, for I-nonnegative classes $\tilde \delta$ the formula \eqref{eq:ql-main} may be replaced by
\[
I_{\tilde \delta}(z) =
\left( \prod_{i=1}^m C(\tilde \delta, \rho_i)^{-1}\right)\left(
\prod_{j=1}^r  C(\tilde \delta, \epsilon_j)^{-1}\right) \bi^!j^* I^{X\sslash T}_{\tilde \delta}(z).
\]
This formula also holds with $I^{Y\sslash G}_\beta$ and $I^{X\sslash T}_{\tilde \delta}$ replaced by $\overline{I}^{Y\sslash G}_{\beta}$ and $\overline{I}^{X\sslash T}_{\tilde \delta}$, respectively.
\end{corollary}

When $X$ is a vector space we can make the formula even more explicit.

\begin{corollary}\label{cor:ql}
If $X$ is a vector space with $T$-weights $\xi_1, \ldots, \xi_n$, then in Theorem \ref{thm:ql} the formula \eqref{eq:ql-main} may be replaced by
\begin{equation}\label{eq:ql-cor}
I_{\tilde \delta}(z)=  \left( \prod_{i=1}^m C(\tilde \delta, \rho_i)^{-1}\right)
\left(\prod_{j=1}^r C^\circ(\tilde \delta, \epsilon_j)^{-1}\right)
\left(\prod_{\ell=1}^{n}C^\circ(\tilde \delta, \xi_\ell)\right) \iota_*
(ev_\star)_*0^!_{\tilde \delta}[F^0_{\tilde \delta}(X\sslash T)]
\end{equation}
The same formula holds for $\overline{I}_\beta^{Y\sslash T}$. 
 Moreover, \eqref{eq:ql-cor} simplifies further in two special cases:
 \begin{enumerate}
 \item If $\tilde \delta \in \Hom(\chi(T), \QQ)$ is I-nonnegative, we have
 \begin{equation}\label{eq:ql-cor-convex}
I_{\tilde \delta}(z) = \left( \prod_{i=1}^m C(\tilde \delta, \rho_i)^{-1}\right)
\left(\prod_{j=1}^r C(\tilde \delta, \epsilon_j)^{-1}\right)
\left(\prod_{\ell=1}^{n}C(\tilde \delta, \xi_\ell)\right)
\one_{g_{\tilde \delta}^{-1}}
\end{equation}
where $\one_{g_{\tilde \delta}^{-1}}$ is the fundamental class of the component $(I_\mu(Y\sslash_G T))_{g_{\tilde \delta}^{-1}}$ of the inertia stack (see Section \ref{sec:inertia1}).
 \item If the embedding $F^0_{\tilde \delta}(Y\sslash T) \to F^0_{\tilde \delta}(X\sslash T)$ in \eqref{eq:key-ql-diagram} is regular and the excess bundle \eqref{eq:def-excess} for the fiber square \eqref{eq:fiber-fixed} defining $0^!_{\tilde \delta}$ is trivial, then 
 \[\iota_*(ev_\star)_*0^!_{\tilde \delta}[F^0_{\tilde \delta}(X\sslash T)] = [F^0_{\tilde \delta}(Y\sslash T)].\]
 \end{enumerate}

\end{corollary}

We note that in examples, the moduli stacks $F^0_{\tilde \delta}(X\sslash T)$ and $F^0_{\tilde \delta}(Y\sslash T)$ can be described explicitly as closed substacks of $I_\mu(X\sslash T)$: a formula for $F^0_{\tilde \delta}(X\sslash T)$ can be found in \cite[Sec~5.3]{orb-qmaps}, and $F^0_{\tilde \delta}(Y\sslash T)$ is the intersection of $I_\mu(Y\sslash T)$ and  $F^0_{\tilde \delta}(X\sslash T)$ in $I_\mu(X\sslash T)$ (Proposition \ref{prop:diagram2}).

Formula \eqref{eq:ql-cor} appears in \cite{wang} when $G=T$ is abelian. The following remark explains why we need to derive this formula for nonabelian $G$ even in light of Theorem \ref{thm:main}. 

\begin{remark}\label{rmk:need-ql}
To compute the $I$-function of $(Y, G)$ with $G$ nonabelian, one might hope to first compute the $I$-function of $(Y, T)$ using the abelian version of Theorem \ref{thm:ql} and then apply Theorem \ref{thm:main} to $(Y, T)$. Unfortunately it is possible that $Y^{ss}(G)$ is smooth but $Y^{ss}(T)$ is not clearly smooth, so that quasimap theory for the intermediate target $Y\sslash T$ may not be defined. This is the case for the example considered in \cite[Sec~5.6]{SW1}.
\end{remark}

\subsubsection{Equivariant formulations and bigger $I$-functions}\label{sec:intro-equivariance}
Let $R$ be a torus with an action on $X$ that commutes with the action of $G$. The \emph{R-equivariant small I-function} of $X\sslash G$ is a formal series $I^{X\sslash G, R}(z)$ of the same shape as \eqref{eq:Ifunc_shape} (see Remark \ref{rmk:def-equivariant}). Suppose $E$ is a representation of $G \times R$, and $s$ is a $(G \times R)$-equivariant section of $E \times X \to X$ with zero locus $Y \subset X$. We assume that conditions (1)--(3) of Sections \ref{sec:abel-setup} and \ref{sec:ql-setup} hold. By allowing regular sequences of length zero, we include the non complete intersection case in our discussion (i.e., set $E=0$). Note that if $\xi$ is a character of $G \times R$, then the line bundle $\L_\xi$ defined in \eqref{eq:chi-to-pic} is descends to an $R$-linearized line bundle on $[Y/G]$.

\begin{theorem}\label{thm:equivariant}
Let $\rho_1, \ldots, \rho_m$ denote the weights of the Lie algebra of $G$ as a $(T \times R)$-representation, where $T$ acts by the adjoint representation and $R$ acts trivially. Let $\epsilon_1, \ldots, \epsilon_r$ denote the weights of $E$ as a $(T\times R)$-representation. Then for $\delta \in \Hom(\Pic^G(X), \QQ)$, the equality \eqref{eq:ql-main} holds after replacing $I^{Y\sslash G}_\beta(z)$ by $I^{Y\sslash G, R}_\beta(z)$, replacing $I^{X\sslash G}_{\tilde \delta}(z)$ by $I^{X\sslash G, R}_{\tilde \delta}(z)$, and replacing the Chern classes $c_1(\L_\xi)$ in the definition of $C(\beta, \xi)$ with their $R$-equivariant counterparts $c_1^R(\L_\xi)$. Also, the analogous result holds for rigidified $I$-functions.
\end{theorem}
The proof of Theorem \ref{thm:equivariant} is essentially the same as the proof of Theorem \ref{thm:ql}. For simplicity, this paper is written with the assumption that $R$ is trivial. Throughout the text, remarks explain what changes when $R$ is not trivial. 

Using the equivariant small $I$-function we can also compute ``bigger'' $I$-functions.
Let $N$ be a subset of $\chi(G)$, and for each $\eta \in N$ let $x_\eta$ be a formal variable.
Let $\{p_i(\bx)\}_{i=1}^K$ be a set of polynomials in the variables $x_\eta$,
 and let $\{t_i\}_{i=1}^K$ be corresponding formal variables. Define the bigger $I$-function
\begin{equation}\label{eq:bigi}
\II^{X\sslash G, R}(z) =  \sum_{\beta}q^\beta  \exp\left( \frac{1}{z} \sum_{i \in 1}^K t_{i}p_i(c_{1}(\L_{\eta})+\beta({\eta})z)\right) I_\beta^{X\sslash G, R}(z)
\end{equation}
where $p_i(c_{1}(\L_{\eta})+\beta({\eta})z)$ is the polynomial $p_i(\bx)$ with $x_\eta$ replaced by $c_1(\L_{\eta}) + \beta(\eta)z$, for each $\eta \in N$. As usual, the sum is over all $I$-effective classes $\beta$, and the rigidification $\overline{\II}^{X\sslash G, R}(z)$ is defined analogously with $\overline I^{X\sslash G, R}_\beta$ in place of $I^{X\sslash G, R}_\beta$.

By \cite[Thm~4.2]{orb-qmaps}, the image of the function $\overline{\II}^{X\sslash G, R}(z)$ in equivariant Chen-Ruan cohomology lies on the Lagrangian cone when the action of $R$ on the coarse moduli space of $X\sslash G$ has isolated fixed points and isolated 1-dimensional orbits. It is conjectured to lie on the cone without these assumptions.

\begin{remark}
One can take $N$ to be a set of $W$-invariant characters of $T$, unioned with the identity character. If we set the $p_i(\bx)$ to be equal $x_\eta$, one for each $\eta \in N$, then the exponential term in \eqref{eq:bigi} is $\exp(z^{-1}\sum_{\eta \in N}t_\eta(c_1(\L_\eta) + \beta(\eta)z))$. Together with Corollary \ref{cor:ql}, this produces a formula for the series $\varphi^*\II^{X\sslash G, R}(z)$ when $X$ is a complete intersection in a vector space.
\end{remark}

\subsection{Conventions and notation}
\label{sec:conventions}
All stacks (and schemes) are defined over the base field $\CC$. Because the definition of the $I$-function has denominators, we will always use Chow groups tensored with $\QQ$. From now on, if $X$ is a stack, we will write $A_*(X) := A_*(X)_\QQ = A_*(X)\otimes_{\ZZ}\QQ$, where on the right hand side $A_*(X)$ is the Chow group of \cite{kresch}. Also, if $X$ is a stack and $S$ is a scheme we will write $X_S := X\times S.$

If $g \in G$ then $(g)$ denotes the conjugacy class of $g$ and $Z_G(g)$ denotes the centrilizer. When there is no danger of confusion we write $Z(g)$ for $Z_G(g)$. If $T$ is a maximal torus of $G$ then we write $N_G(T)$ for the normalizer and $W = N_G(T)/T$ for the Weyl group.

\subsection{Acknowledgements}
This project started when Elana Kalashnikov pointed out the open conjecture in \cite{OP}. Thanks are especially due to Nawaz Sultani for many helpful discussions, and for pointing out errors in my initial attempts at a quantum Lefschetz formula. I am also grateful to Martin Olsson and Yang Zhou for useful discussions and the proofs of Lemmas \ref{lem:open} and \ref{lem:stable-locus}, respectively. An anonymous referee identified a critical mistake in an earlier draft of this paper. The author was partially supported by an NSF Postdoctoral Research Fellowship, award number 200213.

\section{Abelianization and inertia stacks}\label{sec:inertia} We analyze the behavior of the inertia stack under abelianization.
\subsection{Inertia stacks as global quotients}\label{sec:inertia1}
Let $X$ be any variety with an action $a: G \times X \rightarrow X$ by a complex reductive group $G$. For any subscheme $H \subset G$, define $\cover{H}{X}$ to be the fiber product
\begin{equation}\label{eq:inertia-square}
\begin{tikzcd}
\cover{H}{X} \arrow[r] \arrow[d] & H \times X \arrow[d, "{(a, pr_2)}"] \\
X \arrow[r, "\Delta"] & X \times X
\end{tikzcd}
\end{equation}
where $\Delta$ is the diagonal. Informally, we have $\cover{H}{X}$ consists of pairs of elements $(h, x)$ in $H \times X$ such that $h$ fixes $x$. If $H$ is a subgroup, then $\cover{H}{X}$ is the stabilizer group algebraic space in \cite[Tag~0448]{tag}. 
Observe that $G$ acts on $H \times X$ by the rule $g\cdot (h, x) =  (ghg^{-1}, gx)$ for $g \in G, h \in H$, and $x \in X$. This action makes $(a, pr_2)$ $G$-equivariant and hence the same rule defines an action of $G$ on $\cover{H}{X}$.
These observations let us realize the (rigidified) cyclotomic inertia stacks (defined in \cite[Sec~3]{AGV08}) as global quotients. 

\begin{remark}\label{rmk:inertia1}
 By \cite[Tag~06PB]{tag}, we have
\[
I_\mu(X\sslash G) = [\cover{G}{X^s(G)}/G]. 
\]
Likewise, if $T$ is a maximal torus of $G$, we have $I_\mu(X\sslash_G T)  = [\tcover{X^s(G)}/T]$.
\end{remark}

It follows from Remark \ref{rmk:inertia1} that $I_\mu(X\sslash G)$ has a stratification indexed by conjugacy classes $(g)$ of $G$ (resp. elements of $T$), which we write as
\begin{equation}\label{eq:components}
I_\mu(X\sslash G) = \bigsqcup_{(h) \in \Conj(G)} (I_\mu(X\sslash G))_{(h)} \quad \quad \quad (I_\mu(X\sslash G))_{(h)}:= \cover{(h)}{X}\sslash G .
\end{equation}
When $G$ is abelian, conjugacy classes of $G$ are of course simply group elements, so in the case when $T \subset G$ is a maximal torus we have formulae
\[
\begin{gathered}
I_\mu(X\sslash T) = \bigsqcup_{t \in T} (I_\mu(X\sslash T))_t := \bigsqcup_{t \in T} (X^t\sslash T)\\
I_\mu(X\sslash_G T) = \bigsqcup_{t \in T} (I_\mu(X\sslash_G T))_t := \bigsqcup_{t \in T} (X^t\sslash_G T).
\end{gathered}
\]
where $X^t\subset X$ is the fixed locus of the group element $t$. Since conjugacy classes of semisimple elements are closed \cite[Sec~18.2]{humphreys}, assumption (2) in Section \ref{sec:abel-setup} implies that each stratum above is closed. By Lemma \ref{lem:finite} below, there are finitely many strata, so in fact these are decompositions into open and closed substacks of the inertia stacks. There are analogous decompositions of the rigidifications, also into open and closed substacks.

\begin{lemma}\label{lem:finite}
There are finitely many $t \in T$ (resp. conjugacy classes $(h) \in \Conj(G)$) such that the stratum $(I_\mu(X\sslash_G T))_t$ (resp. $(I_\mu(X\sslash G))_{(h)}$) is nonempty.
\end{lemma}
\begin{proof}
By Remark \ref{rmk:inertia1} we have $I_\mu(X\sslash_G T) = [\cover{T}{X^s(G)}/T]$ with $\cover{T}{X^s(G)}$ a closed subscheme of $T \times X^s(G)$. We claim there is a finite subset $\cT \subset T$ such that $\cover{T}{X^s(G)} \subset \cT \times X^s(G)$. This implies the Lemma for $I_\mu(X\sslash_G T)$. By our assumptions in Section \ref{sec:abel-setup}, if $(h)$ is a conjugacy class such that $I_\mu(X\sslash G)_{(h)}$ is not empty, then there is some $t \in (h)$ with $t \in \cT$, so the Lemma holds for $I_\mu(X\sslash G)$ as well.

To prove the claim we show there exists $N \in \ZZ$ such that if $X^t \neq \emptyset$ then the order of $t$ is less than or equal to $N$. Granting this, we can choose $\cT \subset T$ to be the set of elements of order at most $N$. To find $N$ we use two facts: first, $X^s(G)$ is quasicompact, being an open subscheme of the Noetherian scheme $X$, so $\cover{T}{X^s(G)} \subset T \times X^s(G)$ is also quasicompact. Second, by \eqref{eq:inertia-square} and \cite[Tag~02XE]{tag} and the identification $T \times X^s(G)  = X^s(G) \times_{X\sslash_G T} X^s(G)$ there is a fiber diagram
\[
\begin{tikzcd}
\cover{T}{X^s(G)} \arrow[r] \arrow[d] &\arrow[d] T \times X^s(G)  \arrow[r] & X\sslash_G T \arrow[d]\\
X^s(G) \arrow[r] & X^s(G) \times X^s(G) \arrow[r] & X\sslash_G T \times X\sslash_G T
\end{tikzcd}
\]
The diagonal morphism $X\sslash_G T \rightarrow X\sslash_G T \times X\sslash_G T$ is quasicompact and unramified, hence quasifinite (see \cite[Tag~02VF]{tag}), so $\cover{T}{X^s(G)} \rightarrow X^s(G)$ is also quasifinite. Now by \cite[Tag~03JA]{tag} the map $\pi:\cover{T}{X^s(G)} \rightarrow X^s(G)$ is universally bounded, meaning that there exists $N \in \ZZ$ such that every fiber of $\pi$ has size at most $N$. If $X^t \neq \emptyset$, then there is a closed point $x \in X$ such that $tx=x$; i.e., $t$ is in the fiber of $\pi$ at $x$, which we have seen is a group of order at most $N$. Hence the order of $t$ is at most $N$.
\end{proof}

\begin{remark}
When the torus $R$ is not trivial, we get an action of $R$ on the diagram \eqref{eq:inertia-square}: indeed, $R$ acts on $H \times X$ via the action on the second factor. The maps $\Delta$ and $(a, pr_2)$ in \eqref{eq:inertia-square} are equivariant, and hence $\cover{H}{X}$ has a $G\times R$ action. This means that $R$ acts on $I_\mu(X\sslash G)$ and preserves the decomposition into components in \eqref{eq:components}.
\end{remark}
\subsection{Weyl action}
The stacks $I_\mu(X\sslash T)$ and $\overline{I}_\mu(X\sslash T)$ both have an action by the Weyl group $W$ making the rigidification $\varpi: I_\mu(X\sslash T) \rightarrow \overline{I}_\mu(X\sslash T)$ equivariant. For a scheme $S$, an object of $\overline{I}_\mu(X\sslash T)$ is a gerbe $\mathcal{G} \rightarrow S$ with a representable map $\mathcal{G} \rightarrow X\sslash T$, so we can define $w \in N_G(T)$ to act on the map $\mathcal{G} \rightarrow X\sslash T = [X^{s}(T)/T]$ as in \cite[Sec~2.2.1]{nonab1}. Objects of $I_\mu(X\sslash T)$ are the same except that $\mathcal{G}$ is required to be trivial, so we can define the action in the same way.

The morphism $\varphi: X\sslash_G T \rightarrow X\sslash G$ induces maps $I_\mu(X\sslash_G T) \rightarrow I_\mu(X\sslash G)$ and $\overline{I}_\mu(X\sslash_G T) \rightarrow \overline{I}_\mu(X\sslash G)$, which we also call $\varphi$. These maps are $W$-invariant. For every $h \in T$ with $X^h$ nonempty, we have a commuting diagram where the right square is fibered.
 \begin{equation}\label{eq:inertia4}
 \begin{tikzcd}
 X^h \sslash_G T \arrow[d] \arrow[r, hook, "\eta_T"] \arrow[d, "\varphi_h"]& \bigsqcup_{t \in (h)\cap T} (I_\mu(X\sslash_G T))_t \arrow[d] \arrow[r, hook] & I_{\mu}(X \sslash_G T) \arrow[d, "\varphi"]\\
 X^h\sslash_G Z_G(h) \arrow[r, "\sim"', "\eta_G"] & (I_\mu(X\sslash G))_{(h)} \arrow[r, hook] & I_{\mu}(X\sslash G)
 \end{tikzcd}
 \end{equation}
 The map $\eta_T$ is an open and closed embedding while $\eta_G$ is an isomorphism. Note that $\varphi_h$ is flat. Since $(I_\mu(X\sslash G))_{(h)}$ is an open and closed substack of $I_{\mu}(X\sslash G)$, and $X^h \sslash_G T = (I_\mu(X\sslash_G T))_h$ is an open and closed substack of $I_\mu(X\sslash_G T)$, and for every nonempty component $(I_\mu(X\sslash G))_{(h)}$ of $I_{\mu}(X\sslash G)$ there is a semisimple representative $h \in (h)$, we see that $\varphi$ is flat. Hence we may consider the induced pullback morphism $\varphi^*$ on Chow groups.

\begin{lemma}\label{lem:abelianize-chow}
The pullback $\varphi^*$ induces isomorphisms
\[
A_*(I_\mu(X\sslash G)) \xrightarrow{\sim} (A_*(I_\mu(X\sslash_G T)))^W \quad \quad \quad A_*(\overline{I}_\mu(X\sslash G)) \xrightarrow{\sim} (A_*(\overline{I}_\mu(X\sslash_G T)))^W.
\]
\end{lemma}
\begin{proof}
We first prove the statement for unrigidified inertia.

The left square in \eqref{eq:inertia4} induces a commuting diagram of pullbacks
\begin{equation}\label{eq:inertia5}
\begin{tikzcd}
(A_*(X^h\sslash_G T))^{W_{Z(h)}} & \arrow[l, "\eta_T^*"'] (\bigoplus_{t \in (h)\cap T} A_*(X^t\sslash_G T))^W \\
A_*(X^h\sslash_G Z_G(h)) \arrow[u, "\varphi_h^*"] & A_*(I_\mu(X\sslash G)_{(h)})\arrow[l, "\sim", "\eta_G^*"'] \arrow[u, "\varphi^*"]
\end{tikzcd}
\end{equation}
where $W$ (resp. $W_{Z(h)}$) is the Weyl group of $G$ (resp. $Z_G(h)$) with respect to $T$. In the top row, $\bigoplus_{t \in (h)\cap T} A_*(X^t\sslash_G T)$ is a direct sum of vector spaces with an action by the finite group $W$, and $W_{Z(h)}$ is the stabilizer of the summand $A_*(X^h\sslash_G T)$. The map $\eta_T^*$ is induced by projection to this summand. By \cite[Thm~10]{Br98} the map $\varphi_h^*$ is an isomorphism. We will show that $\eta_T^*$ is also an isomorphism. The lemma statement follows.

To show that $\eta_T^*$ is an isomorphism we claim that
\begin{equation}\label{eq:inertia6}
(\eta_T^*)^{-1}(\delta) := \sum_{w \in W/W_{Z(h)}} (w^{-1})^* \delta\end{equation}
is an inverse, where the sum is over a set of coset representatives and $w^{-1}: X^{w\cdot h}\sslash_G T \rightarrow X^h\sslash_G T$ is induced by the $G$-action. Indeed suppose we have an element 
\[(\delta_t)_{t \in (h) \cap T} \in (\bigoplus_{t \in (h)\cap T}A_*(X^t\sslash_G T))^W.\] 
If $t \in (h) \cap T$, then there is some $w \in W$ such that $t = w \cdot h,$ so $W$-invariance determines each $\delta_t$ from $\delta_h$ by the rule $\delta_t = (w^{-1})^*\delta_h.$ This shows that $(\eta_T^*)^{-1}\circ \eta_T^*$ is the identity. To see that $\eta_T^*\circ(\eta_T^*)^{-1}$ is the identity, note that $W_{Z(h)}$ is the subgroup of $W$ that sends the component $X^h\sslash_G T$ of $\bigsqcup_{t \in (h)\cap T} X^t \sslash_G T$ to itself.

Now we prove the statement for rigidified inertia. A diagram analogous to \eqref{eq:inertia4} holds for rigidified inertia stacks, where the left column is replaced by $\overline{\varphi}_h: X^h \sslash_G( T/\langle h \rangle) \rightarrow X^h_G \sslash (Z_G(h)/\langle h \rangle )$. We must check that $T/\langle h \rangle \subset Z_G(h)/\langle h \rangle$ is a maximal torus with Weyl group isomorphic to $W_{Z(h)}$. Granting this, the rest of the proof follows as before.

Let $k$ be the rank of $T$. To see that $T/\langle h \rangle $ is a torus and maximal in $Z_G(h)/\langle h \rangle$ we use the theory of diagonalizable groups in \cite[Sec~9.1]{milne-reductive}. Namely, the quotient $T/\langle h \rangle $ corresponds to the kernel   $K$ of a map $\ZZ^k \rightarrow \langle h \rangle$, which is a free abelian group of rank $k$, so $T/\langle h \rangle $ is a torus. If $Z_G(h)$ has rank $m > k$ then there is a torus $T' \subset Z_G(h)$ of rank $m$ with a surjection $T' \rightarrow T/\langle h \rangle$. In the category of finitely generated abelian groups this corresponds to an inclusion $K \hookrightarrow \ZZ^m$, and the subgroup of $Z_G(h)$ with quotient $T'$ corresponds to the pushout of the diagram
\[
\begin{tikzcd}
K \arrow[r, hook] \arrow[d, hook] & \ZZ^k \arrow[d, dashrightarrow] \arrow[r] & \langle h \rangle \\
\ZZ^m \arrow[r, dashrightarrow] & (\ZZ^k \oplus \ZZ^m) / K
\end{tikzcd}
\]
which is a free abelian group of rank $m$. In particular $Z_G(h)$ has rank $m>k$, a contradiction.
It is straightforward to check that $N_G(T)/\langle h \rangle = N_{G/\langle h \rangle}(T/\langle h \rangle)$ and hence the Weyl groups are isomorphic.
\end{proof}

\begin{remark}
When the torus $R$ is not trivial,
we define $A^R_*(X\sslash G) := A_*([(X\sslash G)/R])$, observing that
\[
A^R_*(X\sslash G) = A_*([(X\sslash G)/R]) = A_*([X^s(G)/(G\times R)]) = A_*^{G\times R}(X^s(G)).
\]
Now Lemma \ref{lem:abelianize-chow} holds $R$-equivariantly, and we may use the same proof after replacing $G$ with $G \times R$ and $T$ with $T \times R$. Observe that this replacement does not change the Weyl group (up to isomorphism).
\end{remark}

The discussion in the remainder of this section is not critical to the paper, but it enables us to interpret the right hand side of \eqref{eq:main} as an element of $A_*(X\sslash_G T)[z,z^{-1}]$ directly (i.e., without going through the proof of Theorem \ref{thm:main}). This is analogous to \cite[Sec~5.4]{nonab1}.

Let $\{h_j\}_{j \in J}\subset T$ be a set of elements such that $I_\mu(X\sslash G) = \bigsqcup_{j \in J} I_\mu(X\sslash G)_{(h_j)}$ (in particular, the $h_j$ belong to distinct conjugacy classes in $G$). We recall (see e.g. \cite[Sec~1]{ellingsrud}) the sign function $\sgn: W \to \{\pm 1\}$, and that
an element $\alpha$ of some $W$-module is said to be $W$-anti-invariant if 
$w \cdot \alpha = (-1)^{\sgn(w)}\alpha$ for every $w \in W$. For each $j$, the pullback
\[
\eta_{j, T}^* \colon \bigoplus_{t \in (h_j)\cap T} A_*(X^t\sslash_G T) \to A_*(X^{h_j}\sslash_G T)
\]
sends $W$-anti-invariant classes to $W_{Z(h_j)}$-anti-invariant classes. This map of anti-invariant classes has an inverse given by
\[
(\eta_{j, T}^*)^{-1}_a(\delta) = \sum_{w \in W/W_{Z(h_j)}} \sgn(w) (w^{-1})^{*}\delta
\]
where the sum is over a set of coset representatives (one can check that $(\eta_{j, T}^*)^{-1}_a$ does not depend on the choice of representatives). 

Let $R$ be the set of roots of $G$ with respect to $T$ and fix $R^+$ a system of positive roots. Since $h_j \in T$, the roots of $Z(h_j)$ are naturally a subset of $R$, and we denote these by $R_{Z(h_j)}$ and the positive roots by $R^+_{Z(h_j)}$. Recall the fundamental $W_{Z(h_j)}$-anti-invariant class $\Delta_j \in A^*(X^{h_j}\sslash_G T)$ defined by $\Delta_j := \prod_{\rho \in R^+_{Z(h_j)}}c_1(\L_\rho)$. Now define
\[
\Delta := \sum_{j \in J} (\eta_{j, T}^*)^{-1}_a(\Delta_j) \in A^*(I_\mu(X\sslash_G T))
\]
and observe that by construction, $\Delta$ is $W$-anti-invariant.

\begin{lemma}\label{lem:divide}
If $\alpha \in A_*(I_\mu(X\sslash_G T))$ is $W$-anti-invariant, then there exists a unique $\beta \in A_*(I_\mu(X\sslash_G T))$ such that $\Delta \cap \beta = \alpha. $ Moreover, $\beta$ is $W$-invariant.
\end{lemma}
\begin{proof}
Since $\eta_{j, T}^*\alpha$ is $W_{Z(h_j)}$-anti-invariant, by \cite[Lem~5.4.1]{nonab1}\footnote{The reference only applies when $X\sslash G$ is a scheme, but the same proof works using equivariant cohomology when $X\sslash G$ is an orbifold. For example, compare the proofs of \cite[Thm~10]{Br98} and \cite[Prop~2.4.1]{nonab1}.} there exists $\beta_j$ such that $\eta_{j, T}^*\alpha = \Delta_j \cap \beta_j$. Applying $(\eta_{j,T}^*)^{-1}_a$ to both sides, we get that the projection of $\alpha$ to the components of $A_*(X\sslash_G T)$ indexed by elements of $(h_j)\cap T$ is equal to \[
\sum_{w \in W/W_{Z(h_j)}} \sgn(w) (w^{-1})^*(\Delta_j \cap \beta_j).
\]
Hence we may take $\beta = \sum_{j \in J} \sum_{w \in W/W_{Z(h_j)}} (w^{-1})^*\beta_j$.
\end{proof}

\begin{remark}\label{rmk:abel-explain}
Symbolically, one may show as in \cite[Lem~5.4.2]{nonab1} that $\Delta$ capped with the right hand side of \eqref{eq:main} is a $W$-anti-invariant class. 
Hence, using Lemma \ref{lem:divide}, we can define the right hand side of \eqref{eq:main} to be the unique ($W$-invariant) class that, when capped with $\Delta$, is equal to $\sum_{\tilde \beta \to\beta} B_{\tilde \beta}(z)$.
Since $\varphi^*$ is injective (by Lemma \ref{lem:abelianize-chow}), the formula \eqref{eq:main} completely determines $I_\mu(X\sslash G)$. An analogous discussion (including the statement of Lemma \ref{lem:divide}) holds for $\overline{I}_\mu(X\sslash G)$.
\end{remark}

\section{The quasimap $I$-function}
We recall the definition of the quasimap $I$-function in \cite{orb-qmaps} and explain how the construction can be made more explicit in our situation.

\subsection{Preliminaries}
We call a representable morphism $\pi: \P\rightarrow C$ of algebraic stacks a principal $G$-bundle if $\pi$ is faithfully flat and locally finitely presented and we are given an action $\mu:G \times\P\rightarrow \P$ leaving $\pi$ invariant such that the map
\[
G \times \P \xrightarrow{(\mu, pr_2)} \P \times_C\P
\]
is an isomorphism.
If $Z$ is an affine variety with left $G$-action, then the product $\P\times Z$ also has a left $G$-action defined by $g\cdot(p, z) = (gp, gz)$, and we define the \textit{mixing space} to be the quotient
\[
\P \times_G Z = (\P\times Z)/G.
\]
If $T \subset G$ is any subgroup and $\T\rightarrow C$ is a principal $T$-bundle, we define the \textit{associated G-bundle} to be
\begin{equation}\label{eq:principal2}
G \times_T \T = (G \times \T)/T \quad \quad \text{where}\;t\cdot(g, s) = (gt^{-1}, ts)\;\text{for}\;t \in T, \;(g,s) \in G \times \T.
\end{equation}
The quotient $G \times_T \T$ is a (left) principal $G$-bundle with the same transition function as $\T$.

\subsection{Quasimaps}

Let $a\geq1$ be an integer. Define
\[
\wP{a} = \frac{\CC^2_{u,v} \setminus\{0\}}{(u,v) \simeq (t^au, tv)} \quad \quad t \in \CC^*
\]
with fixed projective coordinates $[u:v]$, and let $\star = [1:0]$ denote the orbifold point. For this paper we make the following definition, setting $(\wP{a})_S := \wP{a}\times S$ for a scheme $S$.

\begin{definition}\label{def:qmap}
A \textit{stable graph quasimap} to $X\sslash G$ over a base scheme $S$ is a tuple $((\wP{a})_S, \P, \sigma)$ where
\begin{itemize}
\item $\P$ is an algebraic space and $\P \rightarrow (\wP{a})_S$ is a principal $G$-bundle
\item $\sigma$ is a section of $\P \times_G Z$
\end{itemize}
such that, for every geometric point $s \in S$, the inverse image of the unstable locus $X^{us}_\theta$ is finite and disjoint from the orbifold point $\star$.
\end{definition}

We will simply refer to the objects in Definition \ref{def:qmap} as quasimaps. Isomorphisms of quasimaps are defined as usual. The following lemma shows that our quasimaps are precisely the objects parametrized by the moduli stack $Q_{\wP{a}}(X, \beta)$ defined in \cite[Sec~4.2]{orb-qmaps}. 

\begin{lemma}\label{lem:qmap-equivalences}
Let $C$ be a Deligne-Mumford stack. The following categories are equivalent:
\begin{enumerate}
\item The category of representable morphisms $C \rightarrow [X/G]$
\item The category whose objects are principal $G$-bundles $\P$ on $C$ such that $\P$ is an algebraic space, together with a morphism $\P \rightarrow X$
\item The category whose objects are principal $G$-bundles $\P$ on $C$ such that $\P$ is an algebraic space, together with a section of $\P \times_G X\rightarrow C$.
\end{enumerate}
\end{lemma}
\begin{proof}
Morphisms $q: C \rightarrow [X/G]$ are in bijection with $G$-torsors $\P$ on $C$ and equivariant morphisms to $X$. It follows from \cite[Tag~04ZP]{tag} that $q$ is representable if and only if $\P$ is an algebraic space. This shows $(i) \Leftrightarrow (ii).$ The equivalence $(i)\Leftrightarrow(iii)$ is \cite[Ex~4.2]{survey}.
\end{proof}

If $((\wP{a})_S, \P, \sigma)$ is a quasimap, we will use $q: (\wP{a})_S \rightarrow [Z/G]$ to denote the associated map of stacks and $\tilde q: \P \rightarrow X$ to denote the map of principal bundles defined in Lemma \ref{lem:qmap-equivalences}. We will also refer to the quasimap object with the letter $q$, writing $q=((\wP{a})_S, \P, \sigma)$.

We will often work with a certain class of quasimaps which we now describe. Let
\[\begin{gathered}U_S := \AA^1 \times S \xrightarrow{(u,s) \mapsto ([u:1],s)} \wP{a} \times S \quad \quad \quad \quad V_S := \AA^1\times S\xrightarrow{(v,s) \mapsto ([1:v], s)}  \wP{a}\times S\\
\torus_S:= U_S \times_{(\wP{a})_S}V_S.\end{gathered}\]
Observe that the left map is an open embedding while the right is an \'etale map of degree $a$. The right map is also invariant under the standard action of $a^{th}$ roots of unity on $\AA^1_S$, which we denote $\bmu_a$. The image of $V_S$ is an open substack of $(\wP{a})_S$ which we will denote $[V/\bmu_a] \times S$; likewise we denote the image of $\torus_S$ by $[\torus/\bmu_a]\times S$. Note that $\torus_S\simeq \CC^*_S$.
We fix projection morphisms
\[
\kappa_U:\torus_S \xrightarrow{(w,s) \mapsto (w^a,s)} U_S \quad \quad \quad \quad \kappa_V:\torus_S \xrightarrow{( w,s) \mapsto (w^{-1},s)} V_S.
\]
A transition function $\tau: \torus_S \rightarrow G$ determines a principal bundle on $\wP{a}$ by gluing the pullbacks of the trivial bundles on $U_S$ and $V_S$ via the morphism
\[
(w, g) \rightarrow (w, g\tau^{-1}(w)) \quad \quad w \in \torus_S \quad g \in G.
\]
We denote the resulting bundle by $\P_\tau$. The principal bundle $\P_\tau$ and the associated fiber bundle $\P_{\tau}\times_G X$ may be written as a global quotients, and we have the following equivalent descriptions of the latter:
\begin{equation}\label{eq:global}
\begin{aligned}
\frac{(\CC^2_S \setminus\{0\})\times G \times X}{(u, v, g, x) \sim (t^au, tv, \gamma g \tau(t), \gamma x)} &= \frac{(\CC^2_S \setminus\{0\})\times X}{(u, v, x) \sim (t^au, tv, \tau(t)^{-1}x)} \quad \quad (t, \gamma) \in \CC^*\times G\\
(u, v, g, x, t, \gamma) &\mapsto (u, v, g^{-1}x, t)\\
(u, v, 1, x, t, \tau^{-1}(\gamma)) &\mapsfrom (u, v, x, t)
\end{aligned}
\end{equation}
A quasimap $((\wP{a})_S, \P_\tau, \sigma)$ defines two functions  $\sigma_U: U_S \rightarrow X$ and $\sigma_V: V_S \rightarrow X$, where $\sigma_U$ is the composition
\[
U_S \xrightarrow{\sigma|_{U_S}} (\P_\tau \times_G X)|_{U_S} = U_S \times X \xrightarrow{pr_2} X
\]
and $\sigma_V$ is defined similarly. Hence we have
\begin{equation}\label{eq:coord1}
\tau\cdot(\sigma_U\circ \kappa_U) = \sigma_V\circ \kappa_V \quad \quad\text{on}\; \torus_S,
\end{equation}
and conversely a pair of morphisms $\sigma_U: U_S \rightarrow X$ and $\sigma_V: V_S \rightarrow X$ satisfying \eqref{eq:coord1} define a section of $\P_\tau \times_G X$. 
Two quasimaps $((\wP{a})_S, \P_\tau, \sigma)$ and $((\wP{a})_S, \P_\omega, \rho)$ are isomorphic if and only if there are functions $\phi_U: U_S \rightarrow G$ and $\phi_V: V_S \rightarrow G$ such that
\begin{equation}\label{eq:coordinateiso}\begin{gathered}
(\phi_V\circ \kappa_V)\tau = \omega(\phi_U\circ \kappa_U) \quad \quad \text{as maps}\; \torus_S \rightarrow G\\
\phi_U\cdot\sigma_U = \rho_U \quad \quad \text{as maps}\;U_S \rightarrow Z\\
\phi_V\cdot\sigma_V = \rho_V \quad \quad \text{as maps}\;V_S \rightarrow Z.\\
\end{gathered}\end{equation}
If $((\wP{a})_S, \P_\tau, \sigma)$ is a quasimap and $\phi$ is an automorphism, the local morphisms $\sigma_V$ and $\phi_V$ defined above are twisted-equivariant for the action of $\bmu_a$ on $V$, in the following sense: for $v \in V$ and $\mu \in \bmu_a$, we have relationships
\begin{equation}\label{eq:bmu-equivariance}
\sigma_V(\mu v) = \tau(\mu)^{-1}\sigma_V(v) \quad \quad \text{and} \quad \quad \phi_V(\mu v) = \tau(\mu)^{-1}\phi_V(v)\tau(\mu).
\end{equation}
These equations follow from \eqref{eq:coord1} and \eqref{eq:coordinateiso}, respectively. 

\begin{remark}\label{rmk:coord2}If $k$ is an algebraically closed field, then since every principal $G$-bundle is trivial on $U_k = V_k = \AA^1_k$ (see e.g. \cite{RagRam}) we see that any $k$-quasimap is isomorphic to one of the form $((\wP{a})_k, \P_\tau, \sigma)$ for some transition function $\tau$. 
\end{remark}

\begin{example}\label{ex:cohomology}
For future reference, we record the cohomology of line bundles on $\wP{a}$ in terms of the above coordinates on $\wP{a}$. For an integer $d$ we define 
\[\OO_{\wP{a}}(d):= \frac{(\CC^2 \setminus\{0\})\times \CC}{(u, v, x) \sim (t^au, tv, t^dx)} \quad \quad t \in \CC^*.
\]
Our computation is analogous to that outlined in \cite[Sec~18.3]{vakil} when $a=1$. Specifically, we identify 
\[\Gamma([\torus/\bmu_a], \OO_{\wP{a}}(d)) = \spn\{u^mv^n \mid m, n \in \ZZ, am+n=d\}\]
by viewing $u^mv^n$ in the above set as $\CC^*$-equivariant maps $(\CC^*)^2 \rightarrow \CC$.  With this identification, with the convention $\spn(\emptyset)=0$, we have 
\[
\begin{gathered}
H^0(\wP{a}, \OO_{\wP{a}}(d)) = \spn\{u^mv^n \mid am+n=d\;\text{and}\;m, n \geq 0\} \subset \Gamma([\torus/\bmu_a], \OO_{\wP{a}}(d))\\
H^1(\wP{a}, \OO_{\wP{a}}(d)) = \spn\{u^mv^n \mid am+n=d\;\text{and}\;m, n < 0\} \subset \Gamma([\torus/\bmu_a], \OO_{\wP{a}}(d)).
\end{gathered}
\]
\end{example}

Let $k$ be an algebraically closed field. Consider the degree-$a$ cover of $\wP{a}$ by $\PP^1$, given in homogeneous coordinates by
\begin{equation}\label{eq:cover}
\begin{aligned}
\nu:(\PP^1)_k &\rightarrow (\wP{a})_k\\
[u:v] &\mapsto[u^a:v].
\end{aligned}
\end{equation}
We define the \textit{class} of a quasimap $q=((\wP{a})_k, \P, \sigma)$ 
to be the homomorphism $\beta \in \Hom(\Pic^G(X), \QQ)$ given by
    \[
    \beta(\L) = \deg_{\wP{a}}(q^*\L) =\frac{1}{a}\deg_{\PP^1}(\nu^*q^*\L)\quad \quad \L \in \Pic([X/G]).
    \]
A family of quasimaps $((\wP{a})_S, \P, \sigma)$ has class $\beta$ if each of its geometric fibers over $S$ has class $\beta$. We define the \textit{I-effective classes} $\Eff_I(X, G, \theta) \subset \Hom(\Pic^G(X), \ZZ)$ to be those morphisms that are equal to the class of some stable quasimap.

\begin{remark}
The set $\Eff_I(X, G, \theta)$ is in general a proper subset of the effective classes as defined in \cite[Def~2.2]{orb-qmaps}.
\end{remark}

Let ${QG}_{a, \beta}(X\sslash G)$ denote the groupoid of stable class-$\beta$ quasimaps to $Z \sslash G$ with source curve $\wP{a}$. 
\begin{remark}\label{rmk:a}
By \cite[Lem~4.6]{orb-qmaps}, given a class $\beta$, there is a unique integer $a$ for which $QG_{a, \beta}(X\sslash G)$ is nonempty. 
\end{remark}

We will always assume that $a$ is the integer determined by $\beta$ as in Remark \ref{rmk:a} and omit it from the notation. Furthermore, we will omit the target $X\sslash G$ when it is understood. The space ${QG}_{\beta}$ denoted $Q_{\PP_{a,1}}(X, \beta)$ in \cite[Sec~4.2]{orb-qmaps}. In particular it has a map to $\aff{X}{G}$.

\begin{theorem}[{\cite[Prop~4.5]{orb-qmaps}}]\label{thm:qg}
The moduli space $QG_\beta(X\sslash G)$ is a Deligne-Mumford stack of finite type, proper over $\aff{X}{G}$.
\end{theorem}

 Let $T \subset G$ be a maximal torus. From the morphisms $\chi(G) \rightarrow \chi(T)$ and $\Pic^G(X) \rightarrow \Pic^T(X)$ and the morphism \eqref{eq:chi-to-pic}, we have the following diagram:
\begin{equation}\label{eq:degree_diagram}
\begin{tikzcd}
\Hom(\Pic^T(X),\QQ)\arrow[r,"\rpic"]\arrow[d, "\bdtilde"]&\Hom(\Pic^G(X), \QQ)\arrow[d,"\bd"]\\
\Hom(\chi(T),\QQ)\arrow[r, "\rchi"] & \Hom(\chi(G),\QQ))
\end{tikzcd}
\end{equation}

\begin{remark}\label{rmk:finite}
When $\rpic$ is restricted to $I$-effective classes in both the source and target, it has finite fibers. Indeed, if $\beta \in \Hom(\Pic^G(X), \ZZ)$ is $I$-effective, one may argue as in \cite[Lem~3.1.10]{nonab1} that $\rpic^{-1}(\beta)$ contains finitely many $I$-effective classes (in particular, \cite[Thm~3.2.5]{stable_qmaps} only requires the representation $G \rightarrow GL(V)$ to have \textit{finite} kernel). For general $I$-effective $\beta \in \Hom(\Pic^G(X), \QQ)$, one uses the fact that $a\beta\in \Hom(\Pic^G(X), \ZZ)$ is also $I$-effective, as can be shown by precomposing a quasimap of class $\beta$ with the cover \eqref{eq:cover}.
\end{remark}

\subsection{Perfect obstruction theory}
Let $\pi:\wP{a}\times QG_{\beta}(X\sslash G) \rightarrow QG_{\beta}(X\sslash G)$ be the universal curve and $q: \wP{a}\times QG_{\beta}(X\sslash G) \rightarrow [Z/G]$ be the universal map. Then the virtual cycle constructed in \cite[Prop~4.5]{orb-qmaps} is induced by an absolute perfect obstruction theory
\begin{equation}\label{eq:pot}
\phi: \EE_{QG_{\beta}}:=R \pi_*(q^*\LL_{[X/G]}\otimes \omega^\bullet) \rightarrow \LL_{QG_{\beta}}
\end{equation}
where $\omega^{\bullet}$ is the relative dualizing complex for $\pi$, $\LL_{[X/G]}$ is the cotangent complex, and $\phi$ is defined as in \cite[(18)]{cjw} for the tower of morphisms 
\begin{equation}\label{eq:tower}
[X/G]\times \wP{a} \to \wP{a} \to \bullet.
\end{equation}
(By Remark \ref{rmk:a}, $a$ is determined by $\beta$.) For a proof of these statements, see \cite[Sec~3.2]{nonab1}.

\begin{remark}
If $R$ is not trivial, then \eqref{eq:tower} is $R$-equivariant (we define the $R$-action on $\wP{a}$ to be trivial). By \cite[Sec~A.3]{cjw} this induces an $R$-action on $QG_\beta(X\sslash G)$ and an $R$-equivariant structure on the perfect obstruction theory \eqref{eq:pot}.
\end{remark}
\subsection{I-function}
We define each of the ingredients needed to write down the quasimap $I$-function.

\subsubsection{$\CC^*_\lambda$ action}\label{sec:lambda-action}
Let $\CC^*_\lambda$ act on $\wP{a}$
by
\begin{equation}\label{eq:action1}
\lambda \cdot [u:v] = [\lambda u:v], \quad \quad\quad\quad \lambda \in \CC^*.
\end{equation}
This action is chosen so that the coarse moduli map $\wP{a}\rightarrow \PP^1$ is equivariant for the $\CC^*$-action on $\PP^1$ given by \cite[(28)]{nonab1}. This action induces an action on $QG_{\beta}$ via 
\begin{equation}\label{eq:action3}
\lambda \cdot (\wP{a}\times S, \P, \sigma) = (\wP{a}\times S, (\lambda^{-1})^*\P,\sigma\circ \lambda^{-1}).
\end{equation}

\subsubsection{$\CC^*_\lambda$-fixed locus}
For any integer $d \geq 1$, let $\CC^*_{\lambda^{1/d}}$ denote the torus with coordinate $\lambda^{1/d}$. We let $\CC^*_{\lambda^{1/d}}$ act on $\wP{a}$ via the $d^{th}$ power map $\CC^*_{\lambda^{1/d}} \to \CC^*_\lambda$ and the action \eqref{eq:action1}. In coordinates, we have
\[
\lambda^{1/d}\cdot[u:v] = [(\lambda^{1/d})^d u:v] \quad \quad \lambda^{1/d} \in \CC^*.
\]

We define the $\CC^*_\lambda$-\textit{fixed locus} to be a closed substack of $QG_{\beta}$ as in \cite[Sec~3]{CKL17}---it follows from \cite[Prop~5.20]{AHR} that this definition is equivalent to \cite[Def~5.25]{AHR}. That is, a quasimap $q=((\wP{a})_S, \P, \sigma)$ is fixed if there is some $d \geq 1$ such that $\CC^*_{\lambda^{1/d}}$ fixes $q$; i.e., 
$q:(\wP{a})_S \rightarrow [Z/G]$ factors through $[(\wP{a})_S/ \CC^*_{\lambda^d}]$. This factorization implies that the basepoints of a fixed quasimap, if any, are concentrated at $[0:1]$ and $[1:0]$. Let $F_{ \beta}(X\sslash G)$ be the component of the fixed locus with all its basepoints at $[0:1]$. When the target is clear we will write $F_\beta$ for $F_\beta(X\sslash G)$.

\subsubsection{Localization residue}\label{sec:loc-res}
We now define the localization residue 
\begin{equation}\label{eq:res}
\Res^{X\sslash G}_\beta := e_{\CC^*_\lambda}(N^\vir_{F_\beta/QG_\beta})^{-1}\cap [F_\beta]^{\vir}
\end{equation}
as an element of $A_*([F_\beta/\CC^*_\lambda])$. This class is derived from the morphism $\EE_{QG_\beta}|_{F_\beta} \to \LL_{F_\beta}$ obtained by restricting \eqref{eq:pot} to $F_\beta$ and composing with the canonical map $\LL_{QG_\beta}|_{F_\beta} \to \LL_{F_\beta}$. We will use an alternative description of this morphism: by the proof of \cite[Lem~A.3.5]{cjw} it is quasi-isomorphic to the morphism 
\begin{equation}\label{eq:almost-pot}
\phi_{F_\beta}: \EE_{F_\beta} \to \LL_{F_\beta}
\end{equation}
defined directly via \cite[(55)]{cjw} with $B_W=\mathfrak{U}$ equal to a point, $K_W=W=\mathcal{C} := \wP{a}$, $Z=\wP{a}\times [X/G]$, $B_Z=F_\beta$, and $K_Z=F_\beta \times \wP{a}$. 

We explain how \eqref{eq:almost-pot} is $\CC^*_\lambda$-equivariant. Since $\CC^*_\lambda$ acts on the tower of morphisms $[X/G]\times \wP{a} \to \wP{a} \to \bullet$, we see from \cite[Sec~A.3]{cjw} that there is an induced action on $QG_\beta$ and that its perfect obstruction theory \eqref{eq:pot} is $\CC^*_\lambda$-equivariant (meaning that it is the pullback of a morphism in the derived category of $[QG_\beta / \CC^*_\lambda]$). By \cite[Lem~A.3.3]{cjw} the morphism \eqref{eq:almost-pot} is also equivariant, equal to the pullback of a morphism in the derived category of $[F_\beta/\CC^*_\lambda]$ that we will notate $\phi_{[F_\beta/\CC^*_\lambda]/B\CC^*_\lambda}: \EE_{[F_\beta/\CC^*_\lambda]} \to \LL_{[F_\beta/\CC^*_\lambda]/B\CC^*_\lambda}$.

To define the localization residue, the following remark is helpful.

\begin{remark}\label{rmk:perfect}The complex $\EE_{[F_\beta/\CC^*_\lambda]}$ on $[F_\beta/\CC^*_\lambda]$ has a global resolution by vector bundles.
One way to see this uses Lemmas \ref{lem:nearly} and \ref{lem:group-action}, which together imply that $F_\beta$, and hence $[F_\beta/\CC^*_\lambda]$, is the quotient of a variety with an equivariant family of ample line bundles. By \cite[Prop~2.1]{totaro} the stack $[F_\beta/\CC^*_\lambda]$ has the resolution property. 
Now observe that $\EE_{[F_\beta/\CC^*_\lambda]}$ is perfect because its pullback is the restriction of a perfect complex $\EE_{F_\beta} \simeq \EE_{QG_\beta}|_{F_\beta}$, and being perfect is a flat-local property by \cite[Lem~4.1]{HR17}. Moreover it may be represented by a complex of quasi-coherent sheaves on $[F_\beta/\CC^*_\lambda]$ by \cite[Thm~1.2]{HNR} and \cite[Thm~B]{HR17}. Replacing this complex by its truncation we may assume it is bounded below. Finally, to represent $\EE_{[F_\beta/\CC^*_\lambda]}$ by a complex of vector bundles one may argue as in \cite[Tag~0F8E]{tag}, replacing tags 08CQ and 08DN with tags 08G8 and 08FX, respectively.
\end{remark}

The above remark allows us to decompose the complex $\EE_{[F_\beta/\CC^*_\lambda]}$ and the morphism $\phi_{[F_\beta/\CC^*_\lambda]/B\CC^*_\lambda}$ into fixed and moving parts such that these decompositions are compatible with pullbacks to $F_\beta$. In general, if $E$ is a vector bundle on $[F_\beta/\CC^*_\lambda]$ and $\tilde E$ is its pullback to $F_\beta$, we may write $E=E^{\fix}\oplus E^{\mov}$ such that the pullback of this decomposition to $F_\beta$ is equal to the decomposition $\tilde E=\tilde E^{\fix}\oplus \tilde E^{\mov}$ given in \cite[p.979]{CKL17}.

The fixed part $\phi_{F_\beta}^{\fix}$ is a perfect obstruction theory for $F_\beta$ by \cite[Lem~A.3.5]{cjw} and \cite[Lem~3.3]{CKL17}, and hence it defines a virtual fundamental class $[F_\beta]^\vir \in A_*(F_\beta)$. It follows from the constructions that $\phi_{[F_\beta/\CC^*_\lambda]/B\CC^*_\lambda}^{\fix}$ is a relative perfect obstruction theory (it pulls back to $\phi_{F_\beta}^{\fix}$) and the associated virtual class in $A_*([F_\beta/\CC^*_\lambda])$ pulls back to $[F_\beta]$. Hence we write $[F_\beta]^\vir \in A_*([F_\beta/\CC^*_\lambda])$.

Finally, the virtual normal bundle 
\[
N^{\vir}_{F_{\beta}(X\sslash G)} := (\EE_{F_\beta}^{\mov})^\vee
\]
is the pullback of a complex of vector bundles on $[F_\beta/\CC^*_\lambda]$ by Remark \ref{rmk:perfect}. We define the invertible operational Chow class $e_{\CC^*_\lambda}(N^{\vir}_{F_{ \beta}(X\sslash G)}): A_*([F_\beta/\CC^*_\lambda]) \rightarrow A_*([F_\beta/\CC^*_\lambda])$ as in \cite[(57)]{nonab1}, using the definition in \cite[Def~2.4.2]{kresch} for the Euler class of a vector bundle on $[F_\beta/\CC^*_\lambda]$.

\begin{remark}
If $R$ is not trivial, then the $R$-action on \eqref{eq:tower} commutes with the $\CC^*_\lambda$-action, which means that $R$ acts on $F_\beta$ and \eqref{eq:almost-pot} is $\CC^*_\lambda \times R$-equivariant. Hence the localization residue defined in this section descends to a class in $A_*([F_\beta/(\CC^*_\lambda \times R)])$ simply by replacing $\CC^*_\lambda$ with $\CC^*_\lambda \times R$ in the preceeding discussion.
\end{remark}

\subsubsection{Evaluation map}
Let $\star$ denote the gerbe $B\bmu_a \subset \wP{a}$. Recall that $\beta$ determines the integer $a$ (Remark \ref{rmk:a}). We define $ev_\star: F_{ \beta} \rightarrow I_\mu(X\sslash G)$ to send a quasimap $q:\wP{a}\times S \rightarrow [Z/G]$ to its restriction $B\bmu_a \times S \rightarrow X\sslash G$, an object over $S$ in the \textit{unrigidified} cyclotomic inertia stack $I_{\mu}(X\sslash G)$ (see \cite[Prop~3.2.3]{AGV08}). 

To define the $I$-function, we need to define a map $(ev_\star)_*: A_*([F_\beta/\CC^*_\lambda]) \to A_*(I_\mu(X\sslash G))[z, z^{-1}]$. Since the action of $\CC^*_\lambda$ on $F_\beta$ is only trivial after reparametrization, the morphism $ev_\star: F_\beta \to I_\mu(X\sslash G)$ might not be $\CC^*_\lambda$-equivariant, and it is not immediately clear how to define $(ev_\star)_*$ on $A_*([F_\beta/\CC^*_\lambda])$. 
Our solution is to apply Lemma \ref{lem:almost-trivial} to write $\alpha \in A_*([F_\beta/\CC^*_\lambda])$ as an element of
$A_*(F_\beta)[z, z^{-1}]$, where $z$ is the Euler class of the line bundle on $[F_\beta/\CC^*_\lambda]$ induced by the identity character of $\CC^*_\lambda$, and then we apply $ev_\star$ to the coefficients of this series.\footnote{Suppose $F_\beta$ is fixed by
$\CC^*_{\lambda^{1/a}}$, which acts via $\CC^*_{\lambda^{1/a}} \xrightarrow{\mu\mapsto \mu^a}
\CC^*_\lambda$. The formula described here for computing $(ev_\star)_*\alpha$ is equivalent to first lifting $\alpha$ to a class in
$A_*([F_\beta/\CC^*_{\lambda^{1/a}}])$ using the canonical isomorphism $[F_\beta/\CC^*_{\lambda^{1/a}}] \simeq [F_\beta/\CC^*_\lambda]$, then applying
$(ev_\star)_*: A_*([F_\beta/\CC^*_{\lambda^{1/a}}]) \rightarrow A_*([I_\mu(X\sslash G)/\CC^*_{\lambda^{1/a}}])$,
next applying the canonical isomorphism $[I_{\mu}(X\sslash G)/\CC^*_{\lambda^{1/a}}] \simeq [I_{\mu}(X\sslash G)/\CC^*_{\lambda}]$, and finally applying Lemma \ref{lem:almost-trivial}. This follows from the proof of Lemma \ref{lem:almost-trivial}.}

\begin{remark}
If $R$ is not trivial, then we need to define a map
\begin{equation}\label{eq:R1}
(ev_\star)_*: A_k([F_\beta/(\CC^*_\lambda \times R)]) \to A_k([I_\mu(X\sslash G)/R])[z,z^{-1}].
\end{equation}
Since $ev_\star$ is $R$-equivariant by definition, we have  
\begin{equation}\label{eq:R2}
(ev_\star)_*: A^R_k(F_\beta) \to A^R_k(I_\mu(X\sslash G))
\end{equation}
where the groups $A_*^R$ are defined in Section \ref{sec:equivariant}. If $U$ is an open subset of a representation of $V$ where $R$ acts freely, such that the complement of $U$ in $V$ has codimension greater than $\dim(X)-k$, then by Lemmas \ref{lem:equivariant-chow} and \ref{lem:almost-trivial} we have
\begin{align}
A_k([F_\beta/({\CC^*_\lambda\times R})]) \simeq A_k^{\CC^*_\lambda\times R}(F_\beta) &\simeq A^{\CC^*_\lambda}_{k+\dim V - \dim R}(F_\beta \times_R U)\label{eq:R3}\\
&\simeq (A_*([F_\beta \times_R U])[z,z^{-1}])_{k+\dim V - \dim R}\notag
\end{align}
where the subscript on the rightmost term indicates the corresponding summand of the graded $\QQ$-vector space. We define \eqref{eq:R1} by first applying the isomorphism \eqref{eq:R3} and then applying \eqref{eq:R2} to the coefficients of $z^i$.
\end{remark}
\subsubsection{$I$-function definition}\label{sec:ifunc-def}

 Let $\iota$ be the involution of $I_\mu(X\sslash G)$ induced by the inversion automorphism of $B\bmu_a$ (see \cite[Sec~3.5]{AGV08}). Finally let $\varpi: I_\mu(X\sslash G) \rightarrow \overline{I}_\mu(X\sslash G)$ be the rigidification and define
\[
\widetilde{ev}_{\star}: F_{\beta} \rightarrow \overline{I}_\mu(X\sslash G) \quad \quad \quad \widetilde{ev}_\star = a\circ\varpi\circ \iota \circ ev_\star
\]
where $a$ as a function is multiplication by the integer $a$ (which is determined by $\beta$).
Now we can define the $I$-function of $X\sslash G$ as a formal power series in the $q$-adic completion of the semigroup ring generated by the effective classes. 
\begin{definition}\label{def:ifunc}
The (unrigidified) \textit{small $I$-function} of $(X,G,\theta)$ is
\begin{equation}\label{eq:Ifunc}
I^{X\sslash G}(z) = \one_X+ \sum_{\beta\neq 0}q^\beta I^{X\sslash G}_\beta(z) \quad \quad \text{where} \quad \quad I^{X\sslash G}_\beta(z) = (\iota \circ ev_\star)_*\Res^{X\sslash G}_\beta. 
\end{equation}
If $F_\beta$ is empty then $I^{X\sslash G}_\beta$ is defined to be zero.
The coefficient $I^{X\sslash G}_\beta(z)$ is an element of $A_*(I_\mu(X\sslash G))[z, z^{-1}]$.
The rigidified small $I$-function $\overline{I}^{X\sslash G}(z)$ is defined analogously with $\widetilde{ev}_\star$ in place of $\iota \circ ev_\star$.
\end{definition}
\begin{remark}\label{rmk:ifuncs}
Directly from the definitions, we have $\mathbf{a}\varpi_*I^{X\sslash G}(z) = \overline{I}^{X\sslash G}(z),$ where $\overline{I}^{X\sslash G}(z)$ agrees with the series notated $I(0,q,z)$ in \cite{orb-qmaps} and $\mathbf{a}$ denotes the locally constant function equal to multiplication by $a$ on $\overline{I}_{\mu_a}(X\sslash G).$ 
Since $\mathbf{a}\varpi^*\varpi_*$ is equal to the identity, it is equivalent to write $I^{X\sslash G}(z) = \varpi^*\overline{I}^{X\sslash G}(z)$. Hence, the series \cite[(1.6)]{yang} is equal to $\mathbf{a}I^{X\sslash G}(z)$. Multiplying both sides of the equality in \cite[Thm~1.12.2]{yang} by $\mathbf{a}$ shows that $I^{X\sslash G}(z)$ lies on the Lagrangian cone of $X\sslash G$. See also \cite[App~A.1]{SW1}.

 \end{remark}

\begin{remark}\label{rmk:def-equivariant}
When $R$ is not trivial, we define the equivariant small $I$-function $I^{X\sslash G, R}{z}$ via the same formula as \eqref{eq:Ifunc}, replacing $e_{\CC^*_\lambda}$ with $e_{\CC^*_\lambda \times R}$ and replacing $(\iota \circ ev_\star)_*$ with the equivariant pushforward. Since the equality in \cite[Thm~1.12.2]{yang} also holds $R$-equivariantly by essentially the same proof \cite{yang-email}, we see as in Remark \ref{rmk:ifuncs} that $I^{X\sslash G, R}(z)$ is on the $R$-equivariant Lagrangian cone of $X\sslash G$.
\end{remark}

\section{Abelianization for $I$-functions}

Our first goal is to prove the following.

\begin{proposition}\label{prop:diagram}
Let $\beta \in \Hom(\Pic^G(X),\QQ)$ be $I$-effective. For every $\tilde \beta \in \rpic^{-1}(\beta)$, setting $\tilde \alpha = \bdtilde(\tilde \beta)$, there is
\begin{itemize}
\item an open substack $F^0_{\tilde \beta}(X\sslash T)$ of $F_{\tilde \beta}(X\sslash T)$
\item a group element $g_{\tilde \beta} \in G$ and a parabolic subgroup $P_{\tilde \alpha}$ of $Z_G(g_{\tilde \beta})$, and
\item a morphism $\psi_{\tilde \beta}: F^0_{\tilde \beta}(X\sslash T) \rightarrow F_\beta(X\sslash G)$ whose image we denote $F_{\tilde \beta}(X\sslash G)$,
\end{itemize}
fitting into the following commutative diagram with fibered squares:
\begin{equation}\label{eq:big_diagram}
\begin{tikzcd}
 F_{\tilde \beta}(X\sslash G) \arrow[d,hook,"i"]  & F^0_{\tilde \beta}(X\sslash T) \arrow[l,"\psi_{\tilde \beta}"'] \arrow[d,hook]\arrow[r,hook,"h"]& F_{\tilde \beta}(X\sslash T)\arrow[d,hook, "ev_\star"]\\
 X^{g_{\tilde \beta}}\sslash_G P_{\tilde \alpha}  \arrow[dr, "f"]& X^{g_{\tilde \beta}} \sslash_G T \arrow[r,hook,"j"]\arrow[d, "\varphi_{g_{\tilde \beta}}"]\arrow[l,"p"]&  X^{g_{\tilde \beta}} \sslash T\\
& X^{g_{\tilde \beta}} \sslash_G Z_G(g_{\tilde \beta})
\end{tikzcd}
\end{equation}
where the map $\varphi_{g_{\tilde \beta}}$ was defined in \eqref{eq:inertia4}. Moreover, the vertical arrows in the top row are all closed embeddings, and the composition $\eta_G \circ f\circ i$ is the evaluation map $ev_\star$.
\end{proposition}

 Note that the stacks $X^{g_{\tilde \beta}}\sslash_G Z_G(g^{\tilde \beta})$, $X^{g_{\tilde \beta}}\sslash_G T$, and $X^{g_{\tilde \beta}}\sslash T$ are (isomorphic to) open and closed substacks of the inertia stacks $I_\mu(Z\sslash G)$, $I_{\mu}(Z\sslash_G T)$, and $I_{\mu}(Z\sslash T)$, respectively (see Section \ref{sec:inertia1}).

\begin{remark}
If $R$ is not trivial, the diagram \eqref{eq:big_diagram} is $R$-equivariant. For most of the diagram this follows from the definitions in the preceeding section. The map $\psi_{\tilde \beta}$ is also clearly equivariant from its definition in \eqref{eq:defpsi}.
\end{remark}
\subsection{Preliminaries}\label{sec:defs}

Let $(\wP{a})_S \rightarrow [X/T]$ be a family of quasimaps of class $\tilde \beta$, and let $\tilde \alpha = \bdtilde(\tilde \beta)$. By \cite[Lem~4.6]{orb-qmaps} the subscript $a$ is the minimal postive integer such that $a  \tilde \alpha$, a priori a morphism to $\QQ$, is in $ \Hom(\chi(T), \ZZ)$. Let $\intalpha = a\tilde \alpha$ and define $\tau_{\intalpha}$ to be the cocharacter of $T$ associated to $\intalpha$ via the rule
\begin{equation}
\xi(\tau_{\intalpha}(t)) = t^{- \intalpha(\xi)} \quad\quad \text{for any}\quad \quad\xi \in \chi(T).
\end{equation} We write $\T_{\intalpha} := \T_{\tau_{\intalpha}}$ and $\P_{\intalpha} := \P_{\tau_{\intalpha}}$ for the associated principal bundles on $(\wP{a})_S$. We define
\[
g_{\tilde \beta} := \tau_{{\intalpha}}(e^{2 \pi i/a})^{-1}.
\]

\begin{remark}\label{rmk:k-quasimaps}
Let $k$ be an algebraically closed field. It follows from \cite[Thms~2.4, 2.7]{MT} and Remark \ref{rmk:coord2} that every $k$-point of $QG_\beta$ can be represented by a quasimap of the form $((\wP{a})_k, \P_{\intalpha}, \sigma)$, where $\intalpha \in \Hom(\chi(T), \ZZ)$ has the property that $\rchi(\intalpha/a) = \bd(\beta)$. The class $\intalpha$ is unique up to the action of $W$.
\end{remark}

The parabolic subgroup $P_{\tilde \alpha}$ of $Z_G(g_{\tilde \beta})$ and its canonical Levi subgroup $L_{\tilde \alpha}$ are defined using the dynamic method as in \cite[Sec~4.1]{nonab1}. For any $\CC$-scheme $S'$, we have
\[
P_{\tilde \alpha}(S') = \{g \in Z_G(g_{\tilde \beta})(S') \; | \; \tau_{\intalpha}^{-1}g\tau_{\intalpha}:\torus_{S'}\rightarrow Z_G(g_{\tilde \beta})\;\text{extends to a morphism}\;\AA^1_{S'}\rightarrow Z_G(g_{\tilde \beta})\}
\]
and $L_{\tilde \alpha} \subset P_{\tilde \alpha}$ is the centrilizer of the image of $\tau_{\intalpha}$. The extensions mentioned in the definition of $P_{\tilde \alpha}$ are unique if they exist.

\begin{lemma}\label{lem:auts}
Let $S$ be a $\CC$-scheme. The group $P_{\tilde \alpha}$ has a natural inclusion into $\Aut(G \times_T \T_{\intalpha})$, where $\T_{\intalpha}$ is a principal bundle on $(\wP{a})_S$. The embedding sends $g \in P_{\tilde \alpha}$ to the automorphism defined as in \eqref{eq:coordinateiso} by setting $\phi_V(v)=g$.
\end{lemma}
\begin{proof}
By definition, the morphism $\tau_{\intalpha}^{-1}(t)g\tau_{\intalpha}(t): \torus_S \rightarrow Z_G(g_{\tilde \beta})$ has a unique extension to $\AA^1_S$. Call this extension $\phi_U'$. Since $g \in Z_G(g_{\tilde \beta})$, the morphism $\phi_U'$ is invariant under the action of $\bmu_a$ on $\AA^1_S$, hence factors through a morphism $\phi_U:U_S \rightarrow Z_G(g_{\tilde \beta})$. 
\end{proof}

The remainder of this subsection is dedicated to the proof of the following lemma.
\begin{lemma}\label{lem:fixed-by-a}
If a quasimap $((\wP{a})_S, \P, \sigma)$ is in $F_\beta$, then it is fixed by the $a^{th}$ power of the action \eqref{eq:action1}.
\end{lemma}

\begin{proof}
We first show that if $q: \wP{a}\times S \rightarrow [X/G]$ is fixed by an $n^{th}$ power of the action \eqref{eq:action1}, then $a|n$. It suffices to show this for $S = \Spec(\bar k)$ where $\bar k$ is an algebraically closed field. We use $(\cdot, \cdot)$ to denote the gcd of two integers; let $d = (a, n)$. Furthermore write $d = d_a d_n$ with $d_a, d_n \in \mathbb{N}$ such that 
\begin{equation}\label{eq:fixed-by-a1}(d_a, d_n) = (d_a, n/d) = (d_n, a/d) = 1
\end{equation}(such a factorization is not unique).

If $q$ is invariant under the $n^{th}$-power of the action \eqref{eq:action1}, then it factors through the quotient $\pi: \wP{a} \rightarrow \fX_n$, where
\[
\wP{a} = [(\CC^2\setminus\{0\} )\big/_{(a \; 1)} \CC^*] \quad \text{and} \quad \fX_n = [(\CC^2\setminus\{0\}) \big/_{\left(\begin{smallmatrix}
a&1\\
n&0\\
\end{smallmatrix}\right)} (\CC^*)^2 ]
\]
and the subscripts indicate the weights (charge matrix) for each quotient. The group $\CC^* \times \bmu_d$ embeds into the isotropy group of $[1:0]$ in $\fX_n$ via the map
\begin{align}\label{eq:fixed-by-a2}
\CC^* \times \bmu_d \simeq \CC^* \times \bmu_{d_a} \times \bmu_{d_n} &\longrightarrow \CC^* \times \CC^*\\
(s, \xi_1, \xi_2) &\mapsto (\xi_2s^{n/d}, \xi_1s^{-a/d}).
\end{align}
The image of this embedding contains the subgroup $\bmu_a \times 1$ (the image under $\pi$ of the isotropy group at $[1:0]$ in $\wP{a}$). To see this, write $\mu_a \simeq \mu_{d_n} \times \mu_{a/d_n}$ using \eqref{eq:fixed-by-a1}. If $\mu \in \bmu_a$ equals $(\mu_1, \mu_2)$ under this factorization, then $(\mu, 1)$ is equal to the image of $(\sqrt[n/d]{\mu_2}, (\sqrt[n/d]{\mu_2})^{a/d}, \mu_1)$ under the embedding \eqref{eq:fixed-by-a2}. Here $\sqrt[n/d]{\mu_2}$ denotes the unique element of $\bmu_{a/d_n}$ whose $n/d^{th}$ power is $\mu_2$.  

We investigate the map induced by $q$ on the isotropy group $\bmu_a$ at $[1:0]$ in $\wP{a}$. Since $q$ is representable and sends $[1:0]$ to the Deligne-Mumford locus of $[X/G]$, this is an embedding $\bmu_a \rightarrow H$ for some finite group $H \subset G$. But the factorization through $\pi$ implies that $\bmu_a \rightarrow H$ factors through $\CC^* \times \bmu_d$, so we must have an embedding $\bmu_a \rightarrow \bmu_d$. This implies $a|d$, so $d=a$ and $a|n$ as desired.

For the second (last) part of the proof, we let $S$ be arbitrary and we assume that $q$ is fixed by the $\ell a^{th}$ power of the action \eqref{eq:action1}. By assumption $q$ factors through a morphism $q': \fX_{\ell a}\times S \rightarrow [X/G]$; we will show that $q'$ factors through the rigidification
\[
\fX_{\ell a}\times S \rightarrow \fX_a \times S
\]
along the subgroupscheme $1 \times \bmu_\ell \subset (\CC^*)^2$. Let $\fX = \fX_{\ell a}\times S$. We use \cite[Thm~5.1.5(2)]{ACV}: for a test scheme $R$, we say that $\xi \in \fX(R)$ has property $P$ if the map $\Aut_R(\xi) \rightarrow \Aut_R(q'(\xi))$ induced by $q'$ contains $1\times \bmu_\ell$ in its kernel. This defines a presheaf on $\fX$ which we must show is represented by $\fX$ itself. 

First we show that the sheaf of objects with property $P$ is represented by a closed substack of $\fX$. Let $I([X/G])$ be the inertia stack of $[X/G]$ and define $\fY$ to be the fiber product of the diagram
\[
\begin{tikzcd}
\fY \arrow[r] \arrow[d] & q'^*(I([X/G])) \arrow[d, "\Delta"] \\
(\bmu_{\ell})_\fX \arrow[r] \arrow[r, "{(q', e)}"] & q'^*(I([X/G]))\times_{\fX} q'^*(I([X/G]))
\end{tikzcd}
\]
where $\Delta$ is the diagonal and $e$ is the projection to $\fX$ followed by the identity section of $q'^*(I([X/G])).$ Since $I([X/G]) \rightarrow [X/G]$ is separated and representable, $\Delta$ and hence $\fY \rightarrow (\bmu_\ell)_\fX$ are closed embeddings. But $(\bmu_\ell)_\fX= \bmu_\ell \times \fX$ is a disjoint union of $\ell$ copies of $\fX$, so this embedding defines $\ell$ closed substacks of $\fX$. The intersection of these $\ell$ substacks is a closed substack that represents our presheaf.

Next we check that the substack of $\fX$ of objects with property $P$ is supported on every connected component of $\fX$. For this, let $x \in \fX(\bar k)$ be a geometric point  whose image in $(\fX_{\ell a})(\bar k)$ is $[1:0]$. If $n:=\ell a$ then we can set $d_a=1$ and $d_n=d=a$ in \eqref{eq:fixed-by-a1}, and \eqref{eq:fixed-by-a2} becomes an embedding
\begin{align*}
\CC^* \times \bmu_a &\rightarrow \CC^* \times \CC^*\\
(s, \xi) &\mapsto (\xi s^\ell, s^{-1})
\end{align*}
into the isotropy group of $x$. Observe that the image of this embedding contains $(1, \mu)$ for $\mu \in \bmu_\ell$ by setting $s = \mu^{-1}$ and $\xi=1$---in other words, $1 \times \bmu_\ell$ is a subgroup of $\CC^* \times 1 \subset \CC^* \times \bmu_a$ in $\Aut_{\bar k}(x)$. Since $q'$ maps $x$ to the Deligne-Mumford locus, the subgroup $\CC^* \times 1$ must be in the kernel of $\Aut_{\bar{k}}(x) \rightarrow \Aut_{\bar k}(q'(x))$, hence $1 \times \bmu_\ell$ is in the kernel.

By Lemma \ref{lem:open} below and a standard Noetherian induction argument, the substack of objects with property $P$ is open. This completes the proof.

\end{proof}

The proof of the following lemma was explained to me by Martin Olsson.
\begin{lemma}\label{lem:open}
Let $f: X \rightarrow Y$ be a morphism of group schemes over a Noetherian scheme $R$ with $X = R \times G$ for a finite group $G$. If $r \in R$ is a closed point such that the fiber $f_r: X_r \rightarrow Y_r$ is the trivial group homomorphism, then there is a Zariski-open subset $U \subset R$ such that $f|_U$ is also trivial. 
\end{lemma}
\begin{proof}
We claim that the restriction of $f$ to the completed local ring $\Spec(\hat\OO_{R, r})$ is trivial. Granting this, since $\OO_{R, r}$ embeds in its completion by \cite[Tag~00IP]{tag}, we see that the restriction of $f$ to $\Spec(\OO_{R, r})$ is also trivial, and hence $f$ is trivial in a neighborhood of $r$ as claimed.

To prove the claim, let $\fm$ be the maximal ideal of $\OO_{R, r}$. Since $\hat \OO_{R, r} = \displaystyle\lim_{\longleftarrow} \OO_{R, r}/\fm^i$ it suffices to show that the restriction of $f$ to each $\Spec(\OO_{R, r}/\fm^i)$ is trivial, and since $X \rightarrow R$ is a product it suffices to check global sections. By assumption this is true when $i=1$. In general there is a commuting diagram of groups
\[
\begin{tikzcd}
X(\Spec(\OO_{R, r}/\fm^i)) \arrow[d] \arrow[r, "f_i"] & Y(\Spec(\OO_{R, r}/\fm^i)) \arrow[d, "\pi"] \\
X(\Spec(\OO_{R, r}/\fm^{i-1})) \arrow[r, "f_{i-1}"] & Y(\Spec(\OO_{R, r}/\fm^{i-1}))
\end{tikzcd}
\]
where $f_i$ and $f_{i-1}$ are restrictions of $f$. If by induction we assume $f_{i-1}$ is the trivial homomorphism, then $f_i$ factors through the kernel of $\pi$. This kernel is naturally a module for the vector space $\OO_{R, r}/\fm$, and hence as a group it has no nontrivial elements of finite order. Since $X \rightarrow R$ is finite, $f_i$ must also be trivial.
\end{proof}

\subsection{The abelian case}
The goal of this section is to describe the moduli spaces $F_{\tilde \beta}(X\sslash T)$ and their universal families; however, we begin with two results that are not special to the abelian setting.

\begin{lemma}\label{lem:extend}
Let $\cX$ and $\cY$ be algebraic stacks over a scheme $\mathcal{B}$ such that $\cY\to \mathcal{B}$ has finite diagonal. Any two morphisms $f,g:\cX \times \AA^1_{\CC}\rightarrow \cY$ over $\mathcal{B}$ that agree on $\cX \times (\AA^1_{\CC}\setminus\{0\})$ (via a given 2-morphism) agree on all of $\cX \times \AA^1_{\CC}$. Moreover, if $f$ and $g$ are representable then the 2-morphism between the morphisms $f,g:\cX \times \AA^1_{\CC}\rightarrow \cY$ restricts to the given 2-morphism between the restrictions $f,g:\cX \times (\AA^1_{\CC}\setminus\{0\}) \to \cY.$
\end{lemma}
\begin{proof}
Let $f,g:\cX \times \AA^1_{\CC}\rightarrow \cY$ be the given morphisms. The assumptions imply that we have a commuting diagram as below with the square fibered.
\[
\begin{tikzcd}
& Eq \arrow[r] \arrow[d]&\cY \arrow[d, "\Delta"]\\
\cX \times (\AA^1_{\CC}\setminus\{0\}) \arrow[ur] \arrow[r, hook] & \cX \times \AA^1_{\CC} \arrow[r, "{(f,g)}"] & \cY \times_\mathcal{B} \cY
\end{tikzcd}
\]
Let $\cZ \subset Eq$ be the scheme-theoretic image of $\cX \times (\AA^1_{\CC}\setminus\{0\}) \to Eq$, and let $h:\cZ \to \cX\times \AA^1_{\CC}$ be the projection (observe that $h$ is finite). To find a 2-morphism between $f$ and $g$, it is enough to show that $h$ is an isomorphism.

Since formation of $\cZ$ commutes with flat pullback (using \cite[Tags~0CMK, 050Y]{tag}), we may assume $\cX$ is an affine scheme $\Spec(A)$. In this case (since $h$ is finite) we know $\cZ=\Spec(B)$ for $B$ a finite $A[t]$-module. We have a commuting diagram
\[
\begin{tikzcd}
& B\arrow[dl, hook] \\
A[t, t^{-1}]  \arrow[r, hookleftarrow]&  A[t] \arrow[u, "h^\#"']
\end{tikzcd}
\]
where $B \to A[t, t^{-1}]$ is injective because it corresponds to a scheme-theoretic closure. Hence the ring map $h^\#$ determined by $h$ is also injective. To see that it is surjectve, suppose for contradiction that $B$ contains an element of $A[t, t^{-1}]$ of negative degree. Then (taking powers of this Laurent polynomial) $B$ has elements of arbitrarily negative degrees. These cannot be generated over $A[t]$ by any finite subset of $A[t, t^{-1}]$. This completes the proof that $h$ is an isomorphism.

To see that the 2-morphism extends, observe that the 2-morphism between $f$ and $g$ (or their restrictions to $\AA^1_{\CC} \setminus \{0\}$) is part of the data of the induced map to $Eq$. So our problem amounts to showing that the diagram
\[
\begin{tikzcd}
& \cZ \\
\cX \times (\AA^1_{\CC}\setminus\{0\}) \arrow[ur]\arrow[r, hook]& \cX \times \AA^1_{\CC} \arrow[u, "h^{-1}"']
\end{tikzcd}
\]
strictly commutes. This is true when $f$ and $g$ are representable because the diagonal $(f,g)$ is also representable, so both $\cZ$ and $\cX \times (\AA^1_{\CC}\setminus\{0\})$ are representable over $\cY \times \cY$, and the category $\Hom_{\cY\times \cY}(\cX\times (\AA^1_{\CC}\setminus\{0\}, \cZ)$ is equivalent to a set.
\end{proof}

 The following lemma lets us determine quasimaps from just some of their data. To unpack the statement, it is helpful to note that for a principal $G$-bundle $\P$ on a Deligne-Mumford stack $\cB$, the category of sections of the representable morphism $\P \times_G X \to \cB$ is equivalent to a set. That is, if an 2-morphism exists between two sections, it is unique.

\begin{lemma}\label{lem:factors}
Let $q=((\wP{a})_S, \P, \sigma)$ be a quasimap to $[X/G]$.
\begin{enumerate}
\item The section $\sigma: (\wP{a})_S \to \P \times_G X$ is completely determined by its restriction to $[V/\bmu_a] \times S$. In particular,
if $\P = \P_\tau$ for some $\tau$, then $q$ is completely determined by the data $(\tau, \sigma_V)$.\label{lem:uniquemap}
\item If moreover $q$ is fixed, then $\sigma$ is completely determined by its restriction to $B\bmu_a \times S$. In fact, the restriction $q|_{[V/\bmu_a]\times S} \to [X/G]$ factors through the restriction $ev_\star(q): B\bmu_a \times S \to [X/G]$.
\end{enumerate}
\end{lemma}
\begin{proof}
To prove \eqref{lem:uniquemap}, let $[\Omega /\bmu_a] \subset \wP{a}$ denote the complement of $[0:1]$ and $[1:0]$ and note that $[\Omega /\bmu_a] \simeq \AA^1_{\CC}\setminus\{0\}$. Now apply Lemma \ref{lem:extend} with $\cX=\cB=S$ and $\cY = (\P\times_G X)|_{[\Omega/\bmu_a]\times S}:$ if two sections of $\P \times G X$ agree on $\AA^1_{\CC}\setminus\{0\} \simeq [\Omega/\bmu_a] \subset[V/\bmu_a]$, then they agree on $\AA^1_{\CC} \simeq U $.

To prove (2), note that by Lemma \ref{lem:fixed-by-a}, the quasimap $q$ defines a $\CC^*_{\lambda^{1/a}}$-equivariant morphism $[V/\bmu_a]\times S \xrightarrow{[1:v]} (\wP{a})_S \xrightarrow{q} [X/G]$ which we will call $q_V$. Define maps $pr_{23}, m: \CC^*_{\lambda^{1/a}} \times [V/\bmu_a] \times S \to [V/\bmu_a] \times S$ where $pr_{23}$ is given by projection and $m$ is the $\CC^*_{\lambda^{1/a}}$-action. The fact that $q_V$ is equivariant means that there is a natural transformation between the two maps $q_V \circ pr_{23}, q_V \circ m: \CC^*_{\lambda^{1/a}} \times [V/\bmu_a] \times S \to [X/G].$
The map $m$ extends to a morphism $m: \AA^1_{\lambda^{1/a}} \times [V/\bmu_a]$ given on $V$ by the rule $\lambda^{1/a} \cdot v = (\lambda^{1/a})^{-1}v$ (it may help to observe that $[(\lambda^{1/a})^au: v]= [u:(\lambda^{1/a})^{-1}v]$ in $\wP{a}$). So we again have two morphisms $q_V \circ pr_{23}, q_V \circ m: \AA^1_{\lambda^{1/a}} \times [V/\bmu_a] \times S \to [X/G],$ and moreover these factor through the separated substack $X\sslash G \subset [X/G]$. By Lemma \ref{lem:extend} we have $q_V \circ pr_{23} = q_V \circ m$ on the larger domain. Restricting to the fiber $\lambda=0$ we get $q_V=q_V(0)$ as desired.

\end{proof}

We can now begin to analyze $F_{\tilde \beta}(X\sslash T)$.
\begin{lemma}\label{lem:abelian}
The map $ev_\star: F_{\tilde \beta}(X\sslash T) \rightarrow I_\mu(X\sslash T)$ is a closed embedding, and it factors through the component  $(I_{\mu}(X\sslash T))_{g_{\tilde \beta}}.$
\end{lemma}

\begin{proof}
First we show that $ev_\star$ factors through $(I_{\mu}(X\sslash T))_{g_{\tilde \beta}}$. This may be checked at geometric points of $F_{\tilde \beta}{(X\sslash T)}$. Let $k$ be an algebraically closed field. By Remark \ref{rmk:k-quasimaps}, a $k$-quasimap $q: (\wP{a})_k \rightarrow [X/T]$ of degree $\tilde \beta$ is given by data $(\wP{a}, \T_{\intalpha}, \sigma)$. Now $ev_{\star}(q)$ is a $k$-point of $[X/T]$, and since $T$ is abelian, there is a canonical identification of $\Aut_k(ev_{\star}(q))$ with a subgroup of $T$.  It follows from \eqref{eq:global} that the composition $\Aut_k(\star)\rightarrow \Aut_k(ev_{\star}(q))\hookrightarrow T$ is precisely equal to the restriction of $\tau_{\intalpha}^{-1}$ to $\bmu_a$,
 so it follows from the definitions that this quasimap lands in the $g_{\tilde \beta}$-component of $I_{\mu}(X\sslash T)$.

To show that $ev_\star$ is a closed embedding, the argument of \cite[Lem~4.2.1]{nonab1} works with the following adjustments.

The proof of \cite[Lem~4.2.2]{nonab1} shows that its statement holds when $S$ and $S'$ are Deligne-Mumford stacks and the principal bundles are algebraic spaces (so that the corresponding maps to $[\bullet/G]$ are representable). Moreover, the analog of \cite[Lem~4.2.3]{nonab1} holds when $X,Y$ are Deligne-Mumford stacks if one also assumes that $\pi$ induces surjections of automorphism groups at $\CC$-points (see \cite[Lem~3.6.2]{dissertation}).

We use Lemma \ref{lem:factors} in place of \cite[Lem~4.2.4]{nonab1}.

We replace \cite[Lem~4.2.5]{nonab1} with the following. Note that, in conjunction with Remark \ref{rmk:k-quasimaps}, this lemma implies that $ev_\star$ induces bijections of automorphism groups at geometric points.
\begin{lemma}\label{lem:abelian-main-2}
If $q_i = ((\wP{a})_S, \T, \sigma_i)$ are fixed quasimaps and $f$ is an arrow in $I_\mu(X\sslash T)(S)$ from $ev_\star(q_1)$ to $ev_\star(q_2)$, then there is a unique isomorphism of $q_1$ and $q_2$ that maps to $f$ under $ev_\star$.
\end{lemma}

\begin{proof}

Recall that an automorphism of a principal $T$-bundle on a scheme $\cB$ is given by a morphism $\cB \to T$. Hence, the identification $f$ is given by an automorphism of $\T|_{(B\bmu_a)_S}$, or equivalently a morphism $\phi:(B\bmu_a)_S \rightarrow T$. The latter morphism must factor through the coarse moduli space, and hence is equal to the pullback of some morphism $S \rightarrow T$. This induces a morphism $\Phi:(\wP{a})_S \rightarrow T$ by pullback, and hence an automorphism of $\T$ that restricts to the automorphism of $\T|_{(B\bmu_a)_ S}$ determined by $f$.
Now $\Phi$ defines an automorphism of $\T\times_T X$, which we also denote $\Phi$, and by construction the sections $\sigma_2$ and $\Phi \circ \sigma_1$ agree on $B\bmu_a \times S$. By Lemma \ref{lem:factors} they agree everywhere.

To see that $\Phi$ is the unique morphism extending $\phi$, note that $\Phi|_{[V/\bmu_a]\times S}$ factors through the finite subscheme of $T$ consisting of group elements with nontrivial isotropy on $X^s(T)$ (see Lemma \ref{lem:finite}). Since $[V/\bmu_a]$ is connected, $\Phi|_{[V/\bmu_a]\times S}$ is determined by $\phi$. Now $\Phi$ is determined by $\Phi|_{[V/\bmu_a]\times S}$ by Lemma \ref{lem:extend}.
\end{proof}

It follows from \cite[Tag~04YY]{tag} and Lemma \ref{lem:abelian-main-2} that $ev_\star$ is representable. The proof of \cite[Lem~4.2.1]{nonab1} now applies to show that $ev_\star$ is a closed embedding.

\end{proof}

By Lemma \ref{lem:abelian} we have a closed embedding $ev_\star: F_{\tilde \beta}(X\sslash T) \rightarrow I_{\mu}(X\sslash T)$. Define
\[
X_{\tilde \beta} = \cover{T}{X^s(T)}\times_{I_{\mu}(X\sslash T)} F_{\tilde\beta}(X\sslash T),
\]
where the cover $S_{X^s(T)}(T)\rightarrow I_{\mu}(X\sslash T)$ was defined in \eqref{eq:inertia-square} (see also Remark \ref{rmk:inertia1}).
By Lemma \ref{lem:abelian}, the composition
\[
X_{\tilde \beta} \rightarrow S_{X^s(T)}(T) \hookrightarrow T\times X^s(T) \rightarrow X^s(T)
\] is a closed embedding into $X^{g_{\tilde \beta}}\cap X^s(T)$. From the definitions, $\tau_{\intalpha}(\bmu_a)$ acts trivially on $X^{g_{\tilde \beta}}$; hence, there is an action of $T/\tau_{\intalpha}(\bmu_a)$ on $X^{g_{\tilde \beta}}$ (the rigidification). We restrict the action of $T/\tau_{\intalpha}(\bmu_a)$ to the quotient of the subgroup of $T$ defined by the cocharacter $\tau_{\intalpha}$. That is, we define an action $\cdot_{rig}: \CC^* \times X^{g_{\tilde \beta}} \rightarrow X^{g_{\tilde \beta}}$ by identifying the group $\CC^*$ with $\CC^*/\bmu_a \subset T/\tau_{\intalpha}(\bmu_a)$ via $\tau_{\intalpha}$. This action is characterized by the fact that the following diagram commutes, where we have labeled the arrows with the images of $(\omega, z) \in \CC^* \times X^{g_{\tilde \beta}}$.
\begin{equation}\label{eq:rig}
\begin{tikzcd}[column sep=5em]
& \CC^* \times X^{g_{\tilde \beta}} \arrow[d, "{(\omega^a, z)}"'] \arrow[r, "\tau_{\intalpha}(\omega)\cdot z"]& X^{g_{\tilde \beta}} \\
&\CC^* \times X^{g_{\tilde \beta}} \arrow[ur, "\omega\cdot_{rig}z"']
\end{tikzcd}
\end{equation}
\begin{proposition}\label{prop:ufam}
The universal family on $F_{\tilde\beta}(X\sslash T)$ has fiber bundle $\cZ$ on $\wP{a}\times{F_{\tilde\beta}(X\sslash T)}$ and section $\csigma$ defined as follows:
\begin{equation}\label{eq:ufam1}
\cZ = \frac{X_{\tilde\beta} \times \CC^2 \times X}{(x, u,v, y) \sim (tx, s^au, sv, \tau_{\intalpha}(s)^{-1}ty)} \quad \quad \csigma(x,u,v) = (x, u,v, u^{-1}\cdot_{rig}x)
\end{equation}
where  $(t, s) \in T \times \CC^*$, and the formula for $\csigma$ holds on the open locus where $u\neq 0$.
\end{proposition}

\begin{remark}
Proposition \ref{prop:ufam} only gives a formula for $\csigma$ on $[V/\bmu_a] \times F_{\tilde \beta}(X\sslash T) $, but by Lemma \ref{lem:extend} this formula uniquely determines $\csigma$. We do not know a formula for $\csigma$ on all of its domain.
\end{remark}
We define the \textit{tautological family} on $X_{\tilde \beta}$ to be the pullback of the universal family along the map $X_{\tilde \beta} \rightarrow F_{\tilde \beta}(X\sslash T)$.
\begin{remark}\label{rmk:tfam}
Proposition \ref{prop:ufam} is equivalent to the statement that the tautological family is given by the formulae \eqref{eq:ufam1} after setting $t=1$ (i.e., before dividing by the $T$-action). In other words, by Lemma \ref{lem:factors}.\eqref{lem:uniquemap}, to prove Proposition \ref{prop:ufam} it is equivalent to prove that the tautological family is given by the principal bundle $\T_{\intalpha}$ (constant across fibers of $(\wP{a})_{X_{\tilde \beta}}$) and a section $\csigma$ with $\csigma_V(v,x)=v$ for $(v,x) \in V_{X_{\tilde \beta}}$. One can check directly using commutativity of \eqref{eq:rig} that the formula for $\csigma$ given in \eqref{eq:ufam1} is the unique section of $\cZ$ with the required $\csigma_V$.
\end{remark}
\begin{proof}[Proof of Proposition \ref{prop:ufam}]
 The proof of \cite[Prop~4.2.6]{nonab1} works with minor modifications as follows. Let $((\wP{a})_{X_{\tilde \beta}}, \T, \csigma)$ be the tautological family. Let $\tilde \star\times X_{\tilde \beta}$ be the subscheme of $V_{X_{\tilde \beta}}$ equal to the inverse image of $B\bmu_a \times X_{\tilde \beta}  = \star\times X_{\tilde \beta} \in (\wP{a})_{X_{\tilde \beta}}$. 
Then $\T|_{\tilde \star\times X_{\tilde \beta}}$ is isomorphic to 
$X_{\tilde \beta}\times_{F_{\tilde \beta}}X_{\tilde \beta}$ with the structure map to $X_{\tilde \beta}$ equal to the first projection. So $\T|_{\tilde \star\times X_{\tilde \beta}}$ is trivial (it has the diagonal section), and it follows from Lemma \ref{lem:factors} that 
$\T|_{V\times X_{\tilde \beta}}$ is trivial.

Recall the map $\kappa_U: \torus\times X_{\tilde \beta} \rightarrow (U\setminus 0) \times X_{\tilde \beta}$ given by taking the $a^{th}$ power. We claim that $\kappa_U^*$ is injective on Picard groups; in fact, we claim that if $\cF \in \Pic(\CC^* \times X_{\tilde \beta})$, then canonical morphism $\cF \rightarrow (\kappa_{U, *}\kappa_U^*\cF)^{\bmu_a}$ is an isomorphism. This may be checked on an affine scheme $\Spec(R) \rightarrow \CC^*\times X_{\tilde \beta}$ where $\F$ is trivial as follows. The morphism $\kappa$ restricts to a $\bmu_a$-torsor $\Spec(S) \rightarrow \Spec(R)$; let $\alpha: S \rightarrow S \times \bmu_a$ denote the ring map inducing the action. Consider the following diagram:
\[
\begin{tikzcd}
R \arrow[r] & S \arrow[r, shift right=1] \arrow[r, shift left=1] \arrow[d, equal]& S \otimes_R S \arrow[d]\\
S^{\bmu_a} \arrow[r] & S \arrow[r, shift right=1, "{(id, \alpha)}"'] \arrow[r, shift left=1, "{(id, 1)}"] & S\times \bmu_a
\end{tikzcd}
\]
Faithfully flat descent implies that the top row is an equalizer diagram, while the torsor property implies that the right vertical arrow is an isomorphism. Since the equalizer of the bottom maps $S \rightrightarrows S \times \bmu_a$ can be explicitly computed to equal the ring of invariants $S^{\bmu_a}$, the canonical map of equalizers $R \rightarrow S^{\bmu_a}$ must be an isomorphism.

Since $\kappa_U^*$ is injective and $\T|_{\CC^*\times X_{\tilde \beta}}$ is trivial, it follows that
$\T|_{(U\setminus\{0\})\times{X_{\tilde \beta}}}$ is trivial.
Now the argument in \cite[Prop~4.2.6]{nonab1} shows that $\T \simeq \T_{\intalpha}$ and $\csigma_V(v,x)=x$ for $(v,x) \in  V\times X_{\tilde \beta}$. 
We conclude using Remark \ref{rmk:tfam}.
\end{proof}
\subsection{Proof of Proposition \ref{prop:diagram}}
Define
\[
\begin{gathered}F^0_{\tilde \beta}(X\sslash T) := F_{\tilde \beta}(X\sslash T) \cap X^{g_{\tilde \beta}}\sslash_G T \quad \quad \text{in}\; X^{g_{\tilde \beta}} \sslash T,\\
X_{\tilde \beta}^0 := X_{\tilde \beta} \times_{F_{\tilde \beta}(X\sslash T)} F^0_{\tilde \beta}(X\sslash T).
\end{gathered}\] 
Also let $\beta = \rpic(\tilde \beta)$ and define
\begin{equation}\label{eq:defpsi}
\begin{aligned}
\psi_{\tilde \beta}: F^0_{\tilde \beta}(X\sslash T) &\rightarrow F_{\beta}(X\sslash G)\\
((\wP{a})_S, \T, \sigma) &\mapsto ((\wP{a})_S, G \times_T \T, \sigma).
\end{aligned}\end{equation}

The proof of Proposition \ref{prop:diagram} is completed by the following lemma. 
\begin{lemma}\label{lem:nearly}
The substack $X^0_{\tilde \beta}\subset X^{g_{\tilde \beta}}$ is invariant under the action of $P_{\tilde \alpha}$ and the composition
\begin{equation}\label{eq:psi-invariant1}
X_{\tilde\beta}^0 \rightarrow F^0_{\tilde \beta}(X\sslash T) \xrightarrow{\psi_{\tilde \beta}} F_{\beta}(X\sslash G)
\end{equation}
descends to a closed embedding
\begin{equation}\label{eq:induced}
[X_{\tilde\beta}^0/P_{\tilde\alpha}] \rightarrow F_{\beta}(X\sslash G).
\end{equation}
\end{lemma}
The image of \eqref{eq:induced}, which is equal to the image of $\psi_{\tilde \beta}$, is the closed substack of $F_\beta(X\sslash G)$ that we denote $F_{\tilde \beta}(X\sslash G)$.
\begin{proof}
The argument of \cite[Sec~4.3]{nonab1} works with minor modifications, as follows. First, the proof of \cite[Lem~4.3.1]{nonab1} shows that $X^0_{\tilde \beta}$ is a $P_{\tilde \alpha}$-invariant subscheme of $X^{g_{\tilde \beta}} \cap X^s(G).$ 

To show that \eqref{eq:psi-invariant1}
is invariant under the action of $P_{\tilde\alpha}$, we need an explicit description of the element of $\wp \in \Aut(\cZ)$ determined by $p \in P_{\tilde \alpha}$ as in Lemma \ref{lem:auts}, where $\cZ$ is the fiber bundle for the tautological family on $X_{\tilde \beta}^0$. By definition we have $\phi_V(v)=p$, and from \eqref{eq:coordinateiso} we have
\[
\phi_U(u^a)=\tau_{\intalpha}^{-1}(u)p\tau_{\intalpha}(u) \quad \quad \text{for}\;u \in \torus\times X^0_{\tilde \beta}.
\]
Hence, in the homogeneous coordinates of \eqref{eq:ufam1}, we have that $\wp$ is given by 
\[
(x, u, v, y) \mapsto (x, u, v, u^{-1}\cdot_{rig}(p\cdot(u\cdot_{rig} y))).
\]
With this formula, the argument in \cite[Lem~4.3.2]{nonab1} shows that \eqref{eq:psi-invariant1} is invariant as claimed.

Finally, we follow the argument of \cite[Lem~4.3.3]{nonab1} to show that \eqref{eq:induced} is a closed embedding.
It suffices to show that \eqref{eq:induced} is a proper monomorphism. Since \eqref{eq:induced} is induced by the tautological family, it commutes with projections to $\aff{X}{G}$. But $[X_{\tilde \beta}^0/P_{\tilde \alpha}]$ is proper over $\aff{X}{G}$, so \eqref{eq:induced} is proper.
To check that \eqref{eq:induced} is a monomorphism, it is enough to check that the map from the prestack $[X^0_{\tilde \beta}/P_{\tilde \alpha}]^{pre}$ defined in \cite[Prop~2.6]{romagny} is a monomorphism, i.e., fully faithful (see also \cite[Thm~4.1]{romagny}). This prestack has objects equal to the objects of $X^0_{\tilde \beta}$ and arrows coming from the action of $P_{\tilde \alpha}$. Let $a_i: S \rightarrow X_{\tilde \beta}^0$ be two objects of $[X^0_{\tilde \beta}/P_{\tilde \alpha}]^{pre}(S)$. As in \cite[Lem~4.3.3]{nonab1}, an arrow between the images of these objects under \eqref{eq:induced} implies the existence of morphisms $\phi_V: V_S \rightarrow G$ and $\phi_U: U_S \rightarrow G$ satisfying
\begin{equation}\label{eq:outerref}
\begin{gathered}
\phi_V\cdot(a_1\circ pr_2) = a_2\circ pr_2 \quad \quad \text{as maps}\;V_S \rightarrow Z\\
(\phi_V\circ \kappa_V) \tau_{\intalpha} = \tau_{\intalpha}(\phi_U\circ \kappa_U)\quad \quad \text{as maps}\;\torus_S \rightarrow G.
\end{gathered}
\end{equation}
The first equation shows that $\phi_V$ factors as $V\times S \xrightarrow{pr_2} S \xrightarrow{p} G$ for some $p \in G(S)$ sending $a_1$ to $a_2$. Restricting the second equation to the closed subscheme $\mathbf{1}_S \subset \torus_S$ defined by the identity shows that $\phi_U|_{\mathbf{1}_S}=p$, and then restricting the same equation to $(\bmu_a)_S$ shows that $p \in Z_G(g_{\tilde \beta})$. Finally, the second equation also shows that $p \in P_{\tilde \alpha}$, since the desired extension of $\tau_{\intalpha}^{-1}(u)p\tau_{\intalpha}(u)$ to a morphism  $\AA^1_S\rightarrow G$ is $\phi_U\circ \kappa_U$. The equations \eqref{eq:outerref} uniquely determine $p$.
\end{proof}

\subsection{Computing the nonabelian $I$-function}
\subsubsection{Weyl group action}
Define
\[
 F^0_\beta(X\sslash T) = \bigsqcup_{\tilde \beta \rightarrow \beta} F^0_{\tilde \beta}(X\sslash T)
\]
and let $\psi: F^0_\beta(X\sslash T) \rightarrow F_\beta(X\sslash G)$ be defined to equal $\psi_{\tilde \beta}$ on $F_{\tilde \beta}^0(X\sslash T)$. 

\begin{lemma}\label{lem:surjective}
The map $\psi:F^0_\beta(X\sslash T) \rightarrow F_\beta(X\sslash G)$ is surjective.
\end{lemma}
\begin{proof}
We modify the proof of \cite[Lem~5.1.1]{nonab1} as follows. The image of $\psi$ is equal to the union of $F_{\tilde \beta}(X\sslash G) \subset F_\beta(X\sslash G)$ over all $\tilde \beta$ mapping to $\beta$. In particular it is closed, so by \cite[Tag~06G2]{tag} it suffices to show that $\psi$ is surjective on $\CC$-points.
We claim first that if $((\wP{a}), \T_{\tilde \alpha}, \sigma)$ has $\sigma_V$ equal to a constant function, then it is fixed by $\CC^*_{\lambda^{1/a}}$. Indeed, such a quasimap is given in homogeneous coordinates by
\[
(u,v) \mapsto u^{-1}\cdot_{rig} \sigma_V
\]
which one may directly check is invariant under $\CC^*_{\lambda^{1/a}}$.
So it suffices to show that if $q=((\wP{a}), \P_{\intalpha}, \sigma)$ is in $F_\beta(X\sslash G)$, it is isomorphic to a quasimap $((\wP{a}), \P_{\intalpha}, \rho)$ with $\rho_V$ a constant function. 

Since $q$ is $\CC^*_{\lambda^{1/a}}$-fixed, from \eqref{eq:action3} and \eqref{eq:coordinateiso} we have an isomorphism of quasimap families on $\CC^*_{\lambda^{1/a}}$ determined by morphisms $\Psi_V: \CC^*_{\lambda^{1/a}}\times V \to G$ and $\Psi_U: \CC^*_{\lambda^{1/a}}\times U \to G$ satisfying the following (note that we write $\mu:= \lambda^{1/a}$ for the parameter on $\CC^*_{\lambda^{1/a}}$):
\begin{align}
\Psi_V(\mu, w^{-1})\tau_{\intalpha}(w) &= \tau_{\intalpha}(\mu^{-1}w)\Psi_U(w^a)  & \mu \in \CC^*_{\lambda^{1/a}}, &\; w \in \Omega \label{eq:surj1}\\
\Psi_U(u)\sigma_U(u) &= \sigma_U(\mu^{-a}u) & \mu \in \CC^*_{\lambda^{1/a}}, &\;u \in U \notag\\
\Psi_V(v)\sigma_V(v) &= \sigma_V(\mu v) & \mu \in \CC^*_{\lambda^{1/a}}, &\;v \in V. \notag
\end{align}
On the other hand, restricting the quasimap to $[V/\bmu_a] \subset \wP{a}$, we get a $\CC^*_{\lambda^{1/a}}$-equivariant family of maps $[V/\bmu_a] \to X\sslash G$. By Lemma \ref{lem:extend} (using that $q$ is representable) the morphism $\Psi_V$ extends to $\Phi_V: \AA^1_{\lambda^{1/a}} \times V \to G$ satisfying
\begin{equation}\label{eq:surj2}
\Phi_V(\mu, v) \sigma_V(v) = \sigma_V(\mu v) \quad \quad \text{for all}\;\mu \in \AA^1_{\lambda^{1/a}},\; v \in V.
\end{equation}
Here, $\AA^1_{\lambda^{1/a}}$ contains $\CC^*_{\lambda^{1/a}}$ as the complement of the origin.
Now we define $\overline \Omega \simeq \AA^1$ to contain $\Omega$ as the complement of the origin, and we define $\Phi_{\Omega}': (\AA^1_{\lambda^{1/a}} \times \overline \Omega) \setminus\{(0,0)\} \to G$ by
\[
\Phi_{\Omega}' = \left\{ \begin{array}{ll}
\tau_{\intalpha}(\mu^{-1})\Psi_U(\mu,  w^a) & \mu \in \CC^*_{\lambda^{1/a}},\; w \in \overline \Omega\\
\tau_{\intalpha}(w)^{-1}\Phi_V(\mu, w^{-1})\tau_{\intalpha}(w) & \mu \in \AA^1_{\lambda^{1/a}},\;w \in \Omega. \end{array}\right.
\]
The pieces of $\Phi_\Omega'$ agree on their common domain of definition by \eqref{eq:surj1}. By Hartog's theorem we can extend $\Phi_\Omega'$ to $\Phi_\Omega: \AA^1_{\lambda^{1/a}} \times \overline \Omega \to G$. Next we compute that $\Phi_\Omega'$ and hence $\Phi_\Omega$ is invariant under the action of $\bmu_a$ on $\overline \Omega$ and hence descends to a function $\Phi_U: \AA^1_{\lambda^{1/a}} \times U \to G$. When $\lambda \neq 0$ this is a direct computation, and when $w \neq 0$ it follows from \eqref{eq:bmu-equivariance}.

We have constructed $\Phi_U: \AA^1_{\lambda^{1/a}} \times U \to G$ and $\Phi_V: \AA^1_{\lambda^{1/a}} \times V \to G$ such that (by definition)
\begin{equation}\label{eq:number4}
\Psi_V(\mu, w^{-1})\tau_{\intalpha}(w) = \tau_{\intalpha}(w)\Psi_U(w^a)  \quad \quad \quad \mu \in \AA^1_{\lambda^{1/a}}, \; w \in \Omega
\end{equation}
We set $\phi^0_U = \Phi_U(0,u)$ and $\phi^0_V = \Phi_V(0,v)$. The restriction of \eqref{eq:number4} says that this defines an automorphism of $\P_{\tilde \alpha}$. Let $\rho := \phi^0 \circ \sigma$; by \eqref{eq:surj2} we have $\rho_V = \sigma_V(0 \cdot v) = \sigma_V(0)$ a constant function, as desired.

\end{proof}

Define $ev_\star: F^0_{\beta}(X\sslash T) \rightarrow I_{\mu}(X\sslash_G T)$ to equal $ev_\star$ on each component. Since $F^0_\beta(X\sslash T)$ is a stack of maps to $[X/T]$ it has a natural action by the Weyl group $W$ leaving
 $\psi$ invariant and making $ev_\star$ equivariant (see \cite[(12)]{nonab1}). For $\tilde \alpha \in \Hom(\chi(T), \QQ)$ we set $W_{\tilde \alpha} := N_{L_{\tilde \alpha}}(T)/T$, the Weyl group of $T \subset L_{\tilde \alpha}$ (the group $L_{\tilde \alpha}$ was defined in Section \ref{sec:defs}).

\begin{lemma}\label{lem:group-action}
Let $\tilde \beta_i \in \Hom(\Pic^T(X), \QQ)$ be a full set of representatives of distinct $W$-orbits on $\rpic^{-1}(\beta)$. Then the following is a decomposition into open and closed substacks:
\begin{equation}\label{eq:decompose1}
F_\beta(X\sslash G) = \bigsqcup_i F_{\tilde \beta_i}(X\sslash G) \quad \quad \text{where} \; F_{\tilde \beta_i}(X\sslash G) = \psi(F_{\tilde \beta_i}^0(X\sslash T)).
\end{equation}
Moreover, if $\tilde \alpha_i = \bdtilde(\tilde \beta_i)$, the stabilizer of $\tilde \beta_i$ is $W_{\tilde \alpha_i}$.
\end{lemma}
\begin{proof}
The statement and proof of \cite[Lem~5.1.2]{nonab1} carry over verbatim. For \cite[Lem~5.1.2(3)]{nonab1}, it may help to note that $L_{\tilde \alpha}$, a priori the centrilizer of $\tau_{\intalpha}$ in $Z_G(g_{\tilde \beta})$, is also the centrilizer of $\tau_{\intalpha}$ in $G$, since any element of $G$ that commutes with $\tau_{\intalpha}$ necessarily commutes with $g_{\tilde \beta}=\tau_{\intalpha}(e^{-2\pi i/a}).$ The cited lemma, together with Lemma \ref{lem:surjective}, shows that the decomposition \eqref{eq:decompose1} holds and that the stabilizers are as described. The components $F_{\tilde \beta_i}(X\sslash G)$ are all closed by Lemma \ref{lem:nearly}; there are finitely many of them by Remark \ref{rmk:finite}.
\end{proof}

\subsubsection{Proof of Theorem \ref{thm:main}}
Recall that we have defined $F_{\tilde \beta}(X\sslash G) = \psi_{\tilde \beta}(F_{\tilde \beta}^0(X\sslash T))$, a closed substack of $F_{\beta}(X\sslash G)$.
\begin{lemma}\label{lem:computation}
We have the following relationships on $[F_{\tilde \beta}^0(X\sslash T)/\CC^*_\lambda]$:
\begin{align*}
\psi_{\tilde \beta}^*[F_{\tilde \beta}(X\sslash G)]^{\vir}&=[F_{\tilde \beta}^0(X\sslash T)]^{\vir} \\
\psi_{\tilde \beta}^*e_{\CC^*_\lambda}(N^{\vir}_{F_{\tilde \beta}(X\sslash G)})&=\left(\prod_{i=1}^{m}C^\circ(\tilde \beta, \rho_i)\right) e_{\CC^*_\lambda}(N^{\vir}_{F^0_{\tilde \beta(X\sslash T)}}).
\end{align*}
Here, $z$ is the Euler class of the line bundle on $[F^0_{\tilde \beta}(X\sslash T)/\CC^*_\lambda]$ determined by the identity character of $\CC^*_\lambda$.
\end{lemma}

\begin{proof}
Apply \cite[Lem~A.2.3]{cjw} with $B_V = \fU=pt$, $V=\cC=\wP{a}$, $W=[X/T]$, $Z=[X/G]$, $B_W = F_{\tilde \beta}(X\sslash G)$, $B_Z = F^0_{\tilde \beta}(X\sslash T)$, and $\mu_B = \psi_{\tilde \beta}$. We get a morphism of (equivariant) distinguished triangles
\begin{equation}\label{eq:abel-dt}
\begin{tikzcd}
\psi_{\tilde \beta}^*\EE_{F_{\tilde \beta}(X\sslash G)} \arrow[r] \arrow[d, "\phi_{F_{\tilde \beta}(X\sslash G)}"] & \EE_{F_{\tilde \beta}^0(X\sslash T)} \arrow[d, "\phi_{F^0_{\tilde \beta}(X\sslash T)}"] \arrow[r] & R\pi_*(q^*\TT_{[X/T]/[X/G]})^\vee \arrow[r] \arrow[d]& {}\\
\psi_{\tilde \beta}^*\LL_{F_{\tilde \beta}(X\sslash G)} \arrow[r] & \LL_{F_{\tilde \beta}^0(X\sslash T)}  \arrow[r] & \LL_{F_{\tilde \beta}(X\sslash G)/F_{\tilde \beta}^0(X\sslash T)} \arrow[r] & {}
\end{tikzcd}
\end{equation}
where the vertical arrows are as defined in \eqref{eq:almost-pot}. Here, $q: \wP{a} \times F_{\tilde \beta}^0(X\sslash T) \rightarrow [Z/T]$ is the universal quasimap and $\TT_{[X/T]/[X/G]}$ is the relative tangent bundle. 

We compute $\psi^*_{\tilde \beta}e_{\CC^*_\lambda}(N^{\vir}_{F_{\tilde \beta}(X\sslash G)})$ using the top row of \eqref{eq:abel-dt} as in the proof of \cite[Cor~5.2.3]{nonab1}. A priori the tags \cite[Tag~0F8G, 0F9F]{tag} cited in that argument only apply to schemes; however the same arguments can be made to work in our context as in Remark \ref{rmk:perfect}. To compute the weights of the virtual normal bundle we use the following fact.

\begin{lemma}\label{lem:pullback-formula}
Let $f:X \rightarrow Y$ be a morphism of schemes and $\phi: H \rightarrow G$ a morphism of algebraic groups with $H$ acting on $X$ and $G$ acting on $Y$, such that for $h \in H$ and $x \in X$ we have $f(hx)=\phi(h)f(x)$. If $V$ is a $G$-representation, then the fiber product of $(f, \phi): [X/H] \rightarrow [Y/G]$ and $p_2:[(Y\times V)/G] \rightarrow [Y/G]$ is $[(X\times V)/H]$ where $V$ is an $H$-representation via $\phi$.
\end{lemma}
\begin{proof}
It is straightforward to check that there is a fiber product of the prestacks defined in \cite[Prop~2.6]{romagny}. 
\end{proof}

From the description of the universal family in \eqref{eq:ufam1}, the universal map $q: \wP{a} \times F_{\tilde \beta}^0(X\sslash T) \rightarrow [X/T]$ is given as a $\CC^*_{\lambda^{1/a}}$-equivariant morphism by a morphism as in Lemma \ref{lem:pullback-formula} where the group homomorphism is 
\begin{equation}\label{eq:compute1}T\times \CC^*\times \CC^*_{\lambda^{1/a}} \xrightarrow{(t,s, \mu) \mapsto t \tau_{\intalpha}(s)^{-1}\tau_{\intalpha}(\mu)^{-1}} T.\end{equation} 
Hence the $\CC^*_{\lambda^{1/a}}$-equivariant vector bundle $q^*\TT_{[X/T]/[X/G]}$ is given by subspace of the Lie algebra $\mathfrak{g}$ of $G$ with nontrivial weights, as a $T\times\CC^*\times \CC^*_{\lambda^{1/a}}$-representation via the homomorphism \eqref{eq:compute1} and the adjoint representation of $T$ on $\mathfrak{g}$.

As in \cite[Cor~5.2.3]{nonab1} we get the following equality of $\CC^*_{\lambda^{1/a}}$-equivariant bundles:
\[
R^j\pi_*(q^*\TT_{[X/T]/[X/G]}) = \bigoplus_{i=1}^m \L_{\rho_i} \otimes H^j(\wP{a}, \OO_{\wP{a}}(a\tilde \beta(\rho_i)))\otimes \CC_{\mu^{a\tilde \beta(\rho_i)}}
\]
Here we've used $\mu := \lambda^{1/a}$ for the parameter on $\CC^*_{\lambda^{1/a}}$, and $\CC_{\mu^{a\tilde \beta(\rho_i)}}$ is the one-dimensional representation of $\CC^*_{\lambda^{1/a}}$ where $\mu$ acts by the character ${\mu^{a\tilde \beta(\rho_i)}}.$
Now we use Example \ref{ex:cohomology}: for $i=0,1$ the vector space $H^j(\wP{a}, \OO_{\wP{a}}(a\tilde \beta(\rho_i)))$ has a basis of monomials $u^mv^n$ with $m, n\in \ZZ$ and $m+n/a = \tilde \beta(\rho_i)$. If $i=0$ then we require $m,n \geq 0$ and if $i=1$ then we require $m,n<0$. The weight of the monomial $u^mv^n$, \textit{with respect to $\CC^*_\lambda$}, is $\tilde \beta(\rho_i)-m=n/a$, which satisfies $n/a-\tilde \beta(\rho_i)=m \in \ZZ$. This (plus the additivity of $e_{\CC^*_{\lambda}}$) completes the computation of $\psi^*_{\tilde \beta}e_{\CC^*_\lambda}(N^{\vir}_{F_{\tilde \beta}(X\sslash G)})$.

To compute $\psi^*_{\tilde \beta}[F_{\tilde \beta}(X\sslash G)]^{\vir}$ we apply $\fix$ to \eqref{eq:abel-dt} (it affects only the top triangle). Then \cite[Lem~A.3.5]{cjw} implies that the left and middle vertical arrows of the resulting diagram are the canonical perfect obstruction theories for $F_{\tilde \beta}(X\sslash G)$ and $F_{\tilde \beta}^0(X\sslash T)$, respectively. The right vertical arrow is a quasi-isomorphism: this is because it is already an obstruction theory and $R^1\pi_*(q^*\TT_{[X/T]/[X/G]})^{\fix}$ vanishes, as can be seen from the above computation. Now the desired equality $\psi^*_{\tilde \beta}[F_{\tilde \beta}(X\sslash G)]^{\vir} = [F^0_{\tilde \beta}(X\sslash T)]^{\vir} $ follows from \cite[Cor~4.9]{manolache} (see the proof of \cite[Cor~5.2.3]{nonab1}).
\end{proof}

\begin{proof}[Proof of Theorem \ref{thm:main}]
We first show that to prove Theorem \ref{thm:main} it suffices to prove the following equality in $A_*(I_\mu(X\sslash_G T))[z,z^{-1}]$:
\begin{equation}\label{eq:main2}
\varphi^* (ev_{\star})_*Res^{X\sslash G}_\beta =  \sum_{\tilde \beta \rightarrow \beta} \left( \prod_{i=1}^m C(\tilde \beta, \rho_i)^{-1}\right)j^*(ev_{\star})_*\Res^{X\sslash T}_{\tilde \beta}.
\end{equation}
Recall that $\widetilde{ev}_\star = a\circ\iota \circ \varpi \circ ev_\star,$ where $a$ denotes multiplication by $a$. Given \eqref{eq:main2} we obtain Theorem \ref{thm:main} for unrigidified (resp. rigidified) $I$-functions by applying $\iota_*$ (resp. $(a \circ \varpi \circ \iota)_*$) to both sides. We explain what happens when we apply $(a \circ \varpi \circ \iota)_*$. One uses the fibered squares
\[
\begin{tikzcd}
I_\mu(X\sslash G) \arrow[d, "\varpi"] & \arrow[l, "\varphi"] I_\mu(X\sslash_G T) \arrow[d, "\varpi"] \arrow[r, "j"] & I_\mu(X\sslash T) \arrow[d, "\varpi"]\\
\overline{I}_\mu(X\sslash G) & \arrow[l, "\varphi"]\arrow[r, "j"] \overline{I}_\mu(X\sslash_G T) & \overline{I}_\mu(X\sslash T)
\end{tikzcd}
\]
where the rigidification maps $\varpi$ are proper (and smooth) and $\varphi$ and $j$ are flat. These claims may be checked on each component $\overline{I}_{\mu_a}(X\sslash G)$ of the cyclotomic inertia stack.
Now when we apply ${a_*}\varpi_*\iota_*$ we use the equalities ${a}_* \varpi_*\iota_*\varphi^*=\varphi^*{a}\varpi_*\iota_*$ and ${a_*} \varpi_*\iota_*j^*=j^*{a_*}\varpi_*\iota_*$ (from \cite[Lem~3.9]{vistoli}) on the left and right hand sides, respectively. On the right hand side, we also use the projection formula and the facts that $c_1(\L_{\rho_i}) = \iota^*\varpi^*c_1(\L_{\rho_i})$ and $z = \iota^*\varpi^*z$, using the rule in Lemma \ref{lem:pullback-formula} for pulling back line bundles and the convention in Example \ref{ex:chern-classes}. This proves Theorem \ref{thm:main} assuming \eqref{eq:main2} holds.

To prove \eqref{eq:main2} we have two cases: $F_\beta(X\sslash G)$ is either empty or nonempty. If $F_\beta(X\sslash G)$ is empty then the left hand side of \eqref{eq:main2} is defined to be zero. Moreover we claim that $F^0_{\tilde \beta}(X\sslash T)$ is empty for any $\tilde \beta$ mapping to $\beta$: Indeed, a point of $F^0_{\tilde \beta}(X\sslash T)$ would correspond to a fixed quasimap $q=(\wP{a}, \T_{\intalpha}, \sigma)$ to $X\sslash T$ of degree $\tilde \beta$ such that $ev_\star(q) \in I_\mu(X\sslash_G T)$. Then we would have that $q:\wP{a} \to [X/T]$ meets the open substack $X\sslash_G T$, so the associated quasimap $(\wP{a}, G\times_T\T_{\intalpha}, \sigma)$ would be a fixed quasimap to $X\sslash G$ of degree $\rpic(\tilde \beta) = \beta$, contrary to assumption that $F_\beta(X\sslash G)$ is empty. This shows that the composition $j^*(ev_\star)_*$ on the right hand side of \eqref{eq:main2} is zero, since the closed image of $ev_\star$ is disjoint from the open set that is the domain of $j$.

To prove \eqref{eq:main2} when $F_\beta(X\sslash G)$ is not empty, we compute as follows. Note that the left hand side equals $\varphi^*\iota_*I^{X\sslash G}_\beta$. Writing $I^{X\sslash G}_{\beta}$ as a sum over the open and closed substacks $F_{\tilde \beta_i}(X\sslash G)$ of $F_{\beta}(X\sslash G)$, we have
\begin{equation}\label{eq:step1}
\varphi^*\iota_*I^{X\sslash G}_\beta = \sum_{\tilde \beta_i \rightarrow \beta} \varphi^*(ev_\star)_*\Res^{X\sslash G}_{\tilde \beta_i}
\end{equation}
Now we use \eqref{eq:inertia5} to compute $\varphi^*$ with a different formula on each component $F_{\tilde \beta_i}(X\sslash G)$. Namely, from \eqref{eq:inertia5} and \eqref{eq:inertia6}, the formula \eqref{eq:step1} becomes
\begin{equation}\label{eq:step2}
\sum_{\tilde \beta_i \rightarrow \beta} \;\;\sum_{w_1 \in W/W_{Z(g_i)}} (w_1^{-1})^*\varphi_{g_i}^*\eta_G^*(ev_\star)_*\Res^{X\sslash G}_{\tilde \beta_i},
\end{equation}
where we have written $g_i$ for $g_{\tilde \beta_i}$ and $Z(g_i)$ for $Z_G(g_{\tilde \beta_i})$. But $\eta_G^*(ev_\star)_*=\eta_G^*(\eta_G)_*f_*i_*$ by Proposition \ref{prop:diagram}, so \eqref{eq:step2} becomes
\begin{equation}\label{eq:step3}
\sum_{\tilde \beta_i \rightarrow \beta} \;\;\sum_{w_1 \in W/W_{Z(g_i)}} (w_1^{-1})^*\varphi_{g_i}^*f_*i_*\Res^{X\sslash G}_{\tilde \beta_i}.
\end{equation}

The next step is to argue as in \cite[Lem~5.3.1]{nonab1}; namely, for any $\delta \in A_*(X^{g_i}\sslash_G P_{\tilde \alpha_i})$, we have 
\begin{equation}\label{eq:brion}
\varphi_{g_i}^*f_*\delta = \sum_{w\in W_{Z(g_i)}/W_{\tilde \alpha}} (w^{-1})^*\left[ \frac{{p}^*\delta}{\prod_{\rho_i\in R^-_{\tilde\alpha}} c_1(\L_{\rho_i})}\right]
\end{equation}
where $R^-_{\tilde \alpha}$ is the set of roots of $Z_G(g_i)$ whose inner product with the dual character $\intalpha$ is negative. To prove \eqref{eq:brion}, since the Kresch Chow group of a global quotient is the same as the Edidin-Graham equivariant Chow group of the cover, it suffices to prove \eqref{eq:brion} for $G$-equivariant Chow groups. This follows from the definitions and the original statement of \cite[Lem~5.3.1]{nonab1}. 

Formula \eqref{eq:step3}, combined with \eqref{eq:brion}, becomes
\[
\sum_{\tilde \beta_i \rightarrow \beta} \;\;\sum_{w_1 \in W/W_{Z(g_i)}} (w_1^{-1})^*\sum_{w_2 \in W_{Z(g_i)}/W_{\tilde \alpha_i}} (w_2^{-1})^*\left( \frac{p^*i_*\Res^{X\sslash G}_{\tilde \beta_i}}{\prod_{\rho_i \in R_{\tilde \alpha_i}^-} c_1(\L_{\rho_i})}\right)
\]
Combining the two Weyl-group summations into one, we have
\[
\sum_{\tilde \beta_i \rightarrow \beta} \;\;\sum_{w \in W/W_{\tilde \alpha_i}} (w^{-1})^*\left( \frac{p^*i_*\Res^{X\sslash G}_{\tilde \beta_i}}{\prod_{\rho_i \in R_{\tilde \alpha_i}^-} c_1(\L_{\rho_i})}\right),
\]
which is the analog of \cite[(62)]{nonab1}. The analog of \cite[(63)]{nonab1} is a commuting diagram
\[
\begin{tikzcd}
F^0_{w\tilde \beta_i}(X\sslash T) \arrow[r, "w^{-1}"] \arrow[d,hook, "ev_\star"]& F^0_{\tilde \beta_i}(X\sslash T) \arrow[r, "\psi_{\tilde \beta_i}"] \arrow[d,hook, "ev_\star"]& F_{\tilde \beta_i}(X\sslash G) \arrow[d,hook, "i"]\\
X^{g_{w\tilde \beta_i}}\sslash_G T \arrow[r,"w^{-1}"] & X^{g_i}\sslash_G T \arrow[r,"p_{\tilde \alpha_i}"] & \bigsqcup X^{g_i}\sslash_G P_{\intalpha_i}
\end{tikzcd}
\]
The square on the right is fibered by Proposition \ref{prop:diagram} and the horizontal maps are flat. We have $(w^{-1})^*p_{\tilde \alpha_i}^*i_* = (ev_\star)_*(w^{-1})^*\psi_{\tilde \beta_i}^* = (ev_\star)_*\psi_{w\tilde \beta_i}^*$ by \cite[Lem~3.9]{vistoli}. 

From here we apply Lemma \ref{lem:group-action}, arguing as in Section \cite[Sec~5.3]{nonab1}. As the analog of \cite[(66)]{nonab1} we obtain a formula
\begin{equation}\label{eq:step7}
\varphi^* \iota_*I^{X\sslash G}_\beta = \sum_{\tilde \beta \rightarrow \beta}  \frac{(ev_\star)_*\psi_{\tilde \beta}^*\Res^{X\sslash G}_{\tilde \beta}}{\prod_{\rho_i \in R_{\tilde \alpha}^-} c_1(\L_{\rho_i})}.
\end{equation}
Applying Lemma \ref{lem:computation} yields
\[
\varphi^* \iota_*I^{X\sslash G}_\beta =  \sum_{\tilde \beta \rightarrow \beta}  j^*\mathscr{R}_{\tilde \beta}(ev_\star)_*\Res^{X\sslash T}_{\tilde \beta}.
\]
where 
\[
\mathscr{R}_{\tilde \beta} = \frac{1}{\prod_{\rho_i \in R^-_{\tilde \alpha}} c_1(\L_{\rho_i})} \frac{\prod_{i \;\mathrm{s.t.}\; \tilde \beta(\rho_i)>0}\prod_{k-\tilde \beta(\rho_i)\in\ZZ,\,0 < k \leq \tilde \beta(\rho_i)}^{}(c_1(\L_{\rho_i}) + kz)}{\prod_{i \;\mathrm{s.t.}\;\tilde \beta(\rho_i)<0}\prod_{k-\tilde \beta(\rho_i)\in\ZZ,\,\tilde \beta(\rho_i)+1 \leq k <0}^{}(c_1(\L_{\rho_i}) + kz)}.
\]
Recall that $R^-_{\tilde \alpha}$ is the set of roots of $Z_G(g_{\tilde \beta})$ with $\intalpha(\rho_i)=a\tilde \beta(\rho_i)<0$. To calculate the roots of $Z_G(g_{\tilde \beta})$, note that $Z_G(g_{\tilde \beta})$ is the tangent space to the fixed locus of the group generated by $g_{\tilde \beta}$ acting on $G$ by conjugation. So the tangent space to $Z_G(g_{\tilde \beta})$ at the identity is the fixed part of the tangent space to $G$ at the identity; in other words, the roots of the Lie algebra of $Z_G(g_{\tilde \beta})$ are precisely those roots $\rho_i$ of $G$ with $\rho_i(g_{\tilde \beta})=1$. Since 
\[\rho_i(g_{\tilde \beta}) = \rho_i(\tau_{\intalpha}(e^{2\pi i/a})^{-1}) = e^{2\pi i \tilde \alpha(\rho_i)},\]
we see that $\rho_i$ is a root of $Z_G(g_{\tilde \beta})$ if and only if $\tilde \alpha(\rho_i)=\tilde \beta(\rho_i)\in \ZZ$. So $R^-_{\tilde \alpha}$ is the set of roots of $G$ with $\tilde \alpha(\rho_i)=\tilde \beta(\rho_i) \in \ZZ_{<0}.$
By checking the cases $\tilde \beta(\rho_i)<0, \tilde \beta(\rho_i)=0,$ and $\tilde \beta(\rho_i)>0$ each in turn, one may check that this coefficient may also be written as $\mathscr{R}_{\tilde \beta} = \prod_{\rho_i}C(\beta, \rho_i)^{-1}.$
This completes the proof of \eqref{eq:main2}.
\end{proof}

\begin{remark}
If $R$ is not trivial, Lemma \ref{lem:computation} still holds as an equality of classes in $A_*([F^0_{\tilde \beta}(X\sslash T)/(R\times \CC^*_\lambda)])$ after replacing all virtual and characteristic classes with their $R$-equivariant counterparts, and replacing $\rho_i$ by the character of $R \times T$ given by the composition $R \times T \xrightarrow{pr_2} T \xrightarrow{\rho_i} \CC^*$.

Indeed, by \cite[Lem~A.3.3]{cjw} the diagram \eqref{eq:abel-dt} holds $R$-equivariantly, and in place of \eqref{eq:compute1} we have the homomorphism
\begin{equation*}R \times T\times \CC^*\times \CC^*_{\lambda^{1/a}} \xrightarrow{(r, t,s, \mu) \mapsto (r, t \tau_{\intalpha}(s)^{-1}\tau_{\intalpha}(\mu)^{-1})} R \times T.\end{equation*} 
The remainder of the proof of Theorem \ref{thm:main} is routine.
\end{remark}

\section{Quantum Lefschetz for $I$-functions}
We let $Y \subset X, G, \theta, E,$ and $s$ be as in Section \ref{sec:ql-setup}. Let $T\subset G$ be a maximal torus, and denote the weights of $E$ with respect to $T$ by $\epsilon_j$ for $j=1, \ldots, r$. If $Z$ is a scheme with a $G$-action and $F$ is any $G$-representation, set $F_Z := F \times Z$ (the $G$-equivariant trivial bundle), and if $W = [Z/G]$ let $F_W := [F_Z/G]$ be the induced vector bundle on $W$. 
The inclusion $\bi: Y \hookrightarrow X$ induces a map $\bi_*$ making the following diagram commute:
\[
\begin{tikzcd}
\Hom(\Pic^T(Y), \QQ) \arrow[r, "\bi_*"] \arrow[d]& \Hom(\Pic^T(X), \QQ) \arrow[d] \\
\Hom(\Pic^G(Y), \QQ) \arrow[r, "\bi_*"] & \Hom(\Pic^G(X), \QQ)
\end{tikzcd}
\]
\begin{notation}\label{not:ql-degrees}
We will hereafter use these maps implicitly, using symbols $\tilde \beta, \tilde \delta, \beta,$ and $\delta$ for elements of the top left, top right, bottom left, and bottom right corners, respectively. As an example, given $\delta \in \Hom(\Pic^G(X), \QQ)$ we will write $\tilde \delta \mapsto \delta$ for the preimage of $\delta$ under the right vertical map.
\end{notation}

\subsection{Preparation}
We check that our assumptions (1)-(3) in Sections \ref{sec:abel-setup}-\ref{sec:ql-setup} imply that quasimap theory is defined for the complete intersection $(Y, G)$. Recall that $X$ is an affine variety, $G$ is a reductive group acting on $X$, and $\theta$ is a character defining stable and semi-stable loci. The proof of this lemma was explained to me by Yang Zhou.

\begin{lemma}\label{lem:stable-locus}
If $X^{ss}(G) = X^s(G)$ then $Y^{ss}(G) = Y^{s}(G) = X^s(G)\cap Y$.
\end{lemma}
\begin{proof}
Clearly $X^s(G)\cap Y \subset Y^s(G) \subset Y^{ss}(G)$. Let $y \in Y^{ss}(G)$. Then there is a function $f$ of weight $k\theta$ such that $f(y) \neq 0$. Since the map of coordinate rings $\CC[X] \rightarrow \CC[Y]$ is surjective, there is a function $g$ on $X$ that restricts to $f$. A priori $g$ may not be a function of weight $k\theta$, but by Schur's lemma it has a summand $g'$ of weight $k\theta$ and by $g'$ restricts to $f$. So $y \in X^{ss}(G)$. Since $X^{ss}(G) = X^{s}(G)$ the orbit $Gy$ is closed in the open subset of $X$ where $g' \neq 0$, and hence it is closed in the open subset of $Y$ where $f \neq 0$.
\end{proof}

The closed embedding $Y \rightarrow X$ induces a closed embedding ${I}_\mu(Y\sslash_G T) \rightarrow {I}_\mu(X\sslash_G T)$. A key observation is that since $s$ is a regular section on $X^{s}(G)$, it defines a canonical $G$-equivariant isomorphism between the normal bundle $N_{Y^s(G)/X^s(G)}$ and $E \times Y^s$. We likewise get an explicit description of the normal bundle of inertia stacks as follows.

\begin{lemma}\label{lem:ql-inertia}
The map $ {I}_\mu(Y\sslash_G T) \rightarrow {I}_\mu(X\sslash_G T)$ is a regular local immersion in the sense of \cite[Sec~3.1]{kresch}. If for $g \in T$ we write $(I_\mu(Y\sslash_G T))_{g} = Y^g\sslash_G T,$ then the normal bundle to $({I}_\mu(Y\sslash_G T))_{g} $ in  $({I}_\mu(X\sslash_G T))_{g}$ is given by the following sub $T$-representation of $E$:
\begin{equation}\label{eq:iso1}
E^g := \bigoplus_{\epsilon_j \;\text{s.t.}\; \epsilon_j(g) = 1} \CC_{\epsilon_j}.
\end{equation}
\end{lemma}
\begin{proof}
Observe that $(I_\mu(X\sslash_G T))_{g} = X^g\sslash_G T$, so the embedding  $ {I}_\mu(Y\sslash_G T) \rightarrow {I}_\mu(X\sslash_G T)$ is regular if its smooth cover $(Y^s(G))^g \to (X^s(G))^g$ is regular. Since $X^{s}(G)$ and $Y^{s}(G)$ are smooth varieties, for any $g \in T$ the fixed loci $(X^{s}(G))^g $ and $(Y^{s}(G))^g$ are also smooth. It follows from \cite[Tag~069M, 069G]{tag} that $(Y^{s}(G))^g \rightarrow (X^{s}(G))^g$ is a regular embedding. 

The desired normal bundle lifts to the $T$-equivariant normal bundle of $Y^s(G)$ in $X^s(G)$.
Let $\langle g \rangle \subset T$ be the cyclic subgroup generated by $g$. To compute the normal bundle to $(Y^{s}(G))^g \rightarrow (X^{s}(G))^g$, we use \cite[Prop~3.2]{edix} to identify the sheaf of differentials $\Omega_{(Y^s(G))^g}$ with the $\langle g\rangle$-invariant part of $\Omega_{Y^s(G)}$, and we take the $\langle g\rangle$-fixed part of this exact sequence (see \cite[Tag~06BJ]{tag}):
\[
0 \rightarrow E_{Y^s(G)}^\vee \rightarrow \Omega_{X^{s}(G)}|_{Y^s(G)} \rightarrow \Omega_{Y^{s}(G)} \rightarrow 0.
\]
Here we used the identification $N_{Y^s(G)/X^s(G)}\simeq E_{Y^s(G)}$ given by the regular section $s$. 
\end{proof}

\begin{remark}\label{rmk:reason}
The reason for the name ``I-nonnegative'' is as follows. Choose $\tilde \delta \in \Hom(\Pic^T(X), \QQ)$ mapping to an $I$-effective class, and let $E^{g_{\tilde \delta}}_{[X/T]}$ be the vector bundle induced by \eqref{eq:iso1} (so $E^{g_{\tilde \delta}}_{[X/T]}$ pulls back to the normal bundle of the regular embedding $(I_{\mu}(Y\sslash_G T))_{g_{\tilde \delta}}\to (I_\mu(X\sslash_G T))_{g_{\tilde \delta}}$). The class $\tilde \delta$ is I-nonnegative if and only if for every morphism $q: \P^1_{a, 1} \to [X/T]$ of class $\tilde \delta$ the bundle $q^*E^{g_{\tilde \delta}}_{[X/T]}$ is equal to $\bigoplus_j\OO_{\wP{a}}(k_j)$ for some $k_j\geq 0$.

\end{remark}

\begin{remark}
If the torus $R$ is not trivial, then our assumptions in Section \ref{sec:intro-equivariance} imply that $Y \subset X$ is an $R$-invariant subset. Since the $R$ action commutes with the $G$ action, the isomorphism \eqref{eq:iso1} holds $R$-equivariantly.
\end{remark}

Let $\delta \in \Hom(\Pic^G(X), \QQ)$ and define
\[
F_\delta(Y\sslash G) := \bigsqcup_{\beta \mapsto \delta} F_\beta(Y\sslash G).
\]
Similarly, for $\tilde \delta \in \Hom(\Pic^T(X), \QQ)$ such that $\tilde \delta \mapsto \delta$, set 
\[
F_{\tilde \delta}(Y\sslash G) := \bigsqcup_{\tilde \beta\mapsto \tilde \delta} F_{\tilde \beta}(Y\sslash G) \quad \quad \quad F_{\tilde \delta}^0(Y\sslash T) := \bigsqcup_{\tilde \beta\mapsto \tilde \delta} F_{\tilde \beta}^0(Y\sslash T).
\]
where $F_{\tilde \beta}(Y\sslash G)$ is the closed substack of $F_{\rpic(\tilde \beta)}(Y\sslash G)$ and $F_{\tilde \beta}^0(Y\sslash T)$ is the open substack of $F_{\tilde \beta}(Y\sslash T)$ defined in Proposition \ref{prop:diagram}.
For $\tilde \beta \in \Hom(\Pic^T(Y), 
\QQ)$ such that $\tilde \beta \mapsto \tilde \delta$, we have $g_{\tilde \beta} = g_{\tilde \delta}$. We have the following extension of Proposition \ref{prop:diagram}.
\begin{proposition}\label{prop:diagram2}
There is a closed embedding $\bj_G: F_{\tilde \delta}(Y\sslash G) \rightarrow F_{\tilde \delta}(X\sslash G)$ fitting into the following commuting diagram.
\begin{equation}\label{eq:bigmess}
\begin{tikzcd}
F_{\tilde \delta}(Y\sslash G) \arrow[ddr, "ev_\star"']  & \arrow[rrr, bend left, "\bj_T"] F^0_{\tilde \delta}(Y\sslash T) \arrow[l,"\psi_{\tilde \delta}"'] \arrow[d,hook, "ev_\star"] & & F_{\tilde \delta}(X\sslash G)  \arrow[from=lll, crossing over, bend left, "\bj_G"] & F^0_{\tilde \delta}(X\sslash T) \arrow[l,"\psi_{\tilde \delta}"'] \arrow[d,hook, "ev_\star"]\\
& Y^{g_{\tilde \delta}} \sslash_G T \arrow[d, "\varphi_{g}"] & & & X^{g_{\tilde \delta}} \sslash_G T \arrow[d, "\varphi_{g}"]\arrow[from=lll,  "\bi"]\\
& Y^{g_{\tilde \delta}} \sslash_G Z_G(g_{\tilde \delta}) \arrow[rrr, "\bi_G"]& & & X^{g_{\tilde \delta}} \sslash_G Z_G(g_{\tilde \delta}) \arrow[from=uul, crossing over, near end, "ev_\star"']
\end{tikzcd}
\end{equation}
In fact, the triangle on the left is the pullback of the triangle on the right along the map $\bi$. (One of the arrows is drawn dashed to clarify the diagram.)
\end{proposition}
\begin{proof}
 The commuting triangles come from Proposition \ref{prop:diagram} (applied once to $X$ for the degree $\tilde \delta$ and several times to $Y$ for each $\tilde \beta$ mapping to $\tilde \delta$). Note that this diagram uses only $G$-stable points, so we do not need any smoothness or stable=semistable hypotheses on the $T$-stable locus.
 
 We check that the square with $\bj_G$ and $\bi_G$ is fibered. Observe that the square commutes and that $\bi_G$ and $\bj_G$ are closed embeddings, so it suffices to check that the fiber product is contained in $F_{\tilde \delta}(Y\sslash G)$. Let $S$ be a scheme and let  let $q=((\wP{a})_S, \P, \sigma)$ be an $S$-point of $F_{\tilde \delta}(X\sslash G)$, so the underlying principal $T$-bundle of $\P$ is determined by $\tilde \delta$ and we know that $\sigma$ sends $[V/\bmu_a]$ into $X^s(G)$. If $ev_\star(q)$ is an $S$-point of $X^{g_{\tilde \delta}}\sslash_G Z_G(g_{\tilde \delta}) = (I_\mu(Y\sslash G))_{(g_{\tilde \delta})}$, then $\sigma$ sends $B\bmu_a$ into $Y^s(G)$, and by Lemma \ref{lem:factors} we see that $\sigma$ sends $[V/\bmu_a]$ into $Y^s(G)$ and $q$ is an $S$-point of $F_{\tilde \delta}(Y\sslash G)$. 

\end{proof}

Finally, we define the map $0^!_{\tilde \delta}$ appearing in Theorem \ref{thm:ql}. To do so we define a sub $T$-representation $E^{\tilde \delta \geq 0} \subset E$ by
 \[
E^{\tilde \delta\geq 0} := \bigoplus_{\epsilon_j \, s.t.\, \tilde \delta(\epsilon_j) \in \ZZ_{\geq 0}} \CC_{\epsilon_j}.
\]
We let $E^{\tilde \delta \geq 0}_{F^0_{\tilde \delta}(X\sslash T)}$ denote the induced vector bundle on ${F^0_{\tilde \delta}(X\sslash T)}$, and $s_{\tilde \delta \geq 0}$ is the section induced via pullback by $s$ and the projection $E \to E^{\tilde \delta \geq 0}$. We will show that the square
\begin{equation}\label{eq:fiber-fixed}
 \begin{tikzcd}
 F_{\tilde \delta}^0(Y\sslash T) \arrow[r, "\bj_T"] \arrow[d] & F^0_{\tilde \delta}(X\sslash T) \arrow[d, "s_{\tilde \delta \geq 0}"]\\
 F^0_{\tilde \delta}(X\sslash T) \arrow[r, "0_{\tilde \delta}"] & E^{\tilde \delta\geq 0}_{ F^0_{\tilde \delta}(X\sslash T)}
 \end{tikzcd}
\end{equation}
is fibered, where $0_{\tilde \delta}$ is the zero section. Then we define $0^!_{\tilde \delta}$ to be the refined Gysin map associated to this square. 

\begin{lemma}\label{lem:fiber-fixed}
The square \eqref{eq:fiber-fixed} is fibered.
\end{lemma}
\begin{proof}
Let $E^{\tilde \delta} := E^{g_{\tilde \delta}}$ be the subrepresentation of $E$ defined in \eqref{eq:iso1}, and note that $E^{\tilde \delta}$ has weights $\epsilon_j$ were $\tilde \delta(\epsilon_j) \in \ZZ$. It follows from the proof of Lemma \ref{lem:ql-inertia} that if $s_{\tilde \delta}$ is the section of $E^{\tilde \delta}_{I_\mu(X\sslash_G T)_{g_{\tilde \delta}}}$ induced by $s$, then $I_\mu(Y\sslash_G T)_{g_{\tilde \delta}}$ is the zero locus of $s_{\tilde \delta}$.
Combined with the fiber diagram in  Proposition \ref{prop:diagram2} we deduce a fibered square
\[
\begin{tikzcd}
F_{\tilde \delta}^0(Y\sslash T) \arrow[r] \arrow[d] & F^0_{\tilde \delta}(X\sslash T) \arrow[d, "s_{\tilde \delta}"]\\
F^0_{\tilde \delta}(X\sslash T) \arrow[r, "0"] & E^{\tilde \delta}_{ F^0_{\tilde \delta}(X\sslash T)}
\end{tikzcd}
\]
By \cite[Tag~02XB(3)]{tag}, to prove our lemma it is enough to show that $s_{\tilde \delta}$ factors as
\[
F^0_{\tilde \delta}(X\sslash T) \xrightarrow{s_{\tilde \delta \geq 0}} E^{\tilde \delta \geq 0}_{F^0_{\tilde \delta}(X\sslash T)} \hookrightarrow E^{\tilde \delta}_{F^0_{\tilde \delta}(X\sslash T)}.
\]

By Lemma \ref{lem:abelian}, the section $s_{\tilde \delta}$ of $E^{\tilde \delta}_{F^0_{\tilde \delta}(X\sslash T)}$ (resp. $s_{\tilde \delta \geq 0}$ of $E^{\tilde \delta \geq 0}_{F^0_{\tilde \delta}(X\sslash T)}$) is the pullback of the corresponding section of $E^{\tilde \delta}_{[X/T]}$ (resp. $E^{\tilde \delta \geq 0}_{[X/T]}$) along the map
\begin{equation}\label{eq:fiber-fixed2}
F^0_{\tilde \delta}(X\sslash T) \xrightarrow{ev_\star} (I_{\mu}(X\sslash_G T))_{g_{\tilde \delta}}  \hookrightarrow [X/T].\end{equation}
Let $\pi: \wP{a} \times F^0_{\tilde \delta}(X\sslash T)$ be the universal curve and let $q:\wP{a} \times F^0_{\tilde \delta}(X\sslash T) \to [X/T] $ be the universal quasimap. The composition in \eqref{eq:fiber-fixed2} factors as
\[
F^0_{\tilde \delta}(X\sslash T) \to \wP{a} \times F^0_{\tilde \delta}(X\sslash T)  \xrightarrow{q} [X/T]
\]
where the first arrow is the canonical section of the trivial gerbe $B\bmu_a \times F^0_{\tilde \delta}(X\sslash T) \to F^0_{\tilde \delta}(X\sslash T)$.
So the lemma follows if we can show that $q^*s_{\tilde \delta}$ factors through $q^*s_{\tilde \delta\geq 0}$. 
But $q^*s_{\tilde \delta}$ is a section of $q^*E^{\tilde \delta}$, and we can compute this bundle explicitly using Proposition \ref{prop:ufam}: its restriction to each fiber of the universal curve is isomorphic to $\oplus_{j\,s.t.\,\tilde \delta(\epsilon_j) \in \ZZ} \OO(a\tilde \delta(\epsilon_j)$. A global section of this bundle factors through the subbundle with nonnegative weights, as desired.
\end{proof}

From Lemma \ref{lem:fiber-fixed} we obtain the following. (Recall that $\bi$ was defined in \eqref{eq:bigmess}.)
\begin{lemma}\label{lem:compare-gysin}
If $\tilde \delta$ is I-nonnegative, then the refined Gysin maps $0^!_{\tilde \delta}$ and $\bi^!$ agree.
\end{lemma}
\begin{proof}
The statement follows from Corollary \ref{cor:excess} once we demonstrate the existence of two fiber diagrams such that the lower squares have trivial excess bundles (defined in \eqref{eq:def-excess}):
\begin{equation}\label{eq:two-diagrams}
\begin{tikzcd}
F_{\tilde \delta}^0(Y\sslash T) \arrow[r] \arrow[d] & F^0_{\tilde \delta}(X\sslash T) \arrow[d, "s_{\tilde \delta \geq 0}"]\\
F^0_{\tilde \delta}(X\sslash T) \arrow[r, "0_{\tilde \delta}"]\arrow[d] & E^{\tilde \delta\geq 0}_{ F^0_{\tilde \delta}(X\sslash T)}\arrow[d]\\
I_\mu(X\sslash_G T) \arrow[r, "0_{\tilde \delta}"] & E^{\tilde \delta \geq 0}_{I_\mu(X\sslash_G T)}
\end{tikzcd}
 \quad \quad \quad \quad \quad
\begin{tikzcd}
F_{\tilde \delta}^0(Y\sslash T) \arrow[r] \arrow[d] & F^0_{\tilde \delta}(X\sslash T) \arrow[d, "ev_\star"]\\
I_\mu(Y\sslash_G T) \arrow[r, "\bi"]\arrow[d] &I_\mu(X\sslash_G T)\arrow[d, "s_{\tilde \delta \geq 0}"]\\
I_\mu(X\sslash_G T) \arrow[r, "0_{\tilde \delta}"] & E^{\tilde \delta \geq 0}_{I_\mu(X\sslash_G T)}
\end{tikzcd}
\end{equation}
For the diagram on the left, some of these properties follow from Lemma \ref{lem:fiber-fixed} and the rest are straigtforward. On the right, the top square is fibered by Proposition \ref{prop:diagram2}. The bottom square is in general fibered (with trivial excess bundle) only if we replace $E^{\tilde \delta \geq 0}$ with $E^{\tilde \delta}$ (this follows from Lemma \ref{lem:ql-inertia}). The condition that $\tilde \delta$ is I-nonnegative precisely guarantees that $E^{\tilde \delta} = E^{\tilde \delta \geq 0}$. 
\end{proof}

\subsection{Main lemma}
We have the following analog of Lemma \ref{lem:computation}. (The map $\tilde \bi$ was defined in \eqref{eq:bigmess}.)

\begin{lemma}\label{lem:computation2}
We have the following relationship in $A_*([F^0_{\tilde \delta}(Y\sslash T)/\CC^*_{\lambda}])$.
\begin{equation}\label{eq:thisone}
0_{\tilde \delta}^!\psi_{\tilde \delta}^*\frac{[F_{\tilde \delta}(X\sslash G)]^\vir}{ e_{\CC^*_\lambda}(N^\vir_{F_{\tilde \delta}(X\sslash G)})} = \left(\prod_{j=1}^rC^\circ(\tilde \delta, \epsilon_j)\right)\psi_{\tilde \delta}^*\frac{[F_{\tilde \delta}(Y\sslash G)]^\vir}{ e_{\CC^*_\lambda}(N^\vir_{F_{\tilde \delta}(Y\sslash G)})}
\end{equation}
Here, $z$ is the Euler class of the line bundle on $[F^0_{\tilde \delta}(Y\sslash T)/\CC^*_{\lambda}]$ determined by the identity character of $\CC^*_\lambda$.
\end{lemma}
\begin{proof}
Let $\pi: \wP{a}\times F_{\tilde \delta}(Y\sslash G) \rightarrow F_{\tilde \delta}(Y\sslash G)$ and $q: \wP{a}\times F_{\tilde \delta}(Y\sslash G) \rightarrow [Y/G]$ be the universal curve and universal quasimap on the fixed locus (more precisely, on the disjoint union of fixed loci). 
To compute \eqref{eq:thisone} we obtain, as in the proof of Lemma \ref{lem:computation}, a (equivariant) distinguished triangle in the derived category of $F_{\tilde \delta}(Y\sslash G)$
\begin{equation}\label{eq:euler}
\begin{tikzcd}
L\bj_G^*(\EE_{F_{\tilde \delta}(X\sslash G)})\arrow[r]  & \EE_{F_{\tilde \delta}(Y\sslash G)} \arrow[r]  & (R\pi_*q^*E_{[Y/G]}[-1])^\vee \arrow[r]& {}
\end{tikzcd}
\end{equation}
The fixed part of the analog of \eqref{eq:euler} for $\bj_T: F^0_{\tilde \delta}(Y\sslash T) \to F^0_{\tilde \delta}(X\sslash T)$ fits into a morphism of triangles
\begin{equation}\label{eq:dts2}
\begin{tikzcd}
L\bj_T^*(\EE_{F_{\tilde \delta}(X\sslash T)}^{\fix})\arrow[r] \arrow[d] & \EE_{F_{\tilde \delta}(Y\sslash T)}^{\fix} \arrow[r] \arrow[d] & (R\pi_*(q^*\LL_{[Y/T]/[X/T]}\otimes \omegabul))^{\fix} \arrow[d, "{\phi_{\bj_T}}"] \arrow[r]& {}\\
L \bj_T^*(\LL_{F_{\tilde \delta}(X\sslash T)})\arrow[r] & \LL_{F_{\tilde \delta}(Y\sslash T)} \arrow[r] &\LL_{\bj_T} \arrow[r]&{}
\end{tikzcd}
\end{equation}
where the first two vertical arrows are defined in \eqref{eq:almost-pot} and the last vertical arrow is the fixed part of the canonical morphism $\phi_{\bj_T}:=\phi_{F_{\tilde \delta}(Y\sslash T)/F_{\tilde \delta}(X\sslash T)}$ of \cite[(55)]{cjw}. 
We have used that $Y \hookrightarrow X$ is a regular embedding so we have $\LL_{[Y/G]/[X/G]} \simeq E_{[Y/G]}^\vee$.

By Lemma \ref{lem:Gysin-euler}, the left hand side of \eqref{eq:thisone} is equal to
\[
0^!_{\tilde \delta}\psi_{\tilde \delta}^*[F_{\tilde \delta}(X\sslash G)]^\vir \cap (e_{\CC^*_\lambda}(\bj_T^*\psi_{\tilde \delta}^*N^{\vir}_{F_{\tilde \delta}(X\sslash G}))^{-1}
\]
and we have moreover that $\bj_T^*\psi_{\tilde \delta}^* = \psi_{\tilde \delta}^*\bj_G^*.$
To compute $e_{\CC^*_\lambda}(\psi_{\tilde \delta}^*\bj_G^*N^\vir_{F_{\tilde \delta}(X\sslash G)})$ we take the dual of \eqref{eq:euler} and then apply $\psi_{\tilde \delta}^*$. By the additivity of Euler classes (see the proof of Lemma \ref{lem:computation}) we have
\[
\psi_{\tilde \delta}^*e_{\CC^*_\lambda}(N^{\vir}_{F_{\tilde \delta}(Y\sslash G)}) = e_{\CC^*_\lambda}(\psi_{\tilde \delta}^*\bj_G^* N^{\vir}_{F_{\tilde\delta}(X\sslash G)})e_{\CC^*_\lambda}((\psi_{\tilde \delta}^*R\pi_*q^*E_{[Y/G]}[-1])^{\mov}).
\]
By base change for $R\pi_*$, we have $\psi^*_{\tilde \delta}R\pi_*q^*E_{[Y/G]} = R\pi_*q^*E_{[Y/T]},$ where now $q: \wP{a}\times F^0_{\tilde \delta}(Y\sslash T) \rightarrow [Y/T]$ is the universal map.
A computation using Example \ref{ex:cohomology} shows that $e_{\CC^*_\lambda}((R\pi_*q^*E_{[Y/T]}[-1])^{\mov})$ is equal to the product $\prod_{j=1}^r C^\circ(\tilde \delta, \epsilon_j)$.

Since the class $[F_{\tilde \delta}(X\sslash G)]^\vir$ is nonequivariant (meaning that the class in $A_*([F_{\tilde \delta}(X\sslash G)/\CC^*_\lambda])$ is pulled back from $A_*(F_{\tilde \delta}(X\sslash G))$, or that the corresponding polynomial in $A_*(F_{\tilde \delta}(X\sslash G))[z]$ is constant), it suffices to compute $0_{\tilde \delta}^!\psi_{\tilde \delta}^*[F_{\tilde \delta}(X\sslash G)]^\vir$ non-equivariantly. To do so, we first define a virtual pullback $\bj_T^!: A_*(F_{\tilde\delta}(X\sslash T)) \rightarrow A_*(F_{\tilde\delta}({Y\sslash T}))$ using the rightmost column of \eqref{eq:dts2}. This requires us to check that $((R\pi_*q^*E_{[Y/T]}[-1])^{\fix})^\vee$ is perfect of amplitude [-1,0]. The complex $R\pi_*q^*E_{[Y/G]}[-1]$ is perfect in [1,2], with the part in degree 2 equal to $R^1\pi_*q^*E_{[Y/T]}$. The fixed part $(R^1\pi_*q^*E_{[Y/G]})^{\fix}$ vanishes 
by a now-standard computation.

The virtual pullback $\bj_T^!$ and the refined Gysin pullback $0_{\tilde \delta}^!$ are both maps $ A_*(F^0_{\tilde\delta}({X\sslash T})) \rightarrow A_*(F^0_{\tilde \delta}({Y\sslash T}))$; in Lemma \ref{lem:compare} we prove that they are equal.
Granting this, we have
\[
\bj_T^! \psi_{\tilde \delta}^*[F_{\tilde \delta}(X\sslash G)]^{\vir} = \bj_T^![F^0_{\tilde \delta}(X\sslash T)]^{\vir} = [F^0_{\tilde \delta}(Y\sslash T)]^{\vir} = \psi_{\tilde \delta}^*[F_{\tilde \delta}(Y\sslash G)]^{\vir}
\]
The first and third equalities use Lemma \ref{lem:computation} and the second is \cite[Cor~4.9]{manolache}.

\end{proof}

\begin{lemma}\label{lem:compare}
For $\alpha \in A_*(F^0_{\tilde\delta}(X\sslash T))$, we have $0_{\tilde \delta}^!\alpha =   \bj^!_T \alpha.$
\end{lemma}
\begin{proof}
 To simplify notation we omit the degree $\tilde \delta$ from the quasimap fixed loci and write $g=g_{\tilde \delta}$. Also, because it simplifies notation without changing the proof, we will assume $G=T$ so that $X\sslash_G T = X\sslash T$ and $F^0_{\tilde \delta}(X\sslash T) = F_{\tilde \delta}(X\sslash T)$. The remainder of this lengthy argument is divided into sections by subheadings.\\

\noindent
\textit{Reduction to comparison of two virtual pullbacks.}
To put the virtual pullback $\bj_T^!$ and the refined Gysin map $0_{\tilde \delta}^!$ on equal footing, 
we note that $0^!_{\tilde \delta}$ is equivalent to the virtual pullback defined by the perfect obstruction theory \[
\phi_{0}: ev_\star^* \LL_{Y^{g}\sslash T/E^{\tilde \delta \geq 0}_{Y^g\sslash T}}\rightarrow \LL_{F(Y\sslash T)/F(X\sslash T)}
\]
where the morphism is the canonical map of cotangent complexes. This follows from Corollary \ref{cor:excess} and a fiber diagram similar to the left one in \eqref{eq:two-diagrams} (see also Remark \ref{rmk:gysin-virtual}).
Hence our goal is to construct an isomorphism fitting into a commuting diagram 
\begin{equation}\label{eq:compare-pot}
\begin{tikzcd}
\LL_{F(Y\sslash T)/F(X\sslash T)} & \arrow[l, "\phi_{\bj_T}"'] (R\pi_*q^*\LL_{[Y/T]/[X/T]}\otimes \omegabul)^{\fix} \arrow[d,  "\sim"]  \\
&{ev}_\star^*\LL_{Y^{g}\sslash T/E^{\tilde \delta \geq 0}_{Y^g\sslash T}} \arrow[ul, "\phi_{0}"] 
\end{tikzcd}
\end{equation}
where $\phi_{\bj_T}$ is the fixed part of the canonical map \cite[(55)]{cjw} for $F(Y\sslash T) \rightarrow F(X\sslash T)$.

\begin{remark}Using the fact that $\LL_{Y^g\sslash T/E^{\tilde \delta \geq 0}_{Y^g\sslash T}}$ is represented by $E^{\tilde \delta \geq 0}_{Y^g\sslash T}$ in degree -1, it is straitforward to check that an isomorphism exists (see the last step below). We will carefully check that the ``obvious" isomorphism makes the diagram \eqref{eq:compare-pot} commute.
\end{remark}

\noindent
\textit{Defining \eqref{eq:compare-pot} at the level of universal curves.}
There are canonical morphisms of cotangent complexes given by the solid arrows in the diagram
\begin{equation}\label{eq:compare1}
\begin{tikzcd}
\pi^*\LL_{F(Y\sslash T)/F(X\sslash T)} \arrow[r, "\sim"] & \LL_{\PP^1_{F^Y}/\PP^1_{F^X}} & q^*\LL_{[Y/T]/[X/T]} \arrow[l] \arrow[d, dashrightarrow]\\
&&\pi^*{ev}_\star^*\LL_{Y^{g}\sslash T/E^{\tilde \delta \geq 0}_{Y^g\sslash T}}  \arrow[ull]
\end{tikzcd}
\end{equation}
where we have written $\PP^1_{F^X}$ (resp. $\PP^1_{F^Y}$) for $\wP{a}\times F(X\sslash T)$ (resp. $\wP{a}\times F(Y\sslash T)$).
This is a diagram of complexes on $\wP{a}\times F(Y\sslash T)$. 

We claim that after the complexes are pulled back to $[\torus/\bmu_a] \times F(Y\sslash T)$, there exists a dashed arrow in \eqref{eq:compare1} making the diagram commute. This may be defined as follows. Consider the diagram
\begin{equation}\label{eq:heehee}
\begin{tikzcd}
{[V/\bmu_a]} \times F(Y\sslash T) \arrow[r]  \arrow[rr, bend left=20, "q|_{[V/\bmu_a]}"] &B\bmu_a \times  F(Y\sslash T)\arrow[d, "\pi"] \arrow[r, "q|_{B\bmu_a}"] & {[Y/T]}\\
&F(Y\sslash T) \arrow[ur, "ev_\star"']
\end{tikzcd}
\end{equation}
By the description of $q$ in Proposition \ref{prop:ufam}, the top cell commutes---in fact, it strictly commutes.
The bottom cell, however, does not commute: the map of groups going along the top is given by $\tau_{\intalpha}$ (in particular it is injective) but along the bottom it is trivial. However we know $q$ factors through $Y^g\sslash T$. We define $q_j$ to be the morphism such that the composition
\[[V/\bmu_a] \times F(Y\sslash T) \xrightarrow{q_j} Y^g\sslash T \xrightarrow{j} [Y/T]\]
is equal to $q$, and we define $\varpi$ to be the rigidification $ Y^g\sslash T \rightarrow Y^g\sslash \bar T$ where $\bar T :=(T/\langle \tau_{\intalpha(\bmu_a)}\rangle)$. These definitions give us a strict equality $\varpi \circ q_j|_{B\bmu_a} = \varpi \circ ev_\star \circ \pi$. Combined with the top cell of \eqref{eq:heehee}, this gives a strict equality $\varpi \circ q_j = \varpi \circ ev_\star \circ \pi$ of maps $[V/\bmu_a] \times F(Y\sslash T) \to Y^g\sslash \bar T$.

This leads to the following extension of \eqref{eq:compare1}, where the gray equality is defined and makes the diagram commute after restriction to $[V/\bmu_a]\times F(Y\sslash T)$.
\[
\begin{tikzcd}
\pi^*\LL_{F(Y\sslash T)/F(X\sslash T)} \arrow[r, "\sim"] &[-12pt] \LL_{\PP^1_{F^Y}/\PP^1_{F^X}} &[-15pt] q^*\LL_{[Y/T]/[X/T]} \arrow[l] \arrow[d] \\
&&q_j^*\LL_{Y^{g}\sslash T/E^{\tilde \delta \geq 0}_{Y^g\sslash T}} \arrow[ul ] &[-10pt] \arrow[l, "\sim"'] \arrow[d, equals, gray] q_j^*\varpi^*\LL_{Y^{g}\sslash \bar T/E^{\tilde \delta \geq 0}_{Y^g\sslash \bar T}} \\ 
&&\pi^*{ev}_\star^*\LL_{Y^{g}\sslash T/E^{\tilde \delta \geq 0}_{Y^g\sslash T}}  \arrow[uull] & \arrow[l, "\sim"'] \pi^*ev_\star^*\varpi^*\LL_{Y^{g}\sslash \bar T/E^{\tilde \delta \geq 0}_{Y^g\sslash \bar T}}
\end{tikzcd}
\]
From this we see how to define the dashed arrow in \eqref{eq:compare1}: it is equal to the canonical projection $E^\vee_{[\torus/\bmu_a] \times F(Y\sslash T)} \rightarrow (E^{\tilde \delta \geq 0})^\vee_{[\torus/\bmu_a] \times F(Y\sslash T)}$ followed by the automorphism of $(E^{\tilde \delta \geq 0})^\vee_{[\torus/\bmu_a] \times F^0(Y\sslash T)}$ coming from the 2-isomorphism between $\varpi \circ q_j$ and $\varpi \circ ev_\star \circ \pi$---but we have checked that this 2-isomorphism is trivial on the domain $[V/\bmu_a]\times F(Y\sslash T)$, so the stated automorphism is determined by the identity automorphism of $(E^{\tilde \delta \geq 0})^\vee_{[V/\bmu_a]\times {F(Y\sslash T)}}$. 
\\

\noindent
\textit{Taking cohomology of \eqref{eq:compare1}.}
Applying $R\pi_*(- \otimes \omegabul)^{\fix}$ to the solid part of \eqref{eq:compare1} yields the right part of the following commuting diagram.
\begin{equation}\label{eq:compare2}
\begin{tikzcd}
\LL_{F(Y\sslash T)/F(X\sslash T)} &[-14pt] \arrow[l]  R\pi_*(\pi^*\LL_{F(Y\sslash T)/F(X\sslash T)} \otimes \omegabul)^{\fix} &[-14pt] R\pi_*(q^*\LL_{[Y/T]/[X/T]}\otimes \omegabul)^{\fix} \arrow[l]\arrow[dl, dashrightarrow] \\
{ev}_\star^*\LL_{Y^{g}\sslash_G T/E^{\tilde \delta \geq 0}_{Y^g\sslash T}} \arrow[u, "\phi_{0}"] & \arrow[l, "\sim"']R\pi_*(\pi^*{ev}_\star^*\LL_{Y^{g}\sslash_G T/E^{\tilde \delta \geq 0}_{Y^g\sslash T}}\otimes \omegabul)^{\fix} \arrow[u] 
\end{tikzcd}
\end{equation}
The square on the left comes from applying the inverse of the projection isomorphism and the trace map as in the definition of the adjunction-like morphism in \cite[Sec~A.2.1]{cjw}. The composition of the top two arrows is $\phi_{\bj_T}$ in \eqref{eq:compare-pot} and the left vertical arrow is $\phi_{0}$. A routine computation (see e.g. below) shows that the fixed part of the trace map $R\pi_*( \omegabul)^{\fix} \rightarrow \OO_{F(Y\sslash T)}$ is an isomorphism, so the bottom horizontal arrow in \eqref{eq:compare2} is an isomorphism. Hence, to construct \eqref{eq:compare-pot}, it suffices to construct a dashed injective arrow as in \eqref{eq:compare2} making the triangle commute. At the moment, it is not clear how this desired arrow relates to the one in \eqref{eq:compare1}.

Because $R\pi_*(q^*\LL_{[X/T]/[Y/T]}\otimes \omegabul)^{\fix}$ and $R\pi_*(\pi^*{ev}_\star^*\LL_{Y^{g}\sslash_G T/E^{\tilde \delta \geq 0}_{Y^g\sslash T}}\otimes \omegabul)^{\fix}$ are represented by locally free sheaves in degree -1, it suffices to consider the diagram of cohomology sheaves that arises from applying the cohomology functor $H^{-1}$ to \eqref{eq:compare2}. Noting the equality of functors $H^{-1}(R\pi_*(\bullet \otimes \omegabul))^{\fix} = R^1\pi_*(H^{-1}(\bullet)\otimes \omega)^{\fix}$ on complexes supported in degrees $-1$ and below, we see by applying the latter functor to \eqref{eq:compare1} that it suffices to construct an isomorphism so that this diagram commutes:
\begin{equation}\label{eq:compare3}
\begin{tikzcd}
 R^1\pi_*((H^{-1}(\pi^*\LL_{F(Y\sslash T)/F(X\sslash T)})) \otimes \omega) & {(R^1\pi_*(q^*E_{[Y/T]}^\vee\otimes \omega))}^{\fix} \arrow[l]\arrow[d, "\sim"] \\
&(R^1\pi_*(\pi^*{ev}_\star^*(E^{\tilde \delta \geq 0}_{Y^g\sslash T})^\vee\otimes \omega))^{\fix} \arrow[ul] 
\end{tikzcd}
\end{equation}

Now it is clear how the desired dashed arrow (in \eqref{eq:compare3}) can be obtained from the dashed arrow in \eqref{eq:compare1}: since $\pi: \wP{a} \times F^0(Y\sslash T) \rightarrow F^0(Y\sslash T)$ is a trivial $\wP{a}$-bundle, we can compute $R^1\pi_*$ as in Example \ref{ex:cohomology}. In particular, a local section of $R^1\pi_*(q^*E_{[Y/T]}^\vee\otimes \omega)$ is a local section of  $q^*E^\vee_{[Y/T]}\otimes \omega$ on $[\torus/\bmu_a]\times F(Y\sslash T)$. So the dashed arrow in \eqref{eq:compare1} defines a dashed arrow in \eqref{eq:compare3} making the diagram commute. \\

\noindent
\textit{Explicit description of the dashed arrow in \eqref{eq:compare3}.}
It remains to check that the dashed arrow is an isomorphism. Recall that it is induced by the projection $E^\vee \rightarrow (E^{\tilde \delta \geq 0})^\vee$ (and the ``identity'' on sections written in the coordinate $v$ on $[V/\bmu_a]$).
As in the proof of Lemma \ref{lem:computation} we have an isomorphism of $\CC^*_{\lambda^{1/a}}$-representations
\begin{align}
R^1\pi_*(q^* E_{[Y/T]}^\vee \otimes \omega) &= R^1\pi_*\left( \bigoplus_{j} (\pi^*\L^\vee_{\epsilon_j} \otimes \OO_{\wP{a}\times{F(Y\sslash T)}}(-a\tilde \delta(\epsilon_j)-a-1))\otimes \CC_{\mu^{-a\tilde \delta(\epsilon_j)-a}}\right) \nonumber\\
&= \bigoplus_{j} \left(\L^\vee_{\epsilon_j} \otimes H^1(\wP{a}, \OO(-a\tilde \delta(\epsilon_j)-a-1))\otimes \CC_{\mu^{-a\tilde \delta(\epsilon_j)-a}}\right) \label{eq:compare5}
\end{align}
where $\mu:= \lambda^{1/a}$ is the coordinate on $\CC^*_{\lambda^{1/a}}$, and where we have used the fact that $\omega$ on $\wP{a} \times F(Y\sslash T)$ is given by the 1-dimensional $T \times \CC^*\times \CC^*_{\lambda^{1/a}}$-representation $(t,s, \mu) \mapsto s^{-a-1}\mu^{-a}$, for $t \in T$, $s \in \CC^*$, and $\mu \in \CC^*_{\lambda^{1/a}}$.

Analogously, we have
\begin{equation}\label{eq:compare6}
R^1\pi_*(\pi^*ev^*_\star (E^{\tilde \delta\geq 0}_{Y^{g}\sslash T})^\vee \otimes \omega)  = \bigoplus_{j \;\text{s.t.}\; \tilde \delta(\epsilon_j) \in \ZZ_{\geq 0}} \left( \L_{\epsilon_j}^\vee \otimes H^1(\wP{a}, \OO(-a-1))\otimes \CC_{\mu^{-a}}\right).
\end{equation}
This uses the fact that as a $T \times \CC^* \times \CC^*_{\lambda^{1/a}}$-representation, $\pi^*ev_\star^*(E^{\tilde \delta \geq 0}_{Y^{g}\sslash_G T})^\vee$ is trivial with respect to the $\CC^*\times \CC^*_{\lambda^{1/a}}$-factor, being pulled back from the base.

Recall from Example \ref{ex:cohomology} that $H^1(\wP{a}, \OO(-a\tilde \delta(\epsilon)-a-1))$ has a basis of monomials $u^mv^n$ with $am+n = -a\tilde \delta(\epsilon)-a-1$ and $m,n<0$. The weight with respect to $\CC^*$ of a section $u^mv^n$ of \eqref{eq:compare5} is 
 $(-\tilde \delta(\epsilon)-1)-m$, so fixed sections are of the form $u^{-\tilde \delta(\epsilon)-1}v^{-1}$---in particular, $H^1(\wP{a}, \OO(-a\tilde \delta(\epsilon)-a-1)$ has at most a 1-dimensional subspace fixed by $\CC^*_\lambda$, and this space is positive dimensional if and only if $\tilde \delta(\epsilon)=-m-1$ and hence is a nonnegative integer. Likewise $H^1(\wP{a}, \OO(-a-1))$ is 1-dimensional with basis $u^{-1}v^{-1}$, and as a section of \eqref{eq:compare6} this monomial has weight $-1-(-1)=0$ for $\CC^*_\lambda$. Hence the vector space projection $H^1(\wP{a}, \OO(a\tilde \delta(\epsilon)-a-1)\rightarrow H^1(\wP{a}, \OO(-a-1))$ sending $u^{-\tilde \delta(\epsilon)-1}v^{-1}$ to $u^{-1}v^{-1}$ (note that this is the identity after setting $u=1$) induces an isomorphism
\begin{equation}\label{eq:split}(R^1\pi_*(q^*E_{[Y/T]}^\vee\otimes \omega))^{\fix} \xrightarrow{\sim} (R^1\pi_*(\pi^*ev^*_\star(E_{Y^{g}\sslash_G T}^{\tilde \delta \geq 0})^\vee\otimes \omega))^{\fix}.\end{equation}

\end{proof}

\subsection{Proofs of the Quantum Lefschetz results}

\begin{proof}[Proof of Theorem \ref{thm:ql}]
We first observe that to prove Theorem \ref{thm:ql} it suffices to show that the following equality holds in $A_*(I_\mu(Y\sslash_G T))[z,z^{-1}]$:
\begin{equation}\label{eq:ql1}
\sum_{\beta \mapsto \delta} \varphi^*\iota_*I^{Y\sslash G}_\beta =
\sum_{\tilde \delta \mapsto \delta} 
\left( \prod_{i=1}^m C(\tilde \delta, \rho_i)^{-1}\right)\left(
\prod_{j=1}^r C^\circ(\tilde \delta, \epsilon_j)^{-1}\right)(ev_{\star})_*0^!_{\tilde \delta}j_F^*\Res^{X\sslash T}_{\tilde \delta}(z).
\end{equation}

To prove \eqref{eq:ql1} there are two cases: either there does or there does not exist some $\beta$ mapping to $\delta$ such that $F_\beta(Y\sslash G)$ is nonempty. If $F_\beta(Y\sslash G)$ is emtpy for all $\beta$ mapping to $\delta$, then the left side of \eqref{eq:ql1} is zero by definition. By the proof of Theorem \ref{thm:main} we have that $F^0_{\tilde \beta}(Y\sslash T)$ is empty for all $\tilde \beta$ mapping to $\beta$, and hence $F_{\tilde \delta}(Y\sslash T) = \bigsqcup_{\tilde \beta \to \tilde \delta} F^0_{\tilde \beta}(Y\sslash T)$ is empty. Then the right side of \eqref{eq:ql1} is zero by definition.

If some $F_\beta(Y\sslash G)$ is nonempty then we use \eqref{eq:step7} to compute the left hand side of \eqref{eq:ql1}:
\begin{equation}\label{eq:ql10}
\sum_{\beta \mapsto \delta} \varphi^* \iota_*I^{Y\sslash G}_\beta = \sum_{\tilde \beta \mapsto \delta}
\frac{(ev_\star)_*\psi_{\tilde \beta}^*\Res^{Y\sslash G}_{\tilde \beta}}{\prod_{\tilde \beta(\rho_i) < 0} c_1(\L_{\rho_i})}.
\end{equation}

 Recall that $F_{\tilde \delta}(Y\sslash G)$ is equal to the disjoint union $\bigsqcup_{\tilde \beta \mapsto \tilde \delta} F_{\tilde \beta}(Y\sslash G)$, so $(ev_\star)_*[F_{\tilde \delta}(Y\sslash G)]^{\vir} = \sum_{\tilde \beta \mapsto \tilde \delta} (ev_\star)_*[F_{\tilde \beta}(Y\sslash G)]^{\vir}$. Using the fact that if $\tilde \beta$ maps to $\tilde \delta$ then $\tilde \delta(\epsilon) = \tilde \beta(\epsilon)$, we reorganize the sum on the right side of \eqref{eq:ql1} to get
 \[
 \sum_{\tilde \delta \mapsto \delta}
\frac{(ev_\star)_*\psi_{\tilde \delta}^*\Res^{Y\sslash G}_{\tilde \delta}}{\prod_{\tilde \delta(\rho_i) < 0} c_1(\L_{\rho_i})}.
 \]
 By Lemma \ref{lem:computation2} and the projection formula, this equals
 \[
  \sum_{\tilde \delta \mapsto \delta}
  \prod_{j=1}^r C^\circ(\tilde \delta, \epsilon_j)^{-1}
\frac{(ev_\star)_*0^!_{\tilde \delta}\psi_{\tilde \delta}^*\Res^{X\sslash G}_{\tilde \delta}}{\prod_{\tilde \delta(\rho_i) < 0} c_1(\L_{\rho_i})}.
 \]
 Finally, applying Lemma \ref{lem:computation} as at the end of the proof Theorem \ref{thm:main} (from \eqref{eq:step7} onwards) yields the right hand side of \eqref{eq:ql1}.

\end{proof}

\begin{remark}
If $R$ is not trivial, then Lemma \ref{lem:computation2} can be lifted to the analogous equality in $A_*([F^0_{\tilde \delta}(Y\sslash T)/R\times \CC^*_\lambda])$ by replacing all virtual and characteristic classes with their $R$-equivariant counterparts, and replacing the $\epsilon_j$ with the weights of $E$ as a $T\times R$-representation. For the (non-$\CC^*_\lambda$-equivariant) computation of $\bi^!\psi^*_{\tilde \delta}[F_{\tilde \delta}(X\sslash G)]^\vir_R$ we use the description of $R$-equivariant Chow groups in Lemma \ref{lem:equivariant-chow}. The remainder of the proof of Theorem \ref{thm:ql} is routine.

\end{remark}

\begin{proof}[Proof of Corollary \ref{cor:ql-convex}]
If $\tilde \delta$ is I-nonnegative, then by Lemma \ref{lem:compare-gysin} the refined Gysin maps $0^!_{\tilde \delta}$ and $\bi^!$ agree. By \cite[Thm~2.1.12.ix]{kresch} we can commute the closed embedding $(ev_\star)_*$ and the Gysin map $\bi^!$. This proves Corollary \ref{cor:ql-convex} for unrigidified $I$-functions. To prove it for rigidified $I$-functions, we can argue as in the proof of Theorem \ref{thm:main}, using the additional fact that the Gysin map $0^!_{\tilde \delta}$ commutes with proper pushforward by $\varpi$ (see \cite[Prop~B.18]{BS}).
\end{proof}

\begin{proof}[Proof of Corollary \ref{cor:ql}]
Let $X$ be a vector space with weights $\xi_1, \ldots, \xi_n$. As usual we prove Corollary \ref{cor:ql} for nonrigidified $I$-functions. To obtain the rigidified statement, we apply $(a\circ\varpi)_*$, where $a$ denotes multiplication by the integer $a$ determined by $\beta$. Recall that $ev_\star: F_\beta \rightarrow I_\mu(X\sslash G)$ factors through the component $I_{\mu_a}(X\sslash G)$ of $I_\mu(X\sslash G)$. The degree of $\varpi$ restricted to $\overline{I}_{\mu_a}(\cX)$ is $a^{-1}$, and this cancels with the multiplication by $a$.

For the nonrigidified $I$-function, formula \eqref{eq:ql-cor} follows from Theorem \ref{thm:ql} and the formula for $\Res^{X\sslash T}_{\tilde \delta}$ given in \cite[Prop~5.3]{orb-qmaps}.

If $\tilde \delta$ is I-nonnegative, we derive \eqref{eq:ql-cor-convex} from Corollary \ref{cor:ql-convex}. A formula for the rigidified coefficient $\overline{I}^{X\sslash T}_{\tilde \delta}$ is given in \cite[Thm~5.4]{orb-qmaps}; the following formula can be obtained by minor modifications of that argument:
\begin{equation}\label{eq:toric-ifunc}
{I}^{X\sslash T}_{\tilde \delta}(z) = \prod_{\ell=1}^{n}C(\tilde \delta, \xi_\ell)\one_{g_{\tilde \delta}^{-1}}.
\end{equation}
 In \eqref{eq:toric-ifunc}, the class $\one_{g_{\tilde \delta}^{-1}}$ is the fundamental class of $I_\mu(X\sslash T)_{g_{\tilde \delta}^{-1}}$ (which we define to be 0 if $I_\mu(X\sslash T)_{g_{\tilde \delta}^{-1}}$ is empty). Indeed, this stack has pure dimension: it is isomorphic to $[(X^{g_{\tilde \delta}^{-1}}\cap X^s(T))/T]$, and 
 $X^{g_{\tilde \delta}^{-1}}$ is the subspace of $X$ spanned by those weight spaces with $\xi_i(g_{\tilde \delta}^{-1})=1$.
  A priori \eqref{eq:toric-ifunc} only holds for effective $\tilde \delta$. If $\tilde \delta$ is not effective then the left side of \eqref{eq:toric-ifunc} is zero, and one can check that the right side vanishes as well using the explicit desription of $X^{ss}(T)$ and its cohomology ring (see e.g. \cite[Sec~4]{CIJ}).

Now we compute $\bi^!j^*$ applied to \eqref{eq:toric-ifunc}. Since $j$ is an open embedding, the series $j^*I^{X\sslash T}_{\tilde \delta}$ is given by the same formula, where now $\one_{g_{\tilde \delta}^{-1}}$ is the fundamental class of $I_\mu(X\sslash_G T)_{g_{\tilde \delta}^{-1}}$ ($j^*$ notationally does nothing). To apply $\bi^!$ we momentarily drop the notation $\one_{g_{\tilde \delta}^{-1}}$ and instead write out the fundamental class. For any character $\xi$ of $T$, we have
\begin{equation}\label{eq:toric-ifunc2}\bi^!(c_1(\L_\xi)\cap [I_\mu(X\sslash_G T)_{g_{\tilde \delta}^{-1}}]) = c_1(\bi^*\L_\xi)\cap \bi^![I_\mu(X\sslash_G T)_{g_{\tilde \delta}^{-1}}] = c_1(\L_\xi)\cap [I_\mu(Y\sslash_G T)_{g_{\tilde \delta}^{-1}}].\end{equation}
The first equality is Lemma \ref{lem:Gysin-euler}, and the second is \cite[Example~6.2.1]{Fu98} when $X\sslash_G T$ is a scheme. When $X\sslash_G T$ is a Deligne-Mumford stack, it follows immediately from the construction of $\bi^!$ in \cite[Sec~3.1]{kresch}.
The upshot of \eqref{eq:toric-ifunc2} is that the Gysin pullback $\bi^!$ notationally does nothing, and we obtain \eqref{eq:ql-cor-convex}.

Finally, if $\bj_T$ is regular and the excess bundle \eqref{eq:def-excess} for the fiber square \eqref{eq:fiber-fixed} is trivial, then by Corollary \ref{cor:excess} we have that $0^!_{\tilde \delta} = \bj_T^!$ as maps $A_*(F^0_{\tilde \delta}(X\sslash T)_\QQ \to A_*(F^0_{\tilde \delta}(Y\sslash T)_{\QQ}$. By \cite[Prop~5.3]{orb-qmaps} we know $[F^0_{\tilde \delta}(X\sslash T)]^{\vir} = [F^0_{\tilde \delta}(X\sslash T)]$, and as in the previous paragraph the construction of Gysin pullback tells us that $\bj_T^![F^0_{\tilde \delta}(X\sslash T)] = [F^0_{\tilde \delta}(Y\sslash T)].$

\end{proof}

\appendix
\section{Some results about Chow groups}
In this section, all stacks are finite type and quasi-separated over $\CC$. Recall that our Chow groups are all tensored with $\QQ$ (see Section \ref{sec:conventions}). We use the proper pushforward (for morphisms of DM-type) of \cite{skowera} and \cite[Appendix~B]{BS}.

\subsection{Equivariant Chow groups of Deligne-Mumford stacks}\label{sec:equivariant}
Let $X$ Deligne-Mumford stack with pure dimension and a group action by an algebraic group $G$. Given an integer $k$ (possibly negative), let $V$ be a representation of $G$ such that $G$ acts freely on an open subset $U \subset V$ whose complement has codimension greater than $\dim(X) -k$. Let $v = \dim V$. Define
\[
A^{G}_k(X) := A_{k+v-g}(X\times_G U).
\]
where the right hand side is (as usual) Kresch's integral Chow group, tensored with $\QQ$. The following lemma shows that the group $A^G_k(X)$ is independent of the choices made to define it.
\begin{lemma}\label{lem:equivariant-chow}
There is a canonical isomorphism
\[
A_k^G(X) \rightarrow A_k([X/G]).
\]
commuting with equivariant proper Deligne-Mumford pushforward, flat pullback, and Gysin maps for regular embeddings.
\end{lemma}
\begin{proof}
Let $u$ be the dimension of the complement of $U$ in $V$. The open embedding $X\times_G U \subset X\times_G V$ has a complement of dimension $\dim(X) + u-g$ induces an isomorphism
\begin{equation}\label{eq:triv1}
A^\circ_{k+v-g}(X\times_G V) \xrightarrow{\sim} A^\circ_{k+v-g}(X\times_G U)
\end{equation}
since $\dim X - g + u < k+v-g$ by assumption on $V$ and $U$. Here, $A^\circ$ is the naive Chow group of \cite[Def~2.1.4]{kresch}. Since $X \times_G U$ is a Deligne-Mumford stack, we have $A^\circ_{k+v-g}(X\times_G U) \simeq A_{k+v-g}(X\times_G U)$ by \cite[Thm~2.1.12(ii)]{kresch}. Composing the inverse of \eqref{eq:triv1} with the canonical map $A^\circ_{k+v-g}(X\times_G V) \rightarrow A_k([X/G])$ gives a morphism $A_k^G(X) \rightarrow A_k([X/G])$. The inverse may be constructed as in \cite[Rmk~2.1.17]{kresch}, using \cite[Rmk~2.1.16, Cor~2.4.9]{kresch}.

Compatibiliy with proper pushforward and flat pullback follows from formal arguments. For example, if $f: Y \rightarrow X$ is flat of relative dimension $d$, then there is a diagram
\[
\begin{tikzcd}
A_{k+v-g}(X\times_G U) \arrow[d, "f^*"] &[-10pt] A^\circ_{k+v-g}(X\times_G U) \arrow[d, "f^*"] \arrow[l, "\sim"]  &[-10pt] \arrow[l, "\sim"] A^\circ_{k+v-g}(X\times_G V) \arrow[r] \arrow[d, "f^*"]&[-12pt] A_k([X/G])\arrow[d, "f^*"] \\
A_{k+v-g+d}(Y\times_G U)  & A^\circ_{k+v-g+d}(Y\times_G U)  \arrow[l, "\sim"]  & \arrow[l, "\sim"] A^\circ_{k+v-g+d}(Y\times_G V) \arrow[r] & A_{k+d}([Y/G]) 
\end{tikzcd}
\]
where the top and bottom rows are the isomorphism constructed in this lemma, and the maps $f^*$ are flat pullback. The diagram commutes by definition of $f^*$ for the Chow groups of \cite{kresch}.

Compatibility with Gysin maps is a shade trickier since we do not have a definition of $f^!$ for the naive groups $A^\circ$. Suppose $f: Y \rightarrow X$ is a regular local immersion. Choose an element $\alpha$ of $A_{k+v-g}(X\times_G U)$ represented by a class $[W]$ in $A^\circ_{k+v_g}(X\times_G U)$. Let $W' = W \times_{X\times_G U} (Y \times_G U)$. From the description of Gysin pullback in \cite[514]{kresch}, the class $f^!\alpha$ is represented by $[C_{W'/W}] \in A^\circ_{k+v-g}((Y\times_G U)\oplus N)$, where $N\rightarrow Y\times_G U$ is the normal bundle to $f$ and $C_{W'/W}$ is the normal cone. If $[\tilde W] \in A^\circ_{k+v-g}(X\times_G U)$ is a lift of $[W]$ and $\tilde W' = \tilde W \times_{X\times_G V} (Y\times_G V)$ is its pullback, then $[\tilde W']$ is a lift of $[W']$, and it follows from \cite[Prop~2.18,~2.33]{manolache} that $[C_{\tilde W'/\tilde W}]$ is a lift of $[C_{W'/W}]$. But if $\beta$ denotes the image of $\alpha$ under the isomorphism of this lemma, then $[C_{\tilde W'/\tilde W}] \in A^\circ_{k+v-g+d}((Y\times_G V)\oplus N)$ computes $f^! \beta$.
\end{proof}

\begin{lemma}\label{lem:almost-trivial}
Let $F$ be a Deligne-Mumford stack of pure dimension with an action by $\CC^*$ that is trivial after a reparametrization of $\CC^*$. Then there is a canonical graded isomorphism (of groups tensored with $\QQ$)
\begin{equation}\label{eq:triv}A_*(F)[z, z^{-1}] \simeq A_*([F/\CC^*])
\end{equation} where $z$ has degree $-1$.
The isomorphism \eqref{eq:triv} identifies multiplication by $z$ with the operational class $c_1(\OO(1))$, where $\OO(1)$ is the line bundle on $[F/\CC^*]$ induced by the identity character of $\CC^*$ This isomorphism commutes with (equivariant) proper representable pushforward, flat pullback, and Gysin maps for regular embeddings.
\end{lemma}
\begin{proof}
First assume that the $\CC^*$-action on $F$ is trivial (without any reparametrization). Define \eqref{eq:triv} in this case as follows. For $\alpha \in A_k(F)$ and $i \in \ZZ_{\geq 0}$ let $\ell$ be an integer satisfying $\ell > (\dim(F) -k+i)$. Let $\sigma: X \times \PP^{\ell - 1} \rightarrow X$ be the projection. Define \eqref{eq:triv} by the rule
\[
\alpha z^i \mapsto c_1(\OO(1))^i \cap \sigma^*\alpha \in A_{k+\ell-i-1}(F \times \PP^{\ell - 1}),
\]
noting that $A_{k+\ell-i-1}(F \times \PP^{\ell - 1}) \simeq A_{k-i}([F/\CC^*])$ by Lemma \ref{lem:equivariant-chow}. One may check that the resulting map $A_{*}(F)[z] \simeq A_{*}([F/\CC^*])$ is independent of $\ell$ by checking that it is compatible with the identifications in the proof of \cite[Def-Prop~1]{EG98}. It is an isomorphism on summands of pure degree by \cite[Prop~2.5.6]{kresch}.

Now we prove the general case. Let $\CC^*_\lambda$ denote the original group acting on $F$ (whose action is only trivial after reparametrization).
For a positive integer $a$, let $\CC^*_{\lambda^{1/a}} \cong \CC^*$ act on $F$ via the isogeny $\CC^*_{\lambda^{1/a}} \rightarrow \CC^*_\lambda$ given by $\lambda^{1/a} \mapsto (\lambda^{1/a})^a$.
By assumption, there is a positive integer $a$ such that the group action of $\CC^*_{\lambda^{1/a}}$ on $F$ is trivial.
The identity $X \rightarrow X$ and isogeny $\CC^*_{\lambda^{1/a}} \xrightarrow{(\lambda^{1/a})^a} \CC^*_\lambda$ induce an isomorphism
\begin{equation}\label{eq:equ1}
[X/\CC^*_{\lambda^{1/a}}] \xrightarrow{\sim} [X/\CC^*_{\lambda}].
\end{equation}
In fact, this is already an isomorphism of the prestacks defined in \cite[Prop~2.6]{romagny}.\footnote{The stacks $[X/\CC^*_{\lambda^{1/a}}] \rightarrow [X/\CC^*_{\lambda}]$ are only isomorphic over $B\CC^*$ if the $\CC^*_\lambda$-action is trivial.} 

Now define \eqref{eq:triv} to be the composition
\begin{equation}\label{eq:almost-triv}
A_*(F)[z] \xrightarrow{z\mapsto az} A_*(F)[z] \xrightarrow{\sim} A_*([F/\CC^*_{\lambda^{1/a}}]) =  A_*([F/\CC^*_\lambda])
\end{equation}
where the middle arrow is the isomorphism determined at the start of this proof and the equality is induced by \eqref{eq:equ1}. Explicitly, we may take it to be the inverse of flat pullback along \eqref{eq:equ1}, so it sends $a c_1(\OO_{[F/\CC^*_{\lambda^{1/a}}]}(1))$ to
$c_1(\OO_{[F/\CC^*_{\lambda}]}(1))$. This shows that \eqref{eq:almost-triv} identifies multiplication by $z$ with
$c_1(\OO_{[F/\CC^*_{\lambda}]}(1))$ as desired. It remains to check that the isomorphism \eqref{eq:almost-triv} is independent of $a$. But we have already seen that the image of $z$ is independent of $a$, and the restriction of the first two morphisms to $A_*(F)$ is also independent of $a$.

Finally we check that \eqref{eq:triv} is compatible with various kinds of morphisms. When the $\CC^*$-action on $F$ is trivial, these compatibilities follow from the compatibility of the isomorphism in Lemma \ref{lem:equivariant-chow} and the projection formula (for pushforward), \cite[Prop~2.4.6(iii)]{kresch} (for flat pullback), and Lemma \ref{lem:Gysin-euler} (for Gysin pullback). In the general case, let $F$ and $G$ be two $\CC^*$-stacks with action that becomes trivial after reparametrization. Since the composition \eqref{eq:almost-triv} is independent of $a$, we may choose $a$ sufficiently large so that the $\CC^*_{\lambda^{1/a}}$-action on both $F$ and $G$ is trivial. Now the desired compatibilities follow easily from the compatibilities in the case of trivial actions.
\end{proof}

\subsection{(Virtual) Gysin pullbacks}

Let $\bi: X \rightarrow Y$ be a regular local immersion of algebraic stacks (see \cite[Sec~3.1]{kresch}). Suppose we have a fiber square
\begin{equation}\label{eq:square}
\begin{tikzcd}
\arrow[d]Y' \arrow[r, "\bi'"] & X' \arrow[d]\\
Y \arrow[r, "\bi"] & X
\end{tikzcd}
\end{equation}
The refined Gysin pullback $\bi^!: A_*(X') \rightarrow A_*(Y')$ was defined in \cite[Sec~3.1]{kresch}.
The following lemma is the analog of \cite[Prop~6.3]{Fu98}.

\begin{lemma}\label{lem:Gysin-euler}
Let $E \rightarrow X'$ be a vector bundle. Then for $\alpha \in A_*(X')$, we have 
\[
\bi^!(e(E) \cap \alpha) = e(\bi'^*E)\cap \bi^!\alpha
\]
If \eqref{eq:square} is $G$-equivariant, the same statement holds for $G$-equivariant Euler classes.
\end{lemma}
\begin{proof}
If \eqref{eq:square} is $G$-equivariant, then the desired result is a statement about Chow groups for the quotient of \eqref{eq:square} by $G$, so it suffices to consider the non-equivariant case.

We can assume $\alpha$ is a cycle, represented by $[V] \in A^\circ_*(U)$ with $U \rightarrow X_0$ a vector bundle and $X_0 \rightarrow X'$ a projective morphism. Let $Y_0 = X_0 \times_{X'} Y'$ with $\bi_0: Y_0 \rightarrow X_0$ the canonical map and $E_{Y_0}, E_{X_0}$ the pullbacks of $E$ to $Y_0$ and $X_0$, respectively. We have a fiber diagram
\begin{equation}\label{eq:Gysin1}
\begin{tikzcd}
\bi_0^*U \oplus E_{Y_0} \arrow[r] & U \oplus E_{X_0}\\
\bi_0^*U \arrow[u, hook, "s"]\arrow[d] \arrow[r] & U \arrow[u, hook, "s"]\arrow[d]\\
Y \arrow[r, "\bi"] & X
\end{tikzcd}
\end{equation}
where the maps labeled $s$ are inclusions via the zero section. Let $W = V \times_X Y$, a priori a closed substack of $\bi_0^*U$. Since $s$ is a closed embedding, the substacks $V$ and $W$ are isomorphic to their images under $s$. From the definition of the Euler class \cite[Def~2.4.2]{kresch} and the concrete description of the refined Gysin map \cite[514]{kresch}, we have that both
$\bi^!(e(E) \cap \alpha)$ and $e(\bi'^*E)\cap \bi^!\alpha$ are represented by the class of the normal cone of $W$ over $V$, pushed forward to $A^\circ_*(N_{Y/X} \oplus \bi_0^*U \oplus E_{Y_0})$.
\end{proof}

Suppose $X \rightarrow Y$ is a closed embedding of Deligne-Mumford stacks. In this case $H^0(\LL_{X/Y})=0$, so any perfect obstruction theory for $X \rightarrow Y$ is perfect in degree -1; i.e., it is a locally free sheaf in degree -1. Hence the associated vector bundle stack is actually a vector bundle: it is representable over $Y$. A special case of the functoriality of virtual pullback \cite[Thm~4.8]{manolache} is the following excess intersection formula.

\begin{lemma}\label{lem:excess}
Let $\bi: Y \rightarrow X$ be a closed embedding of algebraic stacks that are stratified by global quotient stacks, and let $\phi_1: \sE_1[1] \rightarrow \LL_{Y/X}$ and $\phi_2: \sE_2[1] \rightarrow \LL_{Y/X}$ be relative perfect obstruction theories, where $\sE_i$ is a locally free sheaf on $Y$. For $j=1,2$, let $\bi^!_{\sE_j}$ be the associated virtual pullbacks of \cite{manolache}. If there is a surjection $\sE_2 \rightarrow \sE_1$ that commutes with the maps to $\LL_{Y/X}$, such that the kernel of the surjection is $\sF$ (necessarily locally free), then 
\[
\bi^!_{\sE_2}(\alpha) = e(\sF^\vee) \cap \bi^!_{\sE_1}(\alpha)
\]
for $\alpha \in A_*(X)$.
\end{lemma}

\begin{proof}
First apply \cite[Thm~4.8]{manolache} with $F=Y$, $G=\mathfrak{M}=X$, $g=id$ and $f: Y \hookrightarrow X$ the inclusion. Because there is a short exact sequence $0 \rightarrow \sF \rightarrow \sE_2 \rightarrow \sE_1 \to 0$ with compatible maps to $\LL_{Y/X}$, we have a compatible triple and hence
\[
\bi^!_{\sE_2} = \bi^!_{\sE_1} \circ id^!_{\sF} = id^!_{\sF} \circ \bi^!_{\sE_1},
\]
where the second equality is \cite[Thm~4.3]{manolache} and $id^!_{\sF}$ is (refined) virtual pullback for the identity map on $X$ and the perfect obstruction theory $\sF[1] \rightarrow \LL_{X/X}$. Since the cone stack associated to $\sF[1]$ is $h^1/h^0(\sF^\vee[-1])$, i.e. the total space of the vector bundle $\sF^\vee$, it follows from the definitions and \cite[Cor~2.4.5]{kresch} that $id^!_{\sF}(\alpha) = e(\sF^\vee) \cap \alpha.$
\end{proof}

We will use the following corollary. Suppose we have a fiber diagram of algebraic stacks stratified by global quotients:
\[
\begin{tikzcd}
Y'' \arrow[d] \arrow[r] & X'' \arrow[d] \\
Y' \arrow[r, "\bi'"] \arrow[d, "f"'] & X' \arrow[d] \\
Y \arrow[r, "\bi"] & X
\end{tikzcd}
\]

\begin{remark}\label{rmk:gysin-virtual}
If $\bi$ is a regular local immersion, the refined Gysin map $\bi^!: A_*(X') \to A_*(Y')$ is identical to the virtual pullback defined by the relative perfect obstruction theory $f^*\LL_\bi \to \LL_{\bi'}$.
\end{remark}

If both $\bi$ and $\bi'$ are regular local immersions, then both $\LL_\bi$ and $\LL_{\bi'}$ are represented by locally free sheaves in degree -1: these are the conormal sheaves $N^\vee_\bi$ and $N^\vee_{\bi'}$, respectively. The \textit{excess bundle} for the fibered square with $\bi$ and $\bi'$ is defined to be
\begin{equation}\label{eq:def-excess}
\ker(f^*N^\vee_{\bi} \to N^\vee_{\bi'}).
\end{equation}

\begin{corollary}\label{cor:excess}
If the excess bundle \eqref{eq:def-excess} is zero, then $\bi^!(\alpha) = \bi'^!(\alpha)$ for $\alpha \in A_*(X'')$.
\end{corollary}
\begin{proof}
Follows from Lemma \ref{lem:excess} and Remark \ref{rmk:gysin-virtual}.
\end{proof}

\printbibliography
\end{document}